\documentclass{amsart}
\usepackage{amssymb}
\usepackage{amsmath}
\usepackage{amsthm}
\usepackage{mathrsfs}
\usepackage{hyperref}

\usepackage[all]{xy}
\usepackage[usenames]{color}

\usepackage{graphicx}
\usepackage{tikz}
\usepackage{pgfplots}
\usepackage{amsmath}

\newtheorem{thm}{Theorem}[section]
\newtheorem{lemma}[thm]{Lemma}
\newtheorem{cor}[thm]{Corollary}
\newtheorem{prop}[thm]{Proposition}

\theoremstyle{definition}
\newtheorem{defn}[thm]{Definition}
\newtheorem*{ack}{Acknowledgements}
\newtheorem{remark}[thm]{Remark}
\newtheorem{example}[thm]{Example}

\newtheorem*{theorem*}{Theorem}

\newtheorem*{namedtheorem}{\theoremname}
\newcommand{\theoremname}{testing}

\newcommand{\N}{\mathbb{N}}
\newcommand{\Z}{\mathbb{Z}}

\newcommand{\R}{\mathbb{R}}
\newcommand{\C}{\mathbb{C}}

\newcommand{\s}{\mathbb{S}}

\newcommand{\HF}{\mathrm{HF}}

\newcommand{\norm}[1]{\left\lVert#1\right\rVert}

\DeclareMathOperator{\diam}{diam}
\DeclareMathOperator{\nmod}{mod_n}
\DeclareMathOperator{\dist}{dist}

\DeclareMathOperator{\Hdist}{\dist_{\mathcal{H}}}

\begin{document}

\title[Branched quasisymmetries]{Vertical quasi-isometries and branched quasisymmetries}
%\title[QR filling extension notes]{QR filling extension notes}
\author{Jeff Lindquist and Pekka Pankka}
\address{Department of Mathematical Sciences, University of Cincinnati, P.O.Box 210025}
\email{lindqujy@ucmail.uc.edu}
\address{P.O. Box 64 (Pietari Kalmin katu 5), FI-00014 University of Helsinki, Finland}
\email{pekka.pankka@helsinki.fi}

\begin{abstract}
We introduce a class of mappings called vertical quasi-isometries and show that  branched quasisymmetries $X\to Y$ of Guo and Williams between compact, bounded turning metric doubling spaces admit natural vertically quasi-isometric extensions $\widehat X\to \widehat Y$ between hyperbolic fillings $\widehat X$ and $\widehat Y$ of $X$ and $Y$, respectively. We also give a converse for this result by showing that a finite multiplicity vertical quasi-isometry $\widehat X \to \widehat Y$ between hyperbolic fillings induces a branched quasisymmetry $X \to Y$. 
\end{abstract}

\thanks{J.L. was partially supported by the Academy of Finland projects \#297258 and \#308759. P.P. was partially supported by the Academy of Finland project \#297258 and by the Simons Semester \emph{Geometry and analysis in function and mapping theory on Euclidean and metric measure spaces} at IMPAN, Warsaw.}

\maketitle

\tableofcontents

%%% SEC 1
\section{Introduction}
\label{sec:intro}

In this article, we consider a non-injective counterpart for the classical correspondence of quasi-isometries between hyperbolic spaces and quasiconformal homeomorphisms between their boundaries. In the homeomorphic case, this correspondence has long history which goes back to Mostow's rigidity theorem \cite{Mostow}. 

Recall that a mapping $f\colon X\to Y$ between metric spaces $(X,d_X)$ and $(Y,d_Y)$ is a \emph{quasi-isometry} if there exist constants $\alpha \ge 1$ and $\beta >0$ satisfying
\begin{equation}
\label{QI}\tag{QI}
\frac{1}{\alpha} d_X(x,x') - \beta \le d_Y(f(x),f(x')) \le \alpha d_X(x,x') + \beta
\end{equation}
for all $x,x'\in X$ and such that for every point $y \in Y$, there exists a point $x \in X$ with $d_Y(f(x), y) \leq \beta$ (i.e. $f$ is cobounded). A mapping $f\colon X\to Y$ is \emph{cobounded} if there exists $C>0$ for which $B_Y(fX,C) = Y$.

A homeomorphism $f \colon X\to Y$ between metric spaces $(X,d_X)$ and $(Y,d_Y)$ is \emph{quasiconformal} if there exists a distortion constant $H\ge 1$ for which, for all $x\in X$, 
\begin{equation}
\label{QC}\tag{QC}
\lim_{r \to 0} \frac{\sup_{y\in B_X(x,r)} d_Y(f(x),f(y))}{\inf_{y\in X\setminus B_X(x,r)} d_Y(f(x),f(y))} \le H;
\end{equation}
this is the so-called \emph{metric definition of quasiconformality}. 

We refer to a recent survey of Bourdon \cite{Bourdon} on the vast literature related to Mostow's theorem and merely comment here that the correspondence we allude to is as follows. On one direction, the boundary map $\partial F \colon \s^{n-1}\to \s^{n-1}$ induced by a quasi-isometry $F \colon \mathbb H^n \to \mathbb H^n$ is quasiconformal, quantitatively. To the other direction, each quasiconformal map $f\colon \s^{n-1}\to \s^{n-1}$ is an extension of a quasi-isometry $F\colon \mathbb H^n \to \mathbb H^n$ of the hyperbolic $n$-space. 

As the generality of these definitions hints, the correspondence between these classes of maps is understood in a more general setting of metric spaces. Roughly speaking, under mild conditions on spaces, quasi-isometries of Gromov hyperbolic spaces correspond to quasiconformal maps between their Gromov boundaries. We do not attempt to discuss the various assumptions here and merely refer the interest reader here to e.g.\;~Gromov's seminal paper \cite{Gromov} on Gromov hyperbolic spaces, Bonk--Schramm \cite{Bonk-Schramm}, and Paulin \cite{Paulin}, or again to a survey of Bourdon \cite{Bourdon}, or to the monographs Ghys--de la Harpe \cite{GH} and Drutu--Kapovich \cite{Drutu-Kapovich}.

Instead of considering quasiconformal maps $\partial_\infty X\to \partial_\infty Y$ between the Gromov boundaries, we take as our starting point the class of maps called branched quasisymmetries introduced by Guo and Williams \cite{GW}. As we discuss shortly, this class of mappings is a generalization of quasiregular mappings between Riemannian manifolds, which in turn are a non-injective counterpart of quasiconformal maps.

Since branched quasisymmetries are not injective, even locally, the maps which they induce between hyperbolic fillings are not quasi-isometric. For this reason we introduce a class of maps called vertical quasi-isometries and show that the correspondence of vertical quasi-isometries and branched quasisymmetries is analogous to the one of quasi-isometries and quasiconformal maps.

Before discussing branched quasisymmetries and vertical quasi-isometries in more detail, we briefly recall the relationships between quasiconformal, quasisymmetric, and quasiregular maps.

A homeomorphism $f\colon X\to Y$ is \emph{quasisymmetric} if there exists a homeomorphism $\eta \colon [0,\infty) \to [0,\infty)$ satisfying
\begin{equation}
\label{QS}\tag{QS}
d_Y(f(x),f(x'')) \le \eta\left( \frac{d_X(x,x'')}{d_X(x,x')} \right) d_Y(f(x),f(x'))
\end{equation}
for all triples $x,x',x''\in X$ of distinct points. Quasisymmetric homeomorphisms are always quasiconformal and for a large class of spaces, called Loewner spaces, these two classes of homeomorphisms agree. We refer to Heinonen and Koskela \cite{HK} and a monograph of Heinonen \cite{He} for a detailed discussion. 

Quasiconformal homeomorphisms between smooth manifolds admit also an analytic characterization, which we may take as a definition also for maps which are not homeomorphisms. A continuous map $f\colon M \to N$ between Riemannian $n$-manifolds is \emph{quasiregular} if $f$ is in the Sobolev class $W^{1,n}_{\mathrm{loc}}(M,N)$ and there exists a constant $K\ge 1$ for which the distortion inequality
\begin{equation}
\label{QR}\tag{QR}
\norm{Df(x)}^n \le K J_f(x) 
\end{equation}
holds for almost every $x\in M$. Here $Df(x) \colon T_xM \to T_{f(x)}N$ is the distributional derivative of $f$, $\norm{Df(x)}$ is the operator norm of $Df$, and $J_f(x) = \det Df(x)$ the Jacobian determinant. Note that there is no injectivity condition on the mapping $f$ and typically $f$ is not even locally injective. However, if $f$ is assumed to be a homeomorphism then $f$ is quasiconformal in the sense of the metric definition. We refer to Heinonen-Koskela \cite{Heinonen-Koskela_Invent} for a discussion on the relationship definitions of quasiconformal mappings.

Quasiregular mappings give a higher dimensional analogue for the holomorphic maps between Riemann surfaces. In particular, by a theorem of Reshetnyak, \emph{a non-constant quasiregular mapping between Riemannian $n$-manifolds is a discrete and open map}. Thus, by the {\v C}ernavski\u{\i}--V\"ais\"al\"a theorem \cite{Chernavskii, Vaisala} a non-constant quasiregular mapping between Riemannian $n$-manifolds is a local homeomorphism in a complement of a set of (topological) codimension $2$. We refer to monographs of Reshetnyak \cite{Reshetnyak} and Rickman \cite{Ri} for the theory of quasiregular mappings.

\subsection{Branched quasisymmetries and quasiregular maps}

We are now ready to the definition of branched quasisymmetries of Guo and Williams. A continuous mapping $f\colon X\to Y$ between metric spaces $X$ and $Y$ is a \emph{branched quasisymmetry} if there exists a homeomorphism $\eta \colon [0,\infty) \to [0,\infty)$ for which the distortion inequality
\begin{equation}
\label{BQS}\tag{BQS}
\diam fE' \le \eta\left( \frac{\diam E'}{\diam E}\right) \diam fE
\end{equation}
holds for all intersecting continua $E$ and $E'$ in $X$. Recall that a subset $E \subset X$ is a \emph{continuum} if $E$ is compact, connected, and consists of at least two points. In particular, a continuum has positive diameter. Note that, we divert here slight from the terminology of \cite{GW} and do not assume the mapping $f$ to be discrete and open or to have bounded local multiplicity as in \cite[Definition 6.45]{GW}.

Before continuing the discussion, we make some immediate observations on elementary properties of branched quasisymmetries. For a branched quasisymmetry $f\colon X\to Y$ and intersecting continua $E$ and $E'$, the image $fE$ has positive diameter if and only if $fE'$ has positive diameter. Thus, if $X$ is a continuum and $f$ is non-constant, then $\diam fE>0$ for each continuum $E \subset X$. In particular, $f$ is a light map. Recall that a map $f \colon X\to Y$ is \emph{light} if, for each $y\in Y$, the pre-image $f^{-1}(y)$ is totally disconnected. 

As an almost immediate consequence we also have that, if $X$ is a bounded turning space, a quasisymmetry $f\colon X\to Y$ is a branched quasisymmetry. Recall that a metric space $(X,d)$ has \emph{bounded turning} if there exists a constant $\lambda\ge 1$ having the property that, for all $x,y\in X$, there exists a continuum $E_{xy}\subset X$ of diameter at most $\lambda d(x,y)$ containing points $x$ and $y$. Clearly, not all branched quasisymmetries are quasisymmetric, since a branched quasisymmetry need not be injective; consider for example the winding map $\C \to \C$, $z\mapsto z^2/|z|$.

Guo and Williams show in \cite{GW} that weakly metrically quasiregular mappings are discrete and open branched quasisymmetries; see \cite[Section 6]{GW} for a detailed discussion on the results and terminology. For mappings between Riemannian manifolds, the result of Guo and Williams reads as follows; we give a direct proof in Section \ref{QR and BQS sec}.

\begin{thm}[Guo--Williams]
\label{thm:QR=BQS}
A non-constant continuous map $f\colon M\to N$ between closed and oriented Riemannian $n$-manifolds is quasiregular if and only if $f$ is a discrete, open, and sense-preserving branched quasisymmetry.
\end{thm}

It is well-known from the Euclidean quasiconformal theory that this statement is not quantitative in terms of the distortion conditions \eqref{QR} and \eqref{BQS}. Indeed, it suffices to notice that M\"obius transformations of $\s^n$ are $1$-quasiconformal, and hence quasisymmetric, but not $\eta$-quasisymmetric with the same distortion function $\eta \colon [0,\infty) \to [0,\infty)$. In Section \ref{QR and BQS sec} we give a quantitative version of Theorem \ref{thm:QR=BQS} using an additional normalization. This normalization is similar to e.g.~the barycentric normalization in \cite{Pankka-Souto}.
 
The characterization of quasiregular maps in Theorem \ref{thm:QR=BQS} is, in the end, not very surprising. In one direction, it suffices to show that a quasiregular mapping satisfies the distortion inequality \eqref{BQS}. This is a special case of a result of Guo and Williams \cite[Theorem 6.50]{GW} in the locally Euclidean setting. We give a proof using standard modulus estimates and hyperbolic fillings. To the other direction, we may use an adaptation of a known metric characterization of quasiregular mappings in Euclidean spaces.

\subsection{Extension of branched quasisymmetries into hyperbolic fillings}

In this and the following section, we discuss our main results.

Our first main theorem is an extension of branched quasisymmetries $X\to Y$ to virtual quasi-isometries $\widehat X\to \widehat Y$ of hyperbolic fillings $\widehat X$ and $\widehat Y$ of $X$ and $Y$, respectively.

As in Bonk and Saksman \cite{Bonk-Saksman}, Bonk, Saksman, and Soto \cite{Bonk-Saksman-Soto}, and \cite{L}, we consider the hyperbolic filling of Bourdon and Pajot \cite{Bourdon-Pajot}. We postpone the precise definition to Section \ref{Hyp Fill sec} and merely note here that a hyperbolic filling $(\widehat{X},\ast)$ of a compact doubling metric space $X$ is a pointed Gromov-hyperbolic geodesic metric space with the property that the Gromov-boundary $\partial \widehat{X}$ is quasisymmetric to $X$; see Bourdon--Pajot \cite[Proposition 2.1]{Bourdon-Pajot}. Although a hyperbolic filling $\widehat{X}$ of a space $X$ is not unique as a metric space, all hyperbolic fillings of $X$ given by the construction are quasi-isometric to each other. Also, the Gromov boundary $\partial_\infty \widehat X$ of a hyperbolic filling $\widehat X$ of $X$ is quasisymmetrically equivalent to $X$. Therefore we may identify $\partial_\infty \widehat{X}$ and $X$ in the following statements. 

The mapping $\widehat X \to \widehat Y$ we obtain between the hyperbolic fillings is a vertical quasi-isometry. For the terminology, let $(X,d_X;\ast)$ be a pointed metric space. We say that a discrete path $\gamma \colon \N_0\to X$ is a \emph{$(\alpha,\beta)$-quasigeodesic} if there exists $\alpha \ge 1$ and $\beta >0$ for which 
\[
\frac{1}{\alpha}|n-m| - \beta \le d_X(\gamma(n),\gamma(m)) \le \alpha |n-m| + \beta
\]
for all $n,m\in \N_0$. We say that a $(1,0)$-quasigeodesic is a \emph{(discrete) geodesic} and that a quasi-geodesic $\gamma \colon \N_0 \to X$ is \emph{vertical} if $\gamma$ is a pointed map $(\N_0,0) \to (X,\ast)$, that is, $\gamma(0)=\ast$. 

Heuristically speaking, we define vertical quasi-isometries to be pointed maps $(X,\ast) \to (Y,\ast)$ which map vertical geodesics of $(X,\ast)$ to vertical quasigeodesics of $(Y,\ast)$. The definition allows different vertical geodesics in $(X,\ast)$ to map into the same quasi-isometry class of vertical quasigeodesics in $(Y,\ast)$. More precisely, we give the following definition.
 
\begin{defn}
A mapping $F\colon (X,\ast) \to (Y,\ast)$ between pointed metric spaces is a \emph{vertical $(\alpha,\beta)$-quasi-isometry for $\alpha\ge 1$ and $\beta>0$} if, for each vertical discrete geodesic $\gamma \colon (\N_0,0)\to (X,\ast)$, the discrete path $F\circ \gamma \colon (\N_0,0) \to (Y,\ast)$ is a vertical $(\alpha,\beta)$-quasigeodesic.
\end{defn}

Clearly a pointed quasi-isometry $(X,\ast) \to (Y,\ast)$ is a vertical quasi-isometry. It is also easy to observe that, due to geodesic stability, compositions of vertical quasi-isometries between Gromov-hyperbolic spaces are vertical quasi-isometries; see Lemma \ref{VQI comp} in Section \ref{sec:VQI}.

Our first main theorem states that a branched quasisymmetry $X\to Y$ from a compact metric space, which is also bounded turning, into a compact metric space induces a vertical quasi-isometry $(\widehat{X}, \ast) \to (\widehat{Y}, \ast)$ between hyperbolic fillings.

\begin{thm}
\label{thm:BQS2VQI}
Let $X$ and $Y$ be compact metric spaces and suppose that $X$ has bounded turning, and let $\widehat{X}$ and $\widehat{Y}$ be hyperbolic fillings of $X$ and $Y$, respectively. Then, for each branched quasisymmetry $f\colon X\to Y$, there exists a vertical quasi-isometry $\varphi_f \colon \widehat{X} \to \widehat{Y}$ for which the induced map $\partial \varphi_f \colon \partial_\infty \widehat X\to \partial_\infty \widehat Y$ coincides with $f$.
\end{thm}

Theorem \ref{thm:BQS2VQI} is quantitative in the sense that quasi-isometry constants and distortion functions depend only on each other and the data associated to the spaces. We give a quantitative version of Theorem \ref{thm:BQS2VQI} in Section \ref{BQS to VQI sec}. 

\begin{remark}
Our interest to Theorem \ref{thm:BQS2VQI} stems from an extension problem for quasiregular mappings in V\"ais\"al\"a's ICM article \cite{Vaisala}: \emph{Given a quasiregular mapping $f\colon \R^{n-1}\to \R^{n-1}$, does there exist a quasiregular mapping $F\colon \R^n \to \R^n$ extending $f$?} To our knowledge, this problem is solved only in some special cases. We refer to Rickman \cite{Rickman_Acta}, \cite{Drasin-Pankka}, and \cite{Pankka-Wu} for more discussion. An analog of V\"ais\"al\"a's question for quasiregular extensions $B^n \to B^n$ of quasiregular maps $\s^{n-1}\to \s^{n-1}$ is similarly an open problem. For quasiconformal mappings, these extension problem are solved to the positive by Beurling--Ahlfors \cite{Beurling-Ahlfors} ($n=2$), Carleson \cite{Carleson} ($n=3$), and Tukia--V\"ais\"al\"a \cite{Tukia-Vaisala_Annals} in all dimensions.
\end{remark}

Combining Theorems \ref{thm:QR=BQS} and \ref{thm:BQS2VQI}, we obtain vertically quasi-isometric extensions of quasiregular maps between closed Riemannian manifolds. This extension bears similarity to harmonic extensions $B^n \to B^n$ of quasiregular mappings $\s^{n-1} \to \s^{n-1}$ in \cite{Pankka-Souto}.

\subsection{Extension of vertical quasi-isometries to the boundary}

We turn now to the extension of vertical quasi-isometries $\widehat X\to \widehat Y$ to the branched quasisymmetries $X\to Y$ on boundary. Although, the result admits a more general discussion, we restrict ourselves to the hyperbolic fillings in this article.

Our second main theorem states that a vertical quasi-isometry $\widehat X\to \widehat Y$ of finite multiplicity between hyperbolic fillings of $X$ and $Y$, respectively, induces a branched quasisymmetry $X \to Y$, if the spaces $X$ and $Y$ are compact and $Y$ is doubling.

\begin{thm}
\label{thm:intro-VQI2pBQS}
Let $X$ and $Y$ be compact metric spaces, where $Y$ is doubling, and let $\widehat X$ and $\widehat Y$ be their hyperbolic fillings, respectively. Let $\varphi \colon (\widehat X,\ast) \to (\widehat Y, \ast)$ be a pointed vertical quasi-isometry of finite multiplicity. Then the map $\partial \varphi \colon \partial \widehat X\to \partial \widehat Y$ induced by $\varphi$ is a power branched quasisymmetry of finite multiplicity. 
\end{thm}

Here we say that $\varphi$ is a power branched quasisymmetry if its distortion function $\eta$ can be taken to have a particular form; see Theorem \ref{thm:power}.  Note that, again, under an identification $\partial \widehat X = X$ and $\partial \widehat Y = Y$, the map $\partial \varphi \colon \partial \widehat X\to \partial \widehat Y$ is identified with a branched quasisymmetry $X\to Y$. In Section \ref{VQI to BQS sec} we call the map $\partial \varphi \colon \partial \widehat X\to \partial \widehat Y$, and associated maps $X\to Y$, the \emph{trace of $\varphi \colon \widehat X \to \widehat Y$}; for definitions see Section \ref{VQI to BQS sec}. 

Theorem \ref{thm:intro-VQI2pBQS} is an extension of the corresponding result for quasi-isometries in the sense that in the both cases the quasisymmetry is controlled by a power-type gauge function; see e.g.~Bonk and Schramm \cite[Theorem 6.5]{Bonk-Schramm} for the quasi-isometric result. Note also that, in Theorem \ref{thm:intro-VQI2pBQS}, the finite multiplicity of the vertical isometry $\widehat X\to \widehat Y$ is a sufficient condition for the discreteness of the trace map $X\to Y$. To this end, in Section \ref{sec:traces_openness} we characterize the openness of the trace map in terms of a lifting property for geodesics under the vertical quasi-isometry (Theorem \ref{VQI LP iff open}), and in Section \ref{sec:traces_surjectivity} the surjectivity of the trace map in terms of coboundedness of the vertical quasi-isometry (Theorem \ref{thm:bdry map surj}).

The statement of Theorem \ref{thm:intro-VQI2pBQS} is quantitative in the sense that the distortion function of $f$ depends only on quasi-isometry constants of $\varphi$ and the data of the spaces and the same holds for the multiplicity; see Section \ref{VQI to BQS sec}

\subsection{Structure of the article}

The article is divided into three parts. 
In Part \ref{part:bt-bqs}, consisting of Sections \ref{sec:bounded_turning}--\ref{sec:Koebe}, we discuss elementary theory of branched quasisymmetries on bounded turning spaces. These results are familiar from the theory of quasisymmetric maps. For example, the definition of branched quasisymmetry yields almost immediately a version of the Koebe distortion theorem (Theorem \ref{thm:Koebe}). 

As the main result of these sections, we prove that, under our assumptions on spaces, branched quasisymmetries have a power distortion. We formulate this as follows. For the corresponding results on quasisymmetries between uniformly perfect spaces, see e.g.\;Heinonen \cite[Theorem 11.3]{He}.

\begin{thm}
\label{thm:power}
Let $X$ and $Y$ be compact metric spaces and suppose that $X$ has bounded turning, and let $f\colon X\to Y$ be a branched quasisymmetry. Then there exist constants $C>0$ and $\alpha\in (0,1)$ for which $f$ is $\eta$-BQS for the distortion function $\eta\colon [0,\infty) \to [0,\infty)$ given by 
\[
t\mapsto C \max\{ t^\alpha, t^{1/\alpha}\}.
\] 
\end{thm}

Theorem \ref{thm:power} actually holds in more generality.  If for every two points in $X$ one can find a continuum containing those two points, then one can define the Mazurkiewicz metric $d_M(x,y) := \inf \diam(E)$, where the infimum is taken over all continua such that $x,y \in E$.  Suppose that the identity map $(X,d) \to (X, d_M)$ is a homeomorphism. This is the case, for instance, when $X$ is a domain that is connected and locally path connected; see \cite[Remark 3.2]{KLS}.  Then continua have the same diameters in the metrics $d$ and $d_M$; see, for example, \cite[Proposition 4.12]{KLS}. In this case, $f \colon (X,d) \to Y$ is an $\eta$-branched quasisymmetry if and only if $f \colon (X, d_M) \to Y$ is an $\eta$-branched quasisymmetry, where we identify $X$ and $(X,d_M)$ by the identity map. Thus, we may reformulate Theorem \ref{thm:power} as follows.
%with the identity map identification.
\begin{cor}
\label{cor:power}
Let $X$ and $Y$ be compact metric spaces and suppose that $X$ admits the Mazurkiewicz metric $d_M$ and that $\mathrm{id}\colon X\to (X,d_M)$ is a homeomorphism, and let $f\colon X\to Y$ be a branched quasisymmetry. Then there exist constants $C>0$ and $\alpha\in (0,1)$ for which $f$ is $\eta$-BQS for the distortion function $\eta\colon [0,\infty) \to [0,\infty)$ given by 
\[
t\mapsto C \max\{ t^\alpha, t^{1/\alpha}\}.
\] 
\end{cor}

  As $(X, d_M)$ has bounded turning in this situation, we may apply Theorem \ref{thm:power} and the fact that the distortion functions are the same to conclude that $f$ is a power branched quasisymmetry.  Despite this observation, we focus much of our attention in this paper on bounded turning spaces as these spaces allow the construction of diametric hulls.

Having mappings between non-compact spaces in mind, we also comment the validity of the power quasisymmetry, and other general results, for local branched quasisymmetries. A map $f\colon X\to Y$ is a \emph{local branched quasisymmetry} if there is a homeomorphism $\eta\colon [0,\infty) \to [0,\infty)$ with the property that, for each point $x\in X$, there exists $\varepsilon>0$ for which \eqref{BQS} holds for all intersecting continua $E$ and $E'$ in $B(x,\varepsilon)$. 
Note that, if $X$ is compact, then this definition is equivalent to the condition that there exists $\varepsilon>0$ and a homeomorphism $\eta\colon [0,\infty) \to [0,\infty)$ for which inequality \eqref{BQS} holds for all intersecting continua $E$ and $E'$ of diameter at most $\varepsilon$. In this case, we call $\varepsilon$ the \emph{locality scale of $f$}.

In Part \ref{part:hyperbolic}, consisting of Sections \ref{sec:MG}--\ref{sec:VQI}, we discuss the hyperbolic fillings of compact spaces and properties of vertical quasi-isometries between hyperbolic fillings. 

The main results, Theorems \ref{thm:BQS2VQI} and \ref{thm:intro-VQI2pBQS}, of this article are discussed in Part \ref{part:extension}, which consists of Sections \ref{BQS to VQI sec} and \ref{VQI to BQS sec}. 

We prove Theorem \ref{thm:BQS2VQI} in Section \ref{BQS to VQI sec}, alongside the multiplicity estimate for the vertical quasi-isometry. In that section, we also pass to the terminology that the map  $\varphi \colon \widehat X\to \widehat Y$ induced by the branched quasisymmetry $f\colon X\to Y$ a \emph{hyperbolic filling of $f$}.

Theorem \ref{thm:intro-VQI2pBQS} is proven in Section \ref{VQI to BQS sec}; we use Theorem \ref{thm:power} to reformulate the statement for power branched quasisymmetries. 
As an application of the extension results, we show that local branched quasisymmetries between compact and doubling bounded turning spaces are branched quasisymmetric (Theorem \ref{thm:local-self-improvement}).

Finally, in the appendix we discuss the relationship of quasiregular mappings and branched quasisymmetries between closed Riemannian manifolds and the proof of Theorem \ref{thm:QR=BQS}. In the proof, we use the aforementioned self-improvement property of local quasisymmetries.

\begin{ack}
Both authors thank Nageswari Shanmugalingam for discussions regarding topics of this manuscript and especially on the Mazurkiewicz metric.
J.L thanks the Universities of Helsinki and Cincinnati for their support during this work. He also thanks Angela Wu for many BQS conversations and continued interest. P.P thanks IMPAN and the organizers of the Simons semester \emph{Geometry and analysis in function and mapping theory on Euclidean and metric measure spaces} for creating an inspiring research environment, which contributed to the completion of the manuscript. 
\end{ack}

\section*{Notation} 

We denote a metric space $(X,d)$ simply as $X$.  We use the distance notation $d$ or $d_X$ for distances in the ``boundary'' metric spaces and we use the Polish notation $|x-y|$ for the distances of points $x$ and $y$ in the ``filling'' metric spaces. We use the notations $B(x,r)=\{ x' \in X \colon d(x',x)<r\}$ and $\bar B(x,r) = \{x' \in X \colon d(x',x)\le r\}$ for open and closed metric balls of radius $r>0$ about $x\in X$, respectively. For $x\in X$ and $r>0$, we call $S(x,r) = \{ y \in X\colon d(x,y)=r\}$ the \emph{sphere of radius $r$ centered at $x$}. For each $\lambda>0$ and a ball $B$ of radius $r$ and center $x$ in $X$, we denote $\lambda B$ the ball of radius $\lambda r$ and center $x$ in $X$. For a subset $A\subseteq X$ and $r>0$, we also denote $B(A,r) = \{ x\in X\colon \dist(x,A)<r\}$ the \emph{$r$-neighborhood of $A$}.

Our metric spaces are proper unless otherwise stated.  That is, closed and bounded sets are compact in $X$. 

A pointed space $(X,\ast)$ is a metric space $X$ with a distinguished point $\ast \in X$. A map $f\colon (X,\ast) \to (Y,\ast)$ between pointed spaces is always assumed to be a pointed map, that is, $f(\ast) = \ast$.

\part{Preliminaries on bounded turning and branched quasisymmetry}
\label{part:bt-bqs}

\section{Metric continua and bounded turning spaces}
\label{sec:bounded_turning}

Before discussing bounded turning, we begin by recalling some basic properties and terminology on metric continua. Recall that a metric space $X$ is a \emph{continuum} if $X$ is compact and connected space which is not a point.

\begin{lemma}
\label{lemma:continuum-continuum}
Let $X$ be a continuum. Then, each ball in $X$ contains a continuum. 
\end{lemma}
\begin{proof}
Let $B(x,r)$ be a ball in $X$ and $y\in X$ a point distinct from $x$. Let $E$ be a connected component of $\bar B(x,r) \cap E$ containing $x$. Then $E$ is compact and connected. Suppose $E$ is not a continuum. Then $E$ is a singleton. Since $E$ is a component, we obtain that $X$ is not connected. This is a contradiction. Thus each closed ball contains a continuum. Hence each ball contains a continuum.
\end{proof}

We record also a simple observation on metric spheres and diameters of small metric balls in metric continua. For the statement, we denote $R_X(x) = \max_{y\in X} d(x,y)$ for each $x\in X$. Note that, for each $x\in X$, $R_X(x)\ge (\diam X)/2$.

\begin{lemma}
\label{lemma:continuum-spheres}
Let $X$ be a continuum and $x\in X$. Then, for $x\in X$ and $0<r<R_X(x)$, the sphere $S(x,r)$ is non-empty and $\diam B(x,r) \ge r$.
\end{lemma}
\begin{proof}
Since $r<R_X(x)$, we have that $B(x,r)\ne X$ and there exists $y\in X\setminus B(x,r)$. Since $X$ is connected, we have that $S(x,r)\ne \emptyset$. For the second claim it suffices to observe that, for each $0<r'<r$, we have $S(x,r')\ne \emptyset$, and hence $\diam B(x,r)\ge r'$. The claim follows.
\end{proof}

We are now ready to move from continua to bounded turning spaces. Heuristically, a space has bounded turning if we can join all pairs of points by continua of comparable length, quantitatively. 

\begin{defn}
A metric space $X$ has \emph{$\lambda$-bounded turning} for $\lambda \ge 1$ if, for all $x,y\in X$, there exists a continuum $E_{xy} \subseteq X$ containing points $x$ and $y$ and for which $\diam E_{xy} \le \lambda d(x,y)$.
\end{defn}

Since metric balls in bounded turning spaces need not be connected, we define the notion of a diametric hull.  

\begin{defn}
\label{def:diametric-hull}
Let $X$ be a metric space and $A\subset X$. A continuum $E$ in a metric space $X$ is a \emph{$\lambda$-diametric hull of $A$ for $\lambda \ge 1$} if $A \subset E$ and $\diam E \le 2\lambda \diam A$.
\end{defn}

In a proper bounded turning space, each bounded set has a bounded diametric hull, quantitatively. Recall that a metric space $X$ is \emph{proper} if closed balls of $X$ are compact.

\begin{lemma}\label{continuum lemma}
Let $\lambda \ge 1$ and let $X$ be a proper and $\lambda$-bounded turning space containing at least two points and $B=B(x,r)$ a ball in $X$.  Then there exists a continuum $E_B \subseteq X$ satisfying $B \subseteq E_B \subseteq 2 \lambda B$. \emph{A fortiori}, each bounded subset of $X$ has a $\lambda$-diametric hull. 
\end{lemma}

\begin{proof}
Let $B=B(x,r)$ be ball in $X$. Since $X$ has $\lambda$-bounded turning, we may, for each $x' \in \overline{B}$, fix a continuum $E_{x'}$ connecting $x$ and $x'$, which satisfies $\diam E_{x'} \leq \lambda \diam B$.  Let $E_B = \overline{\cup_{x' \in B} E_{x'}}$.  The set $E_B$ is closed and bounded. Hence $E_B$ is compact.  Since each $E_{x'}$ contains $x$ and is connected, the set $E_B$ is connected. Since $B \subseteq E_B$ and $B$ has at least two points, the set $E_B$ is a continuum.  Every point in $E_B$ is of distance at most $\lambda r$ to $x$, so $E_B \subseteq 2 \lambda B$. Thus $E_B$ is a continuum satisfying the required conditions.

Let now $A\subset X$ be a bounded set and $x\in A$.  Then $A\subset \bar B(x,\diam A)$. Since the above argument holds also for closed balls, there exists a continuum $E_A \subset X$ satisfying $\bar B(x,\diam A) \subset E_A \subset \bar B(x,2\lambda \diam A)$. The claim follows.
\end{proof}

\subsection{Continuum lifting lemma}

We conclude this section with a continuum lifting lemma for a diametric hull neighborhood of a set. Let $X$ be a proper $\lambda$-bounded turning space. A \emph{metric hull neighborhood $E(A,\theta)$ of a subset $A\subset X$ of inner radius $\theta>0$} is the set
\[
E(A,\theta) = \overline{\bigcup_{x\in A} E_{B(x,\theta)}}.
\]
Clearly $B(A,\theta) \subset E(A,\theta)$. Also, $\diam E(A,\theta) \le \diam A + 2\lambda \theta$. Furthermore, if $A$ is connected, then so is $E(A,\theta)$. Indeed, for each $x\in A$, the set $A\cup E_{B(x,\theta)}$ is connected, and $\bigcap_{x\in A} (A\cup E_{B(x,\theta)}) \supset A$. Thus $E(A,\theta)$ is connected. In particular, if $A$ is connected, then $E(A,\theta)$ is a continuum.

\begin{lemma}[Continuum Lifting Lemma]\label{cont lift}
Let $X$ and $Y$ be compact $\lambda$-bounded turning metric spaces, let $f \colon X \to Y$ be a discrete and open mapping, and let $x\in X$.  Let also $G \subseteq Y$ be a continuum containing $f(x)$, and $\theta>0$. Then there exists a continuum $E \subseteq Z$ containing $x$ for which $G \subseteq fE \subseteq E(G,\theta)$.
\end{lemma}

\begin{proof}
Let $E$ be the component of $f^{-1} E(G,\theta)$ containing $x$. Since $f$ is continuous, $E$ is closed and hence compact.  Since $E$ is a component, it is connected. Since $f(x)$ is an interior point of $E(G,\theta)$, we also have that $x$ is an interior point of $E$. Thus $E$ is a continuum by Lemma \ref{continuum lemma}.

It remains to show that $G\subset fE$. Let $H = G \cap fE$. We show $H$ is both open and closed in $G$. Since $fE$ is compact, it is closed in $Y$. Thus $H$ is closed in $G$ by relative topology. To see $H$ is open in $G$, let $y' \in H$ and fix $x' \in E\cap f^{-1}(y')$.  Since $f$ is continuous, there exists $t > 0$ for which $f B(x', t) \subseteq B(y', \theta) \subseteq E(G,\theta)$.  Let $t'>0$ be small enough so that $E_{B(x', t')} \subseteq B(x', t)$.  Then, $fE_{B(e, t')} \subseteq E(G,\theta)$. So $E \cup E_{B(x', t')}$ is a continuum containing $x$ and $f(E\cup E_{B(x',t')}) \subseteq E(G,\theta)$. Since components are maximal connected subsets, we have that $E_{B(x', t')} \subseteq E$.  Since $f$ is open and $B(x', t') \subseteq E$, we have that $fE$ contains a ball centered at $y' = f(x')$.  Thus $H$ is open in $G$.  Since $G$ is connected, we conclude $H = G$.  

Let $w \in G \setminus fB$.  From $H = G$, there must be a point $e' \in E$ with $fe' = w$.  As $w \notin fB$, we must have $e' \notin B$.  Hence, $d(x, e') \geq r$, so $\diam(E) \geq r$.  
\end{proof}

We use this lemma in the following form.
\begin{cor}
\label{cor:cont lift}
Let $X$ and $Y$ be compact and $\lambda$-bounded turning  metric spaces.  Let $f \colon X \to Y$ be a discrete and open mapping and let $B = B(x, r) \subseteq X$ be a ball. Suppose $G \subseteq Y$ is a continuum containing $f(x)$ and not contained in $fB$. Then, for each $\theta>0$, there exists a continuum $E \subseteq X$ containing $x$ such that $E\not\subseteq B$ and $fE \subseteq E(G,\theta)$.
\end{cor}

\section{Branched quasisymmetries}

In this section we record some basic properties of branched quasisymmetries. First, a composition of branched quasisymmetries is a branched quasisymmetry, qualitatively.
\begin{lemma}
Let $f \colon X \to Y$ and $g \colon Y \to Z$ be branched quasisymmetries with gauge functions $\eta_f \colon [0,\infty) \to [0,\infty)$ and $\eta_g \colon [0,\infty) \to [0,\infty)$, respectively. Then the map $h = g \circ f \colon X \to Z$ is a branched $(\eta_g \circ \eta_f)$-quasisymmetry.  
\end{lemma}

\begin{proof}
By passing to a connected component of $X$ if necessary, we may assume that $X$ is connected. A fortiori, we may assume $f$ is not constant.

Let $E, E' \subseteq X$ be intersecting continua.  Since $fE$ and $fE'$ are intersecting continua, we have that
\begin{align*}
\diam(hE) = \diam g(fE) &\leq \eta_g \biggl(\frac{\diam(fE)}{\diam(fE')}\biggr) \diam g(fE') \\
&\leq \eta_g \biggl( \eta_f \biggl( \frac{\diam(E)}{\diam(E')} \biggr) \biggr)
\diam g(fE').
\end{align*}
This completes the proof.
\end{proof}

The second observation is that quasisymmetries are branched quasisymmetries quantitatively; see Guo--Williams \cite[Remark 6.49]{GW}.
\begin{lemma}
\label{lemma:qs-bqs}
An $\eta$-quasisymmetry $f \colon X \to Y$ for $\eta\colon [0,\infty) \to [0,\infty)$ is a branched $\eta'$-quasisymmetry for $\eta' \colon t \mapsto 2\eta(2 t)$.
\end{lemma}

\begin{proof}
Let $E$ and $E'$ be intersecting continua in $X$ and $z_0\in E\cap E'$. Let also $x\in E$ and $x'\in E'$ be points which maximize the distances $d(f(z_0),f(x))$ and $d(z_0,y)$ among the points in $E$ and $E'$, respectively. Then $\diam fE \le 2 d(f(z_0),f(x))$ and $\diam E' \le 2 d(z_0,y)$. Thus
\begin{align*}
\diam fE \leq 2 d(f(z_0),f(x)) &\le 2\eta\left( \frac{d(x_0,x)}{d(x_0,y)}\right) d(f(z_0),f(y)) \\
&\le 2\eta\left( 2\frac{\diam E}{\diam E'}\right) \diam fE'.
\end{align*}
This completes the proof.
\end{proof}

Similarly to quasisymmetries, branched quasisymmetries easily form normal families due to equicontinuity; cf.\;Heinonen \cite[Section 10.25]{He} for analogous discussion. 

\begin{lemma}
Let $X$ be a $\lambda$-bounded turning space for $\lambda \ge 1$, $Y$ a metric space, and $\eta\colon [0,\infty) \to [0,\infty)$ a homeomorphism. Then, for each $M>0$ and continuum $E$ in $X$, the family of $\eta$-branched quasisymmetries $f \colon X\to Y$ satisfying $\diam fE \le M$ is equicontinuous. 
\end{lemma}
\begin{proof}
Let $M>0$ and let $E$ be a continuum in $X$. We fix a point $a\in E$. Let $x$ and $y$ be points in $X$. We fix first a continuum $E_a$ connecting $x$ to $a$ for which $\diam E_a \le \lambda d(a, x)$, and then a continuum $E_{y}$ connecting $x$ and $y$ for which $\diam E_{y} \le \lambda d(x, y)$. Since $x\in E_y \cap E_a$ and $a\in E\cap E_a$, we have that
\begin{align*}
d(f(x),f(y)) &\le \diam fE_{y} \le \eta\left( \frac{\diam E_y}{\diam E_a} \right) \diam fE_a \\
& \le \eta\left( \frac{\diam E_y}{\diam E_a} \right) \eta\left( \frac{\diam E_a}{\diam E}\right) \diam fE \\
&\le \eta\left( \lambda \frac{d(x,y)}{d(x,a)} \right) \eta\left( \lambda \frac{d(x,a)}{\diam E}\right) M. 
\end{align*}
Thus the family is equicontinuous at $x$.
\end{proof}

Note that is it almost immediate from the definition that locally uniform limits of branched quasisymmetries are branched quasisymmetries. In what follows, we say a metric space $X$ is continuum-connected if any two points can be joined by a continuum.
  
\begin{lemma}
\label{lemma:BQS-local-limits}
Let $X$ and $Y$ be metric spaces. Let $(f_i \colon X\to Y)$ be a sequence of branched $\eta$-quasisymmetries which converge locally uniformly to a (possibly constant) map $f\colon X\to Y$. Then $f$ is branched $\eta$-quasisymmetric.
\end{lemma}

\begin{proof}
Let $E$ and $E'$ be intersecting continua in $X$. Since $f_i \to f$ locally uniformly, sequences $(f_iE)$ and $(f_iE')$ converge in the Hausdorff distance to compact sets $fE$ and $fE'$, respectively. From local uniform convergence, we have that $\diam f_iE \to \diam fE$ and $\diam f_iE' \to \diam fE'$ as $i \to \infty$. Thus
\begin{align*}
\diam fE = \lim_{i \to \infty} \diam f_i E &\le \limsup_{i \to \infty} \eta\left( \frac{\diam E}{\diam E'} \right) \diam f_i E' \\
&= \eta\left( \frac{\diam E}{\diam E'} \right) \diam f E'
\end{align*}
The claim is proven.
\end{proof}

\subsection{Examples of branched quasisymmetries}

We first discuss the relationship of branched quasisymmetric homeomorphisms and quasisymmetries. Lemma \ref{lemma:qs-bqs} admits a converse for bounded turning spaces: \emph{Let $X$ and $Y$ be bounded turning spaces. Then a branched quasisymmetric homeomorphism $X\to Y$ is quasisymmetric, quantitatively}; see Guo--Williams \cite[Proposition 6.48]{GW}. 

This correspondence of branched quasisymmetric homeomorphisms and quasisymmetries, however, does not hold if we only assume that only one of the spaces has bounded turning. 
\begin{example}
Consider spaces $X=([0,1]\times \{0\}) \cup G \subset \R^2$, where $G=\{ (x,x^2) \in \R^2 \colon x\in [0,1]\}$, and $Y=[-1,1]$. Note that $X$ does not have bounded turning. Let now $f\colon X\to Y$ be the map $(x,0) \mapsto x$ and $(x,x^2) \mapsto -x$. For all continua $E\subset G$, we have that $(\diam E)/2\le \diam(fE) \le \diam(E)$. Thus, by a simple case study, we observe that $f$ is a branched quasisymmetric homeomorphism. However, $f$ is not a quasisymmetry. Indeed, suppose that $f$ is $\eta$-quasisymmetric. Let $x\in (0,1]$. Then 
\[
2 = \frac{|f(x,0)-f(x,x^2)|}{|f(0,0)-f(x,x^2)|} \le \eta\left( \frac{|(x,0)-(x,x^2)|}{|(0,0)-(x,x^2)|}\right) \le \eta\left( \frac{x^2}{\sqrt{x^2 + x^4}} \right) \le \eta\left( x\right).
\]
This is a contradiction, since $\eta \colon [0,\infty)\to [0,\infty)$ is a homeomorphism. A similar example, in the case that the target is not bounded turning, is given by the inverse $f^{-1}\colon Y \to X$ of $f$. 
\end{example}

It is also easy to find branched quasisymmetries which are neither discrete nor open. Clearly, the inclusion $[0,1] \to [0,1]^2$ is a branched quasisymmetry which is not open. Classical folding maps give bit more elaborate examples of a branched quasisymmetry between two dimensional spaces which are not open. The following example is of this type. In this, and in the following example, we let $\pi_i \colon \R^2 \to \R$ be the coordinate projections $(x_1,x_2) \mapsto x_i$ for $i=1,2$. 

\begin{example}
Consider the spaces $X=Y=[-1,1]\times [0,1]$ and the map $f \colon X \to Y$, $(x_1,x_2) \mapsto (h(x_1),x_2)$, where $h \colon [-1,1]\to [-1,1]$ is a piece-wise linear function 
\[
t\mapsto \left\{ \begin{array}{ll}
3(t+1)-1, & t \le -1/2 \\
|t|, & |t| \le 1/2 \\
t, & t \ge 1/2.
\end{array}\right.
\]
For a continuum $E \subseteq X$, we have that $\diam(\pi_2(fE)) = \diam(\pi_2(E))$ and that $\diam(\pi_1(E)/2 \leq \diam(\pi_1(fE)) \leq 3\diam(\pi_1(E))$. Thus $f$ is a branched quasisymmetry. The map $f$ is not open, since $f((-1/2,1/2)\times (0,1)) = [0,1/2) \times (0,1)$. Note that, by the choice of $h$, the issue of openness of $f$ is not related to the surjectivity of the map. 
\end{example}

The following example shows that a branched quasisymmetry need not be discrete. The example is again folding based. This is not a surprise, since by Stoilow's theorem \cite{Stoilow, Stoilow-book} an open and light map is discrete; see also \cite{Luisto-Pankka}. 
 
\begin{example}
Define a function $g \colon [0,3] \to [0,1]$ as follows. For each $k \in \N_0$, let $a_k = 1 + 2(\sum_{j=1}^k 2^{-j})$ and define
\[
g(x) = \left\{ \begin{array}{ll}
1-x, & 0 \le x \le 1, \\
x-a_k, & a_k \le x \le a_k + 2^{-k-1}, \\
a_{k+1}-x, & a_k + 2^{-k-1} \le x \le a_{k+1};
\end{array}\right.
\]
see Figure \ref{fig:only-fig}.

\begin{figure}[h!]
\begin{tikzpicture}[scale = 2]

\draw[<->,thick] (-.25,0) -- (3.25,0) node[below]{$x$} ;
\draw[<->, thick] (0,-.25) -- (0,1.25) node[left] {$g(x)$};
\foreach \x in {1,2,3} \draw (\x, 2pt) -- (\x, -2pt) node[below]{\x};

\draw[domain = 0:1.01, line width = .7mm, smooth] plot (\x, 1 - \x);
\draw[domain = .99:1.51, line width = .7mm, smooth] plot (\x, \x - 1);
\draw[domain = 1.49:2.01, line width = .7mm, smooth] plot (\x, 2 - \x);
\draw[domain = 2.01:2.26, line width = .7mm, smooth] plot (\x, \x - 2);
\draw[domain = 2.24:2.51, line width = .7mm, smooth] plot (\x, 2.5 - \x);
\draw[domain = 2.49:2.635, line width = .7mm, smooth] plot (\x, \x - 2.5);
\draw[domain = 2.615:2.76, line width = .7mm, smooth] plot (\x, 2.75 - \x);
\end{tikzpicture}
\caption{Part of the graph of the function $g$.}
\label{fig:only-fig}
\end{figure}
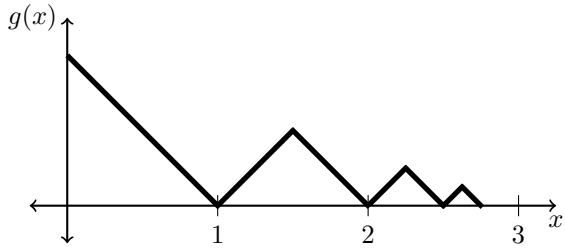

Consider now spaces $X = [0,3] \times [0,1]$ and $Y = [0,1]^2$, and let $f \colon X \to Y$ be the map $(x_1,x_2) \mapsto (g(x_1), x_2)$. As in the previous example, we observe that, for a continuum $E \subseteq X$, we have that $\diam(\pi_2(fE)) = \diam(\pi_2(E))$ and $\diam(\pi_1(E))/2\leq \diam(\pi_1(fE)) \leq \diam(\pi_1(E))$. Hence, $f$ is a branched quasisymmetry. Since $f^{-1}(0,0)$ accumulates to $(3,0)\in X$, we observe that $f$ is not discrete.
\end{example}

\subsection{Diametric hull lemma}

In what follows we will use repeatedly the following observation that, for branched quasisymmetries from bounded turning spaces, images of balls and their diametric hulls have comparable diameters.
 
\begin{lemma}
\label{lemma:BQS-ball continuum equiv}
Let $X$ and $Y$ be compact metrically Ahlfors regular spaces, where $X$ is $\lambda$-bounded turning. Let also $f \colon X \to Y$ be an branched $\eta$-quasisymmetry.  Then, for each ball $B\subset X$ and its $\lambda$-diametric hull $E_B \subset X$, we have
\begin{equation*}
\diam(fE_B) \simeq \diam (fB),
\end{equation*}
where the implicit constants depend only on $\eta$ and $\lambda$.
\end{lemma}

\begin{proof}
Let $B =B(x,r) \subset X$ be a ball. We may assume that $B\ne X$. Since $fB \subseteq fE_B$, we have
\begin{equation*}
\diam (fB)  \leq \diam(fE_B).
\end{equation*}

To the other direction, let $b = 1/(2\lambda)$.  Then, by Lemma \ref{continuum lemma}, we have
\[
bB \subseteq E_{bB} \subseteq B \subseteq E_B \subseteq 2\lambda B.
\]
Since $E_{bB}$ and $E_B$ are continua and $f$ is an $\eta$-branched quasisymmetry, we have
\begin{align*}
\diam(f E_B) \le \eta \left( \frac{\diam(E_B)}{\diam(E_{bB})}\right)\diam(f E_{b B}) \le \eta \left( \frac{\diam( 2\lambda B)}{\diam(bB)}\right) \diam(f B).
\end{align*} 
By Lemma \ref{lemma:continuum-spheres}, we have that $\diam(bB) \ge b r$. Thus
\[
\frac{\diam( 2\lambda B)}{\diam(bB)} \le \frac{4\lambda}{b} = 8\lambda^2.
\]
The claim follows.
\end{proof}

\begin{remark}
\label{rmk:local ball continuum equiv}
Lemma \ref{lemma:BQS-ball continuum equiv} holds also for local branched quasisymmetries in the following form: \emph{Let $f\colon X\to Y$ be a local branched $\eta$-quasisymmetry. Then for $R=\varepsilon/\lambda>0$, where $\varepsilon$ is the locality scale of $f$ and $\lambda$ the bounded turning constant of $X$, we have that $\diam(fE_B) \simeq \diam(fB)$ for each balls $B\subset X$ of radius at most $R$.} The proof is verbatim and is omitted.
\end{remark}

\section{Power branched quasisymmetries}
\label{sec:power_quasisymmetries}

It is well-known that a quasisymmetry between uniformly perfect spaces is a power quasisymmetry; see Tukia and V\"ais\"al\"a \cite[Corollary 3.12]{Tukia-Vaisala_Annals} or e.g.~Heinonen \cite[Theorem 11.3]{He}. As mentioned in the introduction, an analogous result holds also for branched quasisymmetries defined on a bounded turning space. The following theorem is a quantitative version of Theorem \ref{thm:power} stated in the introduction.

\begin{thm}
\label{thm:BQS to PBQS}
Let $X$ and $Y$ be compact metric spaces and suppose that $X$ has $\lambda$-bounded turning for $\lambda \ge 1$, and let $f \colon X \to Y$ be an branched $\eta$-quasisymmetry with $\eta\colon [0,\infty) \to [0,\infty)$.  Then $f$ is an $\eta'$-branched quasisymmetry, where $\eta'\colon [0,\infty) \to [0,\infty)$ is the homeomorphism 
\[
t \mapsto C_{\eta'} \max\{t^q, t^{1/q}\},
\]
where constants $C_{\eta'}>0$ and $q > 0$ depend only on $\eta$ and $\lambda$.
\end{thm}

We divide the proof into three cases and consider the first two cases separately in auxiliary propositions. We begin with a lemma, which records an immediate observation based on definitions.

\begin{lemma}
\label{lemma:theta-lemma}
Let $X$ and $Y$ be compact metrics and suppose in addition that $X$ has $\lambda$-bounded turning for $\lambda \ge 1$. Let $f \colon X \to Y$ be an $\eta$-branched quasisymmetry with $\eta\colon [0,\infty) \to [0,\infty)$ and let also $\theta >0$ and $c\in X$. Suppose $x_0,\ldots, x_s$ is a sequence of points in $X$ with the property that 
\[
d(x_{i+1},c) = \theta d(x_i,c)
\]
for each $i=0,\ldots, s-1$. For each $i=0,\ldots, s$, let $E_i$ be a continuum of diameter at most $\lambda d(x_i,c)$ connecting $x_i$ and $c$. Then 
\[
\frac{1}{\eta(\lambda/\theta)^s} \le 
\frac{\diam fE_s}{\diam f E_0} \le \eta(\lambda\theta)^s.
\]
\end{lemma}

\begin{proof}
By assumption,
\[
\frac{\diam E_{i+1}}{\diam E_i} \le \frac{\lambda d(x_{i+1},c)}{d(x_i,c)} = \lambda \theta
\]
for each $i=0,\ldots, s$. Thus, since all continua $E_0,\ldots, E_s$ have the point $c$ in common, we have, by the definition of branched quasisymmetry, that
\begin{align*}
\frac{\diam fE_s}{\diam fE_0} &= \frac{\diam fE_s}{\diam fE_{s-1}} \frac{\diam fE_{s-1}}{\diam fE_{s-2}} \cdots \frac{\diam fE_1}{\diam fE_0} \\
&\le \eta\left(\frac{\diam E_s}{\diam E_{s-1}}\right) \eta\left(\frac{\diam E_{s-1}}{\diam E_{s-2}}\right) \cdots \eta\left(\frac{\diam E_1}{\diam E_0}\right) 
\le \eta( \lambda \theta)^s.
\end{align*}
The other estimate is obtained similarly.
\end{proof}

\begin{prop}
\label{prop:power-a}
Let $X$ and $Y$ be compact metrics and suppose in addition that $X$ has $\lambda$-bounded turning for $\lambda \ge 1$. Let $f \colon X \to Y$ be an $\eta$-branched quasisymmetry with $\eta\colon [0,\infty) \to [0,\infty)$. Let $E$ and $E'$ be continua in $X$ intersecting at $c\in E\cap E'$. Suppose that there exists a point $a\in E$ satisfying 
\[
\diam fE \le 3 d(f(a),f(c))
\]
and
\[
\diam E' \le 3d(a,c).
\]
Then there exist constants $C>0$ and $\alpha$, depending only on $\eta$ and $\lambda$, for which 
\[
\frac{\diam fE}{\diam fE'} \le C \left( \frac{\diam E}{\diam E'}\right)^\alpha.
\]
\end{prop}

\begin{proof}
Since $\diam E' \le 3d(a,c)$, $c\in E'$, and $E'$ is a continuum, we may fix, by Lemma \ref{lemma:continuum-spheres}, a point $b\in E'$ satisfying 
\[
\diam E' \le 3d(b,c).
\]

First suppose that $d(a,c) \geq d(b,c)$.  Let $x_0 = a$ and let $x_1,\ldots, x_s$ be a sequence of points in $X$ satisfying 
\[
d(x_{i+1},c) = d(x_i,c)/2
\]
for each $i=0,\ldots, s-1$, where $s\in \N_0$ is the unique index satisfying
\[
d(x_s,c)/2 < d(b,c) \le d(x_s, c);
\]
such points exist by Lemma \ref{lemma:continuum-spheres}. 

For each $i=0,\ldots, s$, we fix a continuum $E_i$ of diameter at most $\lambda d(x_i,c)$ connecting $x_i$ and $c$. Then, by Lemma \ref{lemma:theta-lemma}, 
\[
\frac{\diam fE_0}{\diam fE_s} \le \eta(\lambda/2)^s.
\]

Since
\[
\diam fE \le 3 d(f(a),f(c)) \le 3 \diam fE_0
\]
and 
\[
\diam E' \ge d(b,c) > d(x_s,c)/2 \ge \left( \diam E_s\right)/(2\lambda),
\]
we have, by the definition of branched quasisymmetry, that
\begin{align*}
\frac{\diam fE}{\diam fE'} &= \frac{\diam fE}{\diam fE_0} \frac{\diam fE_0}{\diam fE_s}\frac{\diam fE_s}{\diam fE'} \\
&\le 3  \eta\left(\frac{\diam E_s}{\diam E'}\right) \eta( \lambda/2)^s
\le 3\eta(2\lambda) \eta(\lambda/2)^s.
\end{align*}
Since
\[
\frac{\diam E}{\diam E'} \ge \frac{d(a,c)}{3d(b,c)} \ge \frac{2^s d(x_s,c)}{3d(x_s,c)} =\frac{2^s}{3},
\]
we have that
\[
\frac{\diam fE}{\diam fE'} \le C \left( \frac{\diam E}{\diam E'}\right)^\alpha,
\]
where constants $C>0$ and $\alpha>0$ depend only on $\eta$ and $\lambda$.

If instead we have $d(a,c) < d(b,c)$, then we choose $E_0$ to be a continuum connecting $c$ and $a$ with $\diam(E_0) \leq \lambda d(a,c)$.  Then, as above, 
\[
\diam fE \leq 3 \diam f E_0.
\]
We also have 
\[
\diam E' \geq d(b,c) > d(a,c) \geq \frac{\diam E_0}{\lambda}
\]
and so
\[
\frac{\diam fE}{\diam fE'} \leq \frac{\diam fE}{\diam fE_0} \frac{\diam fE_0}{\diam fE'} \leq 3 \eta(\lambda)
\]
and
\[
\frac{\diam E}{\diam E'} \geq \frac{d(a,c)}{3 d(a,c)} = \frac{1}{3},
\]
so we can achieve the desired form of $\eta$ by possibly increasing the value of $C$ from above.
\end{proof}

\begin{prop}
\label{prop:power-b}
Let $X$ and $Y$ be compact metrics and suppose in addition that $X$ has $\lambda$-bounded turning for $\lambda \ge 1$. Let $f \colon X \to Y$ be an branched $\eta$-quasisymmetry with $\eta\colon [0,\infty) \to [0,\infty)$. Let $E$ and $E'$ be continua in $X$ intersecting at $c\in E\cap E'$. Suppose there exists a point $b\in E'$ satisfying
\[
\max\{ \diam E, \diam E'\} \le 3d(b,c).
\]
Then there exists constants $C>0$ and $\alpha>0$, depending only on $\eta$ and $\lambda$, for which
\[
\frac{\diam fE}{\diam fE'} \le C \left( \frac{\diam E}{\diam E'} \right)^\alpha.
\]
\end{prop}

\begin{proof}
Since $c\in E$ and $E$ is a continuum, there exists, by Lemma \ref{lemma:continuum-spheres}, a point $a\in E$ for which 
\[
\diam E \le 3d(a,c).
\]

We first assume that $d(a,c) \leq d(b,c)$.  Let $\theta>0$ be the unique number for which $\eta(\lambda \theta) = 1/2$. Note that $\theta < 1$ because $\eta(1) \geq 1$ (take $E = E'$, for example) and $\lambda \geq 1$.  Let $x_0 = b$. We fix a sequence $x_1,\ldots, x_s$ of points in $X$ satisfying 
\[
d(x_{i+1},c) = \theta d(x_i,c)
\]
for each $i=0,\ldots, s-1$, where $s\in \N_0$ is the unique index satisfying
\[
\theta d(x_s,c) < d(a,c) \le d(x_s,c).
\]
As before, we fix for each $i=0,\ldots, s$, a continuum $E_i$ of diameter at most $\lambda d(x_i,c)$ connecting $x_i$ and $c$.

Since 
\[
\diam E \le 3d(a,c) \le 3d(x_s,c) \le 3\lambda \diam E_s,
\]
we have, by the $\eta$-branched quasisymmetry of $f$, that
\[
\frac{\diam fE}{\diam fE_s} \le \eta\left( \frac{\diam E}{\diam E_s}\right) \le \eta(3\lambda). 
\]
Similarly, since 
\[
\diam E_0 \le \lambda d(b,c) \le \lambda \diam E',
\]
we have that
\[
\frac{\diam fE_0}{\diam fE'} \le \eta\left(\frac{\diam E_0}{\diam E'}\right) \le \eta(\lambda).
\]

Thus, by Lemma \ref{lemma:theta-lemma},
\begin{align*}
\frac{\diam fE}{\diam fE'} &= \frac{\diam fE}{\diam fE_s}  \frac{\diam fE_s}{\diam fE_0} \frac{\diam fE_0}{\diam fE'} \\
&\le \eta(3\lambda)\eta(\lambda) \eta(\lambda \theta)^s 
= \eta(3\lambda)\eta(\lambda) 2^{-s}.
\end{align*}

Since
\[
\frac{\diam E}{\diam E'} \ge \frac{d(a,c)}{3d(b,c)} \ge \frac{\theta d(x_s,c)}{3\theta^{-s} d(x_s,c)} = \frac{\theta}{3} \theta^{s},
\]
we conclude that
\[
\frac{\diam fE}{\diam fE'} \le C \left( \frac{\diam E}{\diam E'}\right)^\alpha,
\]
where $C>0$ and $\alpha>0$ depend only on $\eta$ and $\lambda$.

Now, suppose instead that $d(b,c) < d(a,c)$.  Then, 
\[
\frac{\diam E}{\diam E'} \leq \frac{3 d(b,c)}{d(b,c)} \leq 3
\]
and so 
\[
\frac{\diam(fE)}{\diam(fE')} \leq \eta(3).
\]
Moreover,
\[
\diam(E') \leq 3 d(b,c) < 3d(a,c) \leq 3 \diam(E)
\]
and so in this case we have
\[
\frac{1}{3} \leq \frac{\diam(E)}{\diam(E')}.  
\]
Hence, as before, we can guarantee that $\eta$ has the desired form by choosing $C$ large.
\end{proof}

\begin{proof}[Proof of Theorem \ref{thm:BQS to PBQS}]
Let $E$ and $E'$ be continua in $X$ intersecting at $c\in E\cap E'$. If there exist points $a\in E$ or $b\in E'$ which satisfy the conditions in Proposition \ref{prop:power-a} or \ref{prop:power-b}, respectively, we have that 
\[
\frac{\diam fE}{\diam fE'} \le C \left( \frac{\diam E}{\diam E'} \right)^\alpha,
\]
where $C>0$ and $\alpha>0$ depend only on $\eta$ and $\lambda$ (note that there are two possible $\alpha$ values, one from each lemma; we choose the larger $\eta$ expression depending on $\diam(E)/\diam(E')$ which is what gives rise to the two exponents in the power branched quasisymmetry definition). Thus we may assume that this is not the case.

Let $a\in E$ be a point for which $\diam fE \le 3 d(f(a),f(c))$; such point exists by Lemma \ref{lemma:continuum-spheres} applied to $fE$. Since conditions of Proposition \ref{prop:power-a} do not hold, we conclude that $\diam E'\ge 3d(a,c)$. Thus, by Lemma \ref{lemma:continuum-spheres}, we may fix a point $b\in E'$ for which $3d(b,c) \ge \diam E' \ge 3d(a,c)$, so $d(b,c)\ge d(a,c)$.  As the conditions of Proposition \ref{prop:power-b} do not hold, we have $\diam E \geq 3 d(b,c) \geq \diam E'$.  Thus,
\[
\frac{\diam E}{\diam E'} \geq 1
\]
and so it suffices to show that there is a constant $C>0$, depending only on $\eta$ and $\lambda$, for which 
\[
\diam fE \le C \diam fE'. 
\]

Let $B=B(c,d(b,c))$ and let $E_B$ be a $\lambda$-diametric hull of $B$ in $X$. Since $c\in E_B \cap E'$ and 
\[
\frac{\diam E_B}{\diam E'} \le \frac{2\lambda d(b,c)}{d(b,c)} = 2\lambda,
\]
we have that
\[
\diam fE_B \le \eta(2\lambda) \diam fE'.
\]

Since $d(a,c) \le d(b,c)$, we have that $a\in E_B$. Thus
\[
\diam fE \le 3 d(f(a),f(c)) \le 3\diam fE_B \le 3\eta(2\lambda) \diam fE'.
\]
The claim follows.
\end{proof}

\begin{remark}
\label{rmk:local-BQS to pBQS}
Observe that the discussion in this section holds for a local branched quasisymmetry $f\colon X\to Y$ if the distances are below a scale depending on $f$ and bounded turning constant $\lambda$ of $X$. Thus, we observe that a local branched quasisymmetry $X\to Y$ is a power branched quasisymmetry, quantitatively.
\end{remark}

\section{Koebe distortion theorem}
\label{sec:Koebe}

In this section, we prove the following version of the Koebe distortion theorem for branched quasisymmetries.

\begin{thm}%[Koebe Distortion for BQS maps]
\label{thm:Koebe}
Let $X$ and $Y$ be compact $\lambda$-bounded turning metric spaces, and let $f \colon X \to Y$ be a discrete and open $\eta$-branched quasisymmetry. Let $B = B(x, r) \subseteq X$ be a ball. Then there is a constant $c_0>0$, depending only on $\eta \colon [0,\infty) \to [0,\infty)$ and $\lambda>0$, for which 
\[
B(f(x), c_0 \diam(fB)) \subseteq fB.
\]  
\end{thm}

\begin{proof}
Let $c = 1/(2\lambda)$.  We first show that 
\begin{equation}
\label{eq:Koebe}
\diam(fB) \leq 3\eta(2\lambda^2) \diam(fE_{cB}).
\end{equation}
Let $y \in fB$ be such that $\diam(fB) \leq 3d(f(x), y)$ and fix $z \in B\cap f^{-1}(y)$.  Let $F$ be a continuum connecting $x$ and $z$ such that $\diam(F) \leq \lambda d(x, z)$.  Now, $fF$ is a continuum connecting $f(x)$ and $f(z) = y$. Moreover,
\[
\frac{\diam(F)}{\diam(E_{cB})} \le \frac{\lambda r}{r/(2\lambda)} = 2 \lambda^2.
\]
Since $x \in F \cap E_{cB}$, we have that
\begin{align*}
\diam (fB) \leq 3\diam(fF)
\leq 3\eta \biggl( \frac{\diam(F)}{\diam(E_{cB})}\biggr) \diam(f E_{cB}) 
\leq 3\eta(2\lambda^2) \diam(fE_{cB}).
\end{align*}
This proves \eqref{eq:Koebe}.

Let now $y' \in Y \setminus fB$. It suffices to show that $d(f(x),y') \ge c_0\diam(fB)$, where $c_0 = 1 / (7\eta(2\lambda^2) \eta(2) \lambda)$.
Let $G \subseteq Y$ be a continuum connecting $f(x)$ and $y'$ and satisfying $\diam(G) \leq \lambda d(f(x), y')$.  Let $\theta = \diam(G)/2\lambda$. By Corollary \ref{cor:cont lift} there exists a continuum $E\subset X$ for which $x \in E$, $E\not \subseteq B$, and $fE \subseteq E(G,\theta)$. In particular, $\diam(E) \geq r$ and
\begin{equation*}
\diam(fE) \leq \diam(E(G,\theta)) \leq 2 \diam(G).
\end{equation*}
Since $x \in E_{cB} \cap E$, we have that 
\begin{equation*}
\frac{\diam(f E_{cB})}{\diam(G)} \leq \frac{2\diam(f E_{cB})}{\diam(fE)} \leq 2 \eta \biggl(\frac{\diam(E_{cB})}{\diam(E)}\biggr) \leq 2\eta(2).
\end{equation*}
Thus
\begin{equation*}
\diam(fE_{cB}) \leq 2\eta(2) \lambda d(f(x), y').
\end{equation*}
Combining this with \eqref{eq:Koebe}, we obtain that
\begin{equation*}
\diam(fB) \leq 6\eta(2\lambda^2) \eta(2) \lambda d(f(x), y').
\end{equation*}
Thus, if $c_0 = 1 / (7\eta(2\lambda^2) \eta(2) \lambda)$, we have that $B(f(x), c_0 \diam(fB)) \subseteq fB$. The claim is proven.
\end{proof}

%%%%%%%%%%%%%%%%%%%%%%%%%%%%%%%%%%%%%%%%%%%%%%%%%%%%%%%%%%%%%%%%%%%%%%%%%%%%%%%
%%%%%%%%%%%%%%%%%%%%%%%%%%%%%%%%%%%%%%%%%%%%%%%%%%%%%%%%%%%%%%%%%%%%%%%%%%%%%%%

%%%%%%%%%%%%%%%%%%%%%%%%%%%%%%%%%%%%%%%%%%%%%%%%%%%%%%%%%%%%%%%%%%%%%%%%%%%%%%%
%%%%%%%%%%%%%%%%%%%%%%%%%%%%%%%%%%%%%%%%%%%%%%%%%%%%%%%%%%%%%%%%%%%%%%%%%%%%%%%

\part{Preliminaries on hyperbolic fillings and vertical quasi-isometries}
\label{part:hyperbolic}

\section{Metric graphs and Gromov hyperbolicity}
\label{sec:MG}

In what follows, a metric graph $(\Gamma, \rho_\Gamma)$ is a graph $\Gamma$ with the natural graph metric $\rho_\Gamma$ giving adjacent vertices of $\Gamma$ distance $1$. Vertices $v$ and $w$ of $\Gamma$ are \emph{adjacent} if $\{v,w\}$ is an edge of $\Gamma$. To simplify notation, we also denote $|v-w|$ the distance $\rho_\Gamma(v,w)$ of points $v,w\in \Gamma$. 

A map $\gamma \colon \N_0 \to \Gamma$ is a \emph{discrete path} if, for each $n\in \N_0$, points $\gamma(n)$ and $\gamma(n+1)$ are adjacent in $\Gamma$. A finite discrete path $\gamma \colon \{0,\ldots, m\} \to \Gamma$ for $m\in \N_0$ is defined similarly. In both cases, we denote $|\gamma|$ the image of the path $\gamma$. 

A discrete path $\gamma \colon \N_0 \to \Gamma$ is a \emph{geodesic ray} if, for all $n,m\in \N_0$, we have $|\gamma(n)-\gamma(m)| = |n-m|$. A \emph{(finite) geodesic $\gamma \colon \{0,\ldots, m\} \to \Gamma$}, where $m\in \N_0$, is defined similarly. Two geodesic rays $\gamma \colon \N_0 \to \Gamma$ and $\gamma' \colon \N_0 \to \Gamma$ are said to be equivalent, denoted $\gamma \sim \gamma'$, if the function $n\mapsto |\gamma(n)-\gamma'(n)|$ is bounded.

\subsection{Gromov hyperbolicity and Gromov boundary}

We introduce now very briefly Gromov hyperbolicity and Gromov boundary; see e.g.~Bonk and Schramm \cite{Bonk-Schramm} or Ghys and de la Harpe \cite{GH} for detailed discussion.

The \emph{Gromov product $(\cdot|\cdot)_o \colon \Gamma \times \Gamma \to \R$ in $\Gamma$ for the base point $o\in X$} is the function
\[
(v|w)_o = \frac{1}{2}\left( |v-o|+|w-o| - |v-w| \right).
\]
Graph $\Gamma$ is \emph{Gromov $\delta$-hyperbolic for $\delta\ge 0$} if, for all triples of points $v,w,u\in \Gamma$, the inequality
\[
(v|w)_o \ge \min\{ (v|u)_o, (u|w)_o\} - \delta
\]
holds. 

The Gromov boundary $\partial \Gamma$ of a $\delta$-hyperbolic graph $\Gamma$ is
 defined as follows. Let $o\in \Gamma$. We say that a map $\gamma \colon \N_0 \to \Gamma$ in $\Gamma$ \emph{tends to infinity} if 
\[
(\gamma(i)|\gamma(j))_o \to \infty
\]
as $i,j\to \infty$. The maps $\gamma \colon \N_0 \to \Gamma$ and $\sigma \colon \N_0 \to \Gamma$ tending to infinity are \emph{equivalent} if 
\[
(\gamma(i)|\sigma(i))_o \to \infty
\]
as $i \to \infty$. It is easy to see that this equivalence is really an equivalence relation and that neither tending to infinity nor equivalence of maps depends on the chosen base point $o\in \Gamma$. 

\begin{remark}
\label{rmk:maps-paths}
Since $\Gamma$ is geodesic, we may replace maps $\N_0 \to \Gamma$ in these definitions by paths $\N_0 \to \Gamma$. Indeed, for each map $\gamma \colon \N_0 \to \Gamma$ there exists an increasing function $\iota \colon \N_0 \to \N_0$ and a path $\bar \gamma \colon \N_0 \to \Gamma$ having the property that $\bar \gamma(\iota(j)) = \gamma(j)$ for each $j\in \N_0$ and that $\bar \gamma|_{\{\iota(j),\ldots, \iota(j+1)\}} \colon \{\iota(j),\ldots \iota(j+1)\} \to \Gamma$ is a geodesic from $\gamma(j)$ to $\gamma(j+1)$. It is now easy to check that $\bar \gamma$ tends to infinity and that $\gamma$ and $\bar \gamma$ are equivalent.
\end{remark}

The Gromov boundary $\partial \Gamma$ of $\Gamma$ is the space of equivalence classes of paths $\N_0 \to \Gamma$ tending to infinity; note that we use here Remark \ref{rmk:maps-paths}. The Gromov product $(\cdot,\cdot)_o$ extends to the boundary $(\cdot,\cdot)_o \colon \partial \Gamma \times \partial \Gamma \to \R$ by formula
\[
(\xi,\zeta)_o = \sup \{ \liminf_{j \to \infty}(\gamma(j),\sigma(j))_o \colon \gamma \in \xi, \sigma \in \zeta\}
\]
for $\xi,\zeta\in \partial \Gamma$. 

The Gromov boundary $\partial \Gamma$ carries a class of metrics associated to the Gromov product. More precisely, there exists an absolute constant $\varepsilon_0>0$ having the property that, for each $\delta$-hyperbolic space $X$, base point $o\in X$, and $\varepsilon < \varepsilon_0/\delta$, the Gromov boundary $\partial X$ admits a metric $d_{o,\varepsilon} \colon \partial \Gamma\times \partial \Gamma \to [0,\infty)$ satisfying
\[
\frac{1}{2}e^{-\varepsilon (\xi,\zeta)_o} \le d_{o,\varepsilon}(\xi, \zeta) \le e^{-\varepsilon (\xi,\zeta)_o}
\]
for all $\xi,\zeta \in \partial \Gamma$. The metrics in this class, called \emph{visual metrics of $\partial \Gamma$}, are quasisymmetrically equivalent. We refer again to Bonk and Schramm \cite[Section 6]{Bonk-Schramm} or Ghys and de la Harpe \cite{GH} for detailed discussion.

\section{Hyperbolic Fillings}\label{Hyp Fill sec}

In this section we discuss hyperbolic fillings of compact metric spaces. We refer to \cite{L} for a more detailed discussion. We begin with auxiliary definitions related to coverings. 

\begin{defn}
Let $Z$ be a compact metric space and $P$ an $\varepsilon$-net in $Z$ for $\varepsilon>0$. A covering $\mathscr U = \{ B(z,2\varepsilon) \colon z\in P\}$ of $Z$ is an \emph{$\varepsilon$-covering (associated to $P$)}. 
\end{defn}

For an $\varepsilon$-covering $\mathscr U$ associated to an $\varepsilon$-net $P$ of $Z$, we also denote $c_{\mathscr U} \colon \mathscr U \to P$ the function having the property that, for each $U=B(z,2\varepsilon)\in \mathscr U$, $c_{\mathscr U}(U)=z\in P$. 
%$ is the element of $P$ in $U$. 

\begin{defn}
A sequence $(\mathscr U_n)_n$ of coverings of $Z$ is an \emph{$s$-sequence (for $s>1$)} if, for each $n\in \Z$, the covering $\mathscr U_n$ is an $s^{-n}$-covering. 
\end{defn}

For the definition of a hyperbolic filling, we define an associated class of graphs. A graph $\Gamma(\mathcal U)$ is the \emph{incidence graph of a sequence $\mathcal U = (\mathscr U_n)_n$ of coverings of $Z$} if $\Gamma(\mathcal U)$ has the disjoint union $\bigsqcup_{n\in \N_0} \mathscr U_n$ as its vertex set and vertices $U\in \mathscr U_n$ and $V\in \mathscr U_m$ of $\Gamma(\mathcal U)$ are joined by an edge in $\Gamma(\mathcal U)$ if and only if $U\cap V\ne \emptyset$ and $|m-n|\le 1$.

\begin{remark}
Note that, if $Z$ is metrically doubling, then $\Gamma(\mathcal U)$ is a graph of bounded degree.
\end{remark}

The incidence graph $\Gamma(\mathcal U)$ carries two natural functions. The \emph{level function $\ell \colon \Gamma(\mathcal U)\to \Z$} is the unique function satisfying $v\in \mathscr U_{\ell(v)}$ for each $v\in \hat Z$. The \emph{center function $c \colon \Gamma(\mathcal U)\to Z$} is the function satisfying $v\mapsto c_{\mathscr U_{\ell(v)}}(v)$. Note that, for every $v\in \Gamma(\mathcal U)$, we have $v = B(c(v), 2s^{-\ell(v)})$.

\bigskip

Let $Z$ be a compact metric space and $\mathcal U = (\mathscr U_n)_n$ and $s$-sequence. Since $Z$ has finite diameter, there exists a maximal index $n_Z\in \Z$ for which $\mathscr U_{n_Z} = \{Z\}$. We call the corresponding vertex $O_{\hat Z}\in \Gamma$ the \emph{root of $\Gamma(\mathcal U)$} and call the the subgraph $\Gamma_+(\mathcal U)$ of $\Gamma(\mathcal U)$, obtained by removing the vertices associated to coverings $\bigsqcup_{n<n_Z} \mathscr U_n$, the \emph{root pruned subgraph of $\Gamma$}.

\begin{defn}
Let $Z$ be a compact metric space. A metric graph $(\widehat Z, \rho_{\widehat Z})$ is a \emph{hyperbolic $s$-filling of $Z$ for $s>1$} if $\hat Z$ is the root pruned incidence graph $\Gamma_+(\mathcal U)$ of an $s$-sequence $\mathcal U$ of coverings of $Z$. A metric graph $\hat Z$ is a hyperbolic filling of $Z$ if $\widehat Z$ is a hyperbolic $s$-filling of $Z$ for some $s>1$. 
\end{defn}

For each $s>0$, we denote $\HF_s(Z)$ the \emph{set of hyperbolic $s$-fillings of $Z$} and 
\[
\HF(Z) = \bigcup_{s>0} \HF_s(Z)
\]
the \emph{set of hyperbolic fillings of $Z$}.

Recall that a vertex in $\widehat Z$ is a ball in $Z$. We emphasize this with notation $B_v = v\subset Z$ for $v\in \widehat Z$.

\begin{remark}
In what follows, we may assume, by rescaling, that the space $Z$ has diameter $1$. Thus the root vertex $O_Z$ is the unique vertex of level $0$. Thus it suffices to consider the subgraph of $\widehat Z$ consisting of vertices on non-negative levels. We follow this convention in what follows.
\end{remark}

\begin{remark}
In what follows, we tacitly assume that the $s$-sequence $(\mathscr U_n)_n$ of coverings of $Z$ is given, and merely use the notation $\widehat Z$ for a hyperbolic filling in question. In this case, we denote $\ast = O_Z$ the root of $\widehat Z$, and denote $\ell = \ell_{\widehat Z} \colon \widehat Z\to \Z$ and $c = c_{\widehat Z} \colon \widehat Z\to Z$ the \emph{level and center functions} associated to $\widehat Z$.
\end{remark}

\subsection{Hyperbolicity of hyperbolic fillings}

The coarse geometry of a hyperbolic filling largely stems from the following observation on Gromov products. We refer to \cite{L} for a proof.

\begin{lemma}[{\cite[Lemma 3.1]{L}}]
\label{lemma:Gromov-product}
Let $Z$ be a compact space, and $\widehat Z\in \HF_s(Z)$. Then, for $v,w\in \widehat Z$, we have
\[
s^{-(v,w)} \simeq \diam(B_v \cup B_w),
\]
where the constants depend only on $s$.
\end{lemma}

Essentially from this lemma we obtain that hyperbolic fillings are Gromov hyperbolic spaces, quantitatively. This justifies the used terminology.
\begin{lemma}[{\cite[Lemma 3.2]{L}}]
Let $Z$ be a compact metric space. Then a hyperbolic filling $\widehat Z\in \HF_s(Z)$ of $Z$ is a (visual) Gromov hyperbolic $\delta$-space for $\delta$ depending only on $s$.
\end{lemma}

Lemma \ref{lemma:Gromov-product} also yields that the Gromov boundary of a hyperbolic filling is quasisymmetric to original space. More precisely, we have the following.

\begin{lemma}[{\cite[Lemma 3.3]{L}}]
\label{lemma:filling_boundary}
Let $Z$ be a compact metric space and $\widehat Z \in \HF_s(Z)$ a hyperbolic filling of $Z$. If $\partial \widehat Z$ is given a visual metric, then the map $\partial \widehat Z \to Z$, $[\gamma] \mapsto \lim_{n\to \infty} c_Z(\gamma(n))$, is a well-defined $\eta$-quasisymmetric homeomorphism, where the homeomorphism $\eta$ depends only on the data associated to space $Z$ and $s$.
\end{lemma}

\begin{remark}
The reader might have by now noticed that the spaces discussed in \cite{L} are Ahlfors regular. This assumption is, however, not necessary for these particular results.
\end{remark}

\section{Elementary geometry of hyperbolic fillings}

Before discussing the properties of vertical geodesics and quasigeodesics, we record some elementary properties of vertex balls $B_v$ in $\hat Z$. In what follows, we assume that $\diam Z=1$; in particular $\ast = O_Z = Z$.

\begin{remark}
Recall that, if $Z$ is connected, it is a continuum and we have, by Lemma \ref{lemma:continuum-spheres}, that
\[
\diam B_v = \diam B(c_{\widehat Z}(v), 2s^{-\ell_{\widehat{Z}}(v)}) \ge s^{-\ell_{\widehat{Z}}(v)}.
\]
for each $v\in \hat Z$.
\end{remark}

Our first observation on the filling $\widehat Z$ is that a diameter of a set in $Z$ is comparable to the diameter of a smallest vertex ball containing it. We record this observation as follows.

\begin{lemma}\label{vertex comp}
Let $Z$ be a compact metric space, $s>1$, and $\widehat Z\in \HF_s(Z)$. Let also $E\subset Z$ and let $v\in \widehat Z$ be a vertex of maximal level satisfying $E \subset B_v$. Then
\[
s^{-\ell_{\widehat Z}(v)-1} \le \diam E \le 4s^{-\ell_{\widehat Z}(v)}.
\]
\end{lemma}
\begin{proof}
The upper bound is trivial as $\diam(E) \leq \diam(B_v) \leq 4s^{-\ell_{\widehat Z}(v)}$.

For the lower bound, suppose $\diam E < s^{-\ell_{\widehat Z}(v)-1}$ and let $n=\ell_{\widehat Z}(v)$. Since centers $c(w)$ of vertices $w\in \hat Z$ on level $(n+1)$ form an $s^{-(n+1)}$-net, we conclude that there exists a vertex $w\in \hat Z$ of level $(n+1)$ satisfying $\dist(E,c_{\widehat Z}(w))<s^{-(n+1)}$. Thus $E \subset B(c_{\widehat Z}(w), 2s^{-(n+1)}) = B_w$. This contradicts the maximality of the level of $v$. Thus $\diam E \ge s^{-\ell_{\widehat Z}(v)-1}$.
\end{proof}

The second estimate is a common ancestor lemma for vertices in the hyperbolic filling. We formulate this as follows.
\begin{lemma}\label{struct lemma}
Let $Z$ be a compact metric space and $\widehat Z  \in \HF_s(Z)$.  Let $v, v' \in \widehat Z$ and $i = \min\{\ell_{\widehat Z}(v), \ell_{\widehat Z}(v')\}$.  Suppose $p\in \N_0$ satisfies $d(c_{\widehat Z}(v), c_{\widehat Z}(v')) < s^{-(i-p)}$.  Then there exists a vertex $w \in \hat Z$ of level $\ell_{\widehat Z}(w) = \max(i - p, 0)$ for which $c_{\widehat Z}(v), c_{\widehat Z}(v') \in B_w$.  
\end{lemma}

\begin{proof}
We may assume that $\ell(v) \leq \ell(v')$, so $i = \ell(v)$.  If $i - p \leq 0$, we take $w = O_{\hat{Z}}$.  Otherwise, let $w\in \hat Z$ be a vertex on level $(i - p)$ satisfying $d(c(v), c(w)) \leq s^{-(i-p)}$.  Then,
\begin{equation*}
d(c(v'), c(w)) \leq d(c(v'), c(v)) + d(c(v), c(w)) < 2s^{-(i-p)}
\end{equation*}
so $c(v') \in B_w$.  This proves the claim.
\end{proof}

As usual, we obtain as a direct corollary a distance estimate for the vertices $v$ and $v'$ in $\widehat Z$ in terms of their centers in $Z$. 

\begin{cor}
\label{cor:struct lemma}
Let $Z$ be a compact metric space and $\widehat Z  \in \HF_s(Z)$.  Let $v, v' \in \widehat Z$ and $i = \min\{\ell_{\widehat Z}(v), \ell_{\widehat Z}(v')\}$.  Suppose $p\in \N_0$ satisfies $d(c_{\widehat Z}(v), c_{\widehat Z}(v')) < s^{-(i-p)}$.  Then
\begin{equation*}
|v - v'| \leq 2p + |\ell(v) - \ell(v')|.
\end{equation*}
\end{cor}
\begin{proof}
Let $w\in \widehat Z$ be vertex as in Lemma \ref{struct lemma}. We may again assume that $\ell(v) \le \ell(v')$. The vertices $v$ and $w$ as well as vertices $v'$ and $w$ are connected by geodesics of lengths at most $p$ and $p+(\ell(v')-\ell(v))$, respectively. The claim follows.
\end{proof}

The next elementary estimate is converse to the previous ancestor lemma. We estimate the distance of centers of vertex balls in terms of the distance in the hyperbolic filling. For the statement, and for forthcoming sections, we introduce now the universal structure constant
\begin{equation}
\label{eq:As}
A_s = 8\frac{s}{s-1}
\end{equation}
for hyperbolic fillings with scaling constant $s>1$.

\begin{lemma}\label{basic filling lemma 1}
Let $(Z,d)$ be a compact metric space and $\widehat{Z} \in \HF_s(Z)$. Then, for $v,w\in \widehat Z$, we have
\[
d(c_{\widehat Z}(v),c_{\widehat Z}(w)) \le A_s s^{|v-w| - \max\{\ell_{\widehat Z}(v),\ell_{\widehat Z}(w)\}}.
\]
\end{lemma}

\begin{remark}
The estimate in Lemma \ref{basic filling lemma 1} is equivalent to
\[
d(c_{\widehat Z}(v),c_{\widehat Z}(w)) s^{\max\{\ell_{\widehat Z}(v),\ell_{\widehat Z}(w)\}} \le A_s s^{|v-w|}.
\]
Note, however, that there is no converse estimate, since we may have $c_{\widehat Z}(v)=c_{\widehat Z}(w)$ for distinct vertices $v$ and $w$ in $\hat Z$.
\end{remark}

\begin{proof}[Proof of Lemma \ref{basic filling lemma 1}]
We may assume that $\ell_{\widehat Z}(v) \le \ell_{\widehat Z}(w)$. Let $n=|v-w|$ and let $w = u_0,u_1,\ldots, u_n = v$ be a chain of adjacent vertices in $\widehat Z$ connecting $w$ and $v$. Then, for each $i=0,\ldots, n-1$, balls $u_i=B(c(u_i), 2s^{-\ell_{\widehat Z}(u_i)})$ and $u_{i+1}=B(c_{\widehat Z}(u_{i+1}), 2s^{-\ell_{\widehat Z}(u_{i+1})})$ of $Z$ intersect and $\min\{\ell_{\widehat Z}(u_i), \ell_{\widehat Z}(u_{i+1})\} \ge \ell_{\widehat Z}(w) - (i+1)$. Thus 
\[
d(c_{\widehat Z}(v_i),c_{\widehat Z}(v_{i+1})) \le 8 s^{-\ell_{\widehat Z}(w)+(i+1)}
\]
for each $i=0,\ldots, n-1$. Hence
\begin{align*}
d(c_{\widehat Z}(v),c_{\widehat Z}(w)) &\le \sum_{i=0}^{n-1} 8 s^{-\ell_{\widehat Z}(w) +(i + 1)} = 8 \frac{s^n-1}{s-1} s^{-\ell_{\widehat Z}(w)+1} \\
&\le 8 \frac{s}{s-1} s^{n-\ell_{\widehat Z}(w)}
= A_s s^{|v-w| - \max\{\ell_{\widehat Z}(v),\ell_{\widehat Z}(w)\}}.
\end{align*}
The claim follows.
\end{proof}

\section{Vertical geodesics}

A hyperbolic filling $\widehat Z$ has a distinguished class of pointed geodesics $(\N_0,0) \to (\widehat Z,\ast)$, which we call \emph{vertical geodesics of $\widehat Z$}. The term vertical comes from the observation that a geodesic $\gamma \colon (\N_0,0) \to (\widehat Z,\ast)$ satisfies $\ell(\gamma(n)) = n$ for each $n \in \N_0$. We denote by $\Gamma(\widehat Z,\ast)$ the family of all vertical geodesics $(\N_0,0) \to (\widehat Z, \ast)$.

A vertical geodesic $\gamma\colon (\N_0,0) \to (\widehat Z,\ast)$ \emph{visits $v\in \widehat Z$} if $v\in |\gamma|$, that is, $\gamma(\ell(v))=v$. Further, we say that a vertical $\gamma$ is \emph{centered at $z\in Z$} if $d(z,c_{\widehat X}(\gamma(k))) < s^{-k}$ for all $k\in \N_0$. For each $z\in Z$, there exists a vertical geodesic centered at $z$. Indeed, it suffices to fix for each $n\in \N_0$ a vertex $v_n \in \widehat Z$ on level $n$ for which $c_{\widehat X}(v_n)$ has the smallest distance to $z$. More precisely, we have the following lemma.

\begin{lemma}
\label{lemma:centered-geodesics}
Let $\widehat Z$ be a hyperbolic filling of a compact metric space $Z$, $z\in Z$, and $v\in \widehat Z$ for which $z\in B_v$. Then there exists a vertical geodesic $\gamma \colon (\N_0,0) \to (\widehat Z,\ast)$ at $z$ visiting $v$. Moreover, if $d(z,c_{\widehat X}(v))<s^{-\ell(v)}$, then there exists a centered geodesic at $z$ visiting $v$.
\end{lemma}
\begin{proof}

Let $\Gamma \subset \widehat Z$ be the maximal subgraph of $\widehat Z$ consisting of those vertices $v'\in \widehat Z$ for which $\ell(v') > \ell(v)$ and $B_v \cap B_{v'}\ne \emptyset$. Then $\Gamma$ is a metric graph with root $v$ and level function $\ell' \colon \Gamma \to \N_0$ given by $v' \mapsto \ell(v')-\ell(v)$. 

For each $n\in \N_0$, we fix a vertex $v_n \in \Gamma$ for which $\ell'(v_n) = n$ and $z\in B_{v_n}$. Let now $\sigma \colon (\N_0,0) \to (\Gamma,v)$ be a map $n \mapsto v_n$. Then $\sigma$ is a discrete path satisfying $|\ell(\sigma(n))-\ell(\sigma(n'))| = |n - n'|$ for all $n,n'\in \N_0$. Thus $\sigma$ is a geodesic.

Let now $\sigma' \colon \{0,\ldots, \ell(v)\} \to \widehat Z$ be a geodesic from $\ast$ to $v$ in $\widehat Z$. Then the map $\gamma \colon (\N_0,0) \to (\widehat Z,\ast)$ defined by $\gamma(k) = \sigma'(k)$ for $k\le \ell(v)$ and $\sigma(\ell_{\widehat Z}(v)+k) = \sigma(k)$ for each $k\in \N_0$ is a vertical geodesic centered at $z$ visiting $v$. This concludes the proof in the first case. The other case is similar.
\end{proof}

\subsection{Convergence of vertical geodesics}

We record now two lemmas on properties of vertical geodesics; in both lemmas, we assume the normalization $\diam Z = 1$ of the diameter of $Z$. The first lemma gives an estimate for the containment of vertex balls along a vertical geodesic.
In the statement, $A_s$ is the structural constant in \eqref{eq:As}.

\begin{lemma}\label{basic filling lemma 2}
Let $Z$ be a continuum, $\widehat{Z} \in \HF_s(Z)$, and let $\gamma \colon (\N_0,0) \to (\widehat Z,\ast)$ be a vertical geodesic.  Then, for each $n_0 \in \N_0$, we have 
\[
B_{\gamma(n_0)} \subset \bigcup_{n \geq n_0} B_{\gamma(j)} \subset (A_s + 1)B_{\gamma(n_0)}.
\]
\end{lemma}

\begin{proof}
Since $\gamma$ is a vertical geodesic, we have $\ell(\gamma(j))=j$ for each $j\in \N_0$. Furthermore, since $B_{\gamma(j)}\cap B_{\gamma(j+1)} \ne \emptyset$ for each $j\in \N_0$, we have that
\[
d(c(v_j),c(j+1)) \le \diam B_{\gamma(j)} + \diam B_{\gamma(j+1)} \le 4 s^{-j} + 4s^{-j-1} \le 8 s^{-j}.
\]
Thus, for $n\ge n_0$, we have
\[
d(c(v_n), c(v_{n_0})) \le \sum_{j=n_0}^{n-1} 8 s^{-j} \le 8 \frac{s}{s-1} s^{-n_0} = A_s s^{-n_0} \le A_s \diam B_{\gamma(n_0)}.
\]
Since $\diam B_{\gamma(n)} \le \diam B_{\gamma(n_0)}$ for each $n\ge n_0$, we conclude that
\[
B_{\gamma(n_0)} \subset \bigcup_{n\ge n_0} B_{\gamma(n)} \subset (A_s + 1) B_{\gamma(n_0)}.
\]
The claim follows.
\end{proof}

\begin{cor}
\label{cor:convergence}
Let $Z$ be a continuum, $\widehat{Z} \in \HF_s(Z)$, and let $\gamma \colon (\N_0,0) \to (\widehat Z,\ast)$ be a vertical geodesic. Then there exists $z\in Z$ for which $c_{\widehat Z}(\gamma(n)) \to z$ as $n\to \infty$ and   
%Then  
\[
d(z, c_{\widehat Z}(\gamma(n))) \leq A_s s^{-n}.
\]
for each $n\in \N_0$. Moreover, for a vertical geodesic $\gamma \colon (\N_0,0) \to (\widehat Z,\ast)$ centered at $z\in Z$, we have that $c_{\widehat Z}(\gamma(n)) \to z$ as $n\to \infty$.
\end{cor}
\begin{proof}
For each $n\in \N$, let 
\[
E_n = \overline{\bigcup_{j \ge n} B_{\gamma(j)}}.
\]
Then $(E_n)$ is a decreasing sequence of compact sets whose intersection $E=\bigcap_{n\in \N_0} E_n$ is a point $z\in Z$. Indeed, for each $n\in \N_0$, we clearly have $E_n \supset E_{n+1}$ and, by Lemma \ref{basic filling lemma 2}, we have that $\diam E_n \le (A_s+1)\diam B_n = 2(A_s +1)s^{-n} \to 0$.  From the estimates in Lemma \ref{basic filling lemma 2}, we see that $d(z, c_{\widehat{Z}}(\gamma(n))) \leq A_s s^{-n}$.  The first claim follows. For the second claim is immediate as $z$ is unique.
\end{proof}

Having Corollary \ref{cor:convergence} at our disposal, we say that a vertical geodesic $\gamma \colon (\N_0,0) \to (\widehat Z,\ast)$ \emph{tends to $z\in Z$} if $c_{\widehat Z}(\gamma(n)) \to z$ as $n\to \infty$. For $z\in Z$, we denote $\Gamma_x(\widehat Z,\ast)$ the family of all vertical geodesics in $\widehat Z$ tending to $z$.

\subsection{Shadows}

Based on Corollary \ref{cor:convergence}, we also define the \emph{shadow of $z\in Z$ in $\widehat Z$ on level $n\in \N_0$} by 
\begin{equation}
\label{eq:shadow_of_z}
S_{\widehat Z}(z;n) = \left\{w \in \widehat Z \colon \ell_{\widehat Z}(w)=n,\ c_{\widehat Z}(w) \in \bar B(z, 2(A_s + 1)s^{-n})\right\}.
\end{equation}
We immediately observe that each vertical geodesic $\gamma \colon (\N_0,0) \to (\widehat Z,\ast)$ tending to $z$ satisfies $\gamma(n) \in S_{\widehat Z}(z,n)$ for each $n\in \N_0$.

For metrically doubling spaces, we have a uniform estimate for the size of the shadow.

\begin{lemma}\label{basic filling lemma 3}
Let $Z$ be a metrically doubling continuum and $\widehat{Z} \in HF_s(Z)$. Then there exists a constant $m\ge 1$ depending only on $s$ and the doubling constant of $Z$ for which 
\[
\# S_{\widehat Z}(z;n) \le m
\]
for each $z\in Z$ and $n\in \N_0$.
\end{lemma}

\begin{proof}
Let $n\in \N_0$. The balls $B(c_{\widehat Z}(w),s^{-n}/3)$ for $w\in S_{\widehat Z}(z,n)$ are mutually disjoint and contained in the ball $B(z, 2(A_s +2) s^{-n})$. The claim follows now from the doubling of $Z$.
\end{proof}

\subsection{Separation of geodesics}

As the last elementary lemma on vertical geodesics we prove a separation lemma for geodesics which meet at a given level. We state this more precisely as follows.

\begin{lemma}\label{geod split lemma}
Let $Z$ be a compact metric space and $\widehat{Z} \in HF_s(Z)$. Let $E \subseteq Z$ be a continuum and let $v \in \widehat Z$ be a vertex of highest level satisfying $E \subseteq B_v$. Then, for each $J\in \N$, there exists a constant $k_0\in \N_0$, depending only on $J$ and $s$, and there exist vertical geodesics $\gamma_1, \dots, \gamma_J \colon (\N_0,0) \to (\widehat Z,\ast)$ centered at points $z_1,\ldots, z_J \in E$, respectively, all visiting vertex $v$, which have the following property: 
\begin{enumerate}
\item For $k \geq k_0$, if $i \neq j$ then $\gamma_i(\ell(v) + k) \neq \gamma_j(\ell(v) + k)$.\label{item:gpl3} 
\end{enumerate}
\end{lemma}

\begin{proof}
Let $J\in \N_0$. For each $n \in \N$, let $W_n$ be a maximal $s^{-n}$ separated subset of $E$. By Lemma \ref{vertex comp}, we have that $\diam E \ge s^{-\ell(v)-1}$. Since $E$ is connected, we also have that $\diam(E) \leq 2s^{-n} |W_n|$. We fix now the smallest $n \in \N$ for which $s^n s^{-\ell(v)-1} > 2J$. Since $\diam E \ge s^{-\ell(v)-1}$ by Lemma \ref{vertex comp}, we have that $|W_n| > J$.  

Let $z_1, \dots, z_J \in W_n$. By Lemma \ref{lemma:centered-geodesics}, we may fix, for $1 \leq j \leq J$, a vertical geodesic $\gamma_j \colon \N_0 \to \widehat Z$ centered at $z_j$ visiting $v$, that is, satisfying $\gamma_j(\ell(v)) = v$. This is possible as $E \subseteq B_v$.

It remains to prove \eqref{item:gpl3}.  Let $m_0\in \N_0$ be such that $4s^{-m_0} < 1$ and let $m > m_0$.  Since $z_j \in B_{\gamma_j(n + m)}$, we have $d(z_j, c_{\widehat Z}(\gamma_j(n + m))) < 2s^{-(n + m)} < s^{-n}/2$.  Let $i \neq j$. Then $d(z_j, z_i) \geq s^{-n}$ and
\begin{equation*}
\begin{split}
d(z_i, c_{\widehat Z}(\gamma_j(n + m))) &\geq d(z_j, z_i) - d(z_j, c_{\widehat Z}(\gamma_j(n + m))) > s^{-n} - s^{-n}/2 = s^{-n}/2.
\end{split}
\end{equation*}
Hence, $\gamma_j(n + m) \neq \gamma_i(n + m)$.  It remains to relate $n + m_0$ with $\ell(v)$.  As we chose the smallest $n\in \N_0$ satisfying $s^n s^{-\ell(v)-1} > 2J$, we have that $n \leq \log_s(2J) + \ell(v) + 2$.  Let $k_0 \in \N$ be the smallest integer such that $k_0 \geq \log_s(2J) + 2 + m_0$.  Then, for $k \geq k_0$ and $i \neq j$, we have that
\begin{equation*}
\ell(v) + k \geq n - \log_s(2J) - 2 + k_0 \geq n + m_0
\end{equation*}
and so $\gamma_i(\ell(v) + k) \neq \gamma_j(\ell(v) + k)$.
\end{proof}

\section{Vertical quasigeodesics}
\label{sec:VQG}

Recall that a map $\gamma \colon \N_0 \to X$ into a metric space $(X,d)$ is a \emph{$(\alpha,\beta)$-quasigeodesic} if
\[
\frac{1}{\alpha}|m-n| - \beta \le d(\gamma(m),\gamma(n)) \le \alpha |m-n| + \beta
\]
for all $m,n\in \N_0$. Following the terminology for geodesics, we say that a pointed quasigeodesic $(\N_0,0) \to (\widehat Z,\ast)$ into a hyperbolic filling $\widehat Z$ of a space $Z$ is \emph{vertical}.

We record first two versions of a geodesic stability lemma for infinite rays in hyperbolic fillings; see e.g.~Ghys and de la Harpe \cite[Theorem 5.6]{GH} or Bonk and Schramm \cite[Proposition 5.4]{Bonk-Schramm} for a proof.

\begin{lemma}\label{bdd hausdorff vqi rays}
Let $(Z, d)$ be a compact metric space, $s>1$, and $\widehat{Z} \in \HF_s(Z)$. Let also $\gamma \colon (\N_0,0) \to (\widehat{Z},\ast)$ be a vertical $(\alpha, \beta)$-quasigeodesic.  Then there is a constant $H = H(\alpha, \beta, \delta)>0$, where $\delta$ is the Gromov hyperbolicity constant of $\widehat{Z}$, and a vertical geodesic $\gamma' \colon (\N_0,0) \to (\widehat{Z},\ast)$ satisfying $\Hdist (|\gamma|, |\gamma'|) \leq H$.
\end{lemma}

\begin{lemma}\label{quasigeod close lemma}
Let $Z$ be a compact metric space, $s>1$, and $\widehat{Z} \in \HF_s(Z)$. Let $\alpha \ge 1$, $\beta \geq 0$, and $z\in Z$. Then there is a constant $H_1>0$ depending on $\alpha$, $\beta$, and Gromov hyperbolicity constant of $\widehat Z$ for the following: If $\gamma \colon (\N_0,0) \to (\widehat Z,\ast)$ and $\gamma' \colon (\N_0,0) \to (\widehat Z,\ast)$ are $(\alpha, \beta)$-quasigeodesics tending to $z$, then $\Hdist(|\gamma|, |\gamma'|) \leq H_1$.  
\end{lemma}

\begin{remark}
As a direct consequence of Lemma \ref{bdd hausdorff vqi rays}, if $\gamma \colon (\N_0,0) \to (\widehat Z, \ast)$ is a vertical quasigeodesic and $\gamma' \colon (\N_0,0) \to (\widehat Z,\ast)$ is a vertical geodesic as in Lemma \ref{bdd hausdorff vqi rays}, then $\gamma$ and $\gamma'$ tend to the same point on $Z$, that is, $\lim_{n \to \infty}c_{\widehat Z}(\gamma(n)) = \lim_{n \to \infty} c_{\widehat Z}(\gamma(n))\in Z$. 
\end{remark}

\subsection{Shadows revisited}

We now adapt Lemma \ref{basic filling lemma 3} to vertical quasigeodesics. 
For $\alpha\ge 1$ and $\beta>0$, we define an \emph{$(\alpha,\beta)$-shadow of $z\in Z$ in $\widehat Z$ on level $n\in \N_0$} by 
\[
S^{\alpha,\beta}_{\widehat Z}(z;n) = \{v \in \widehat{Z} \colon \ell_{\widehat Z}(v) = n,\ c_{\widehat Z}(v) \in B(z, 3(A_s + 1)s^Hs^{-n})\},
\]
where $H=H(\alpha, \beta, \delta)$ is the constant in Lemma  \ref{bdd hausdorff vqi rays}.

\begin{lemma}\label{qg filling lemma}
Let $Z$ be a compact metric space, $\widehat{Z} \in HF_s(Z)$, and let $\alpha\ge 1$ and $\beta >0$. Then, for $z\in Z$ and a vertical $(\alpha,\beta)$-quasigeodesic $\gamma \colon (\N_0,0) \to (\widehat Z, \ast)$ tending to $z$, we have
\[
\{\gamma(k)\in \widehat Z \colon \ell(\gamma(k))=n\} \subset S^{\alpha,\beta}_{\widehat Z}(z,n)
\]
for each $n\in \N_0$. In particular,
\[
d(z,c(\gamma(k))) \leq 3(A_s + 1)s^H s^{-\ell(\gamma(k))}
\]
for each $k\in \N_0$.
\end{lemma}

\begin{proof}
By Lemma \ref{bdd hausdorff vqi rays}, there exists a vertical geodesic $\gamma' \colon (\N_0,0) \to (\widehat{Z},\ast)$ satisfying $\Hdist (\gamma, \gamma') \leq H$, where $H$ is as in Lemma \ref{bdd hausdorff vqi rays}. Let $n\in \N_0$ and suppose there exists $k\in \N_0$ for which $\ell(\gamma(k)) = n$.  Then there exists $k'\in \N_0$ for which $|\gamma'(k') - \gamma(k)| \leq H$. Thus $\ell(\gamma'(k')) \ge \ell(\gamma(k)) - H = n - H$. By Lemma \ref{basic filling lemma 1}, we further have that 
\[
d(c(\gamma'(k')), c(\gamma(k))) \leq A_s s^{H - n}.
\]
Since $\gamma'$ tends to $z$, we have, by Corollary \ref{cor:convergence}, that
\[
d(z, c(\gamma'(k'))) \leq 2(A_s+1) s^{H - n}.
\]
Thus
\[
d(z,c(\gamma(k))) \le 3(A_s +1)s^H s^{-n}.
\]
The claim follows.
\end{proof}

\begin{remark}\label{vert qg bound}
Lemma \ref{qg filling lemma} yields the following property. Let $\widehat Z \in \HF_s(Z)$ and let $\gamma \colon (\N_0,0)\to (\widehat Z,\ast)$ and $\sigma \colon (\N_0,0) \to (\widehat Z,\ast)$ be a vertical geodesic and a vertical $(\alpha,\beta)$-quasigeodesic  both tending to $z\in Z$. Suppose $v \in |\gamma|$ and $w\in |\sigma|$ have the same level in $\widehat Z$. Then $d(c(v),c(w)) \leq K_1 s^{-\ell(v)}$, where $K_1 = 6(A_s + 1)s^H$.
\end{remark}

\begin{lemma}\label{qg filling lemma 2}
Let $Z$ be a compact, doubling metric space, $\widehat{Z} \in HF_s(Z)$, and let $\alpha\ge 1$ and $\beta \ge 0$. Then there exists a constant $m\ge 1$, depending only on $\alpha$, $\beta, \delta$, $s$, and the doubling constant of $Z$, for which
\[
\# S^{\alpha,\beta}_{\hat Z}(z;n) \le m
\]
for each $n\in \N_0$.
\end{lemma}

\begin{proof}
The estimate on $\# S^{\alpha,\beta}_{\hat Z}(z;j)$ follows as in the proof of Lemma \ref{basic filling lemma 3} with the difference that the constant now also depends on $H$.
\end{proof}

\section{Vertical quasi-isometries between hyperbolic fillings}
\label{sec:VQI}

In this section, we discuss briefly two basic properties of vertical quasi-isometries: the Lipschitz property and stability under composition. We begin with the Lipschitz property.

\begin{lemma}\label{upper qi}
Let $X$ and $Y$ be compact metric spaces, $s>1$, and let $\widehat{X} \in \HF(X)$ and $\widehat{Y} \in \HF(Y)$.  Let $\alpha\ge 1$, $\beta>0$, and let $\varphi \colon (\widehat{X},\ast) \to (\widehat{Y},\ast)$ be an $(\alpha, \beta)$-vertical quasi-isometry.  Then, for all $v, v' \in \widehat{X}$, we have 
\[
|\varphi(v) - \varphi(v')| \leq 2(\alpha+\beta) |v - v'|.
\]
\end{lemma}

\begin{proof}
Since $\widehat X$ is (discretely) geodesic, it suffices to prove the result for adjacent vertices $v$ and $v'$ in $\widehat X$, that is, for vertices satisfying $|v-v'|= 1$. The claim follows then from the triangle inequality.  

Let $v$ and $v'$ be adjacent vertices in $\widehat X$. If $\ell(v) \neq \ell(v')$, then $v$ and $v'$ lie on a vertical geodesic $\gamma \colon (\N_0,0) \to (\widehat X,\ast)$. Since $\varphi$ is an $(\alpha,\beta)$-vertical quasi-isometry, we have that 
\[
|\varphi(v)-\varphi(v')| \le \alpha + \beta = (\alpha+\beta)|v-v'|.
\]

Suppose $\ell(v) = \ell(v')$.  Then, $B_{v} \cap B_{v'} \neq \emptyset$.  Let $z \in B_v \cap B_{v'}$ and let $w\in \widehat X$ be a vertex satisfying $\ell(w) = \ell(v) - 1$ and $z \in B_w$.  Then $|v-w| = |w - v'| = 1$.  Since
\begin{equation*}
|\varphi(v) - \varphi(v')| \leq |\varphi(v) - \varphi(x)| + |\varphi(x) - \varphi(v')|
\end{equation*}
and since both edges $\{v, w\}$ and $\{w, v'\}$ lie on vertical geodesics, we have that
\[
|\varphi(v)-\varphi(v')| \le 2(\alpha+\beta)|v-v'|.
\]
The claim follows.
\end{proof}

We move now to compositions of vertical quasi-isometries. Heuristically, the claim follows from the geodesic stability in hyperbolic fillings.

\begin{lemma}\label{VQI comp}
Let $X, Y, Z$ be doubling metric spaces, and let $\widehat{X}\in \HF(X)$, $\widehat{Y}\in \HF(Y)$, and $\widehat{Z}\in \HF(Z)$ be hyperbolic fillings. Let $\varphi \colon (\widehat{X}, \ast) \to (\widehat{Y}, \ast)$ and $\psi \colon (\widehat{Y}, \ast) \to (\widehat{Z}, \ast)$ be $(\alpha_\varphi, \beta_\varphi)$ and $(\alpha_\psi, \beta_\psi)$-vertical quasi-isometries, respectively.  Then $\omega = \psi \circ \varphi \colon (\widehat{X}, \ast) \to (\widehat{Z}, \ast)$ is an $(\alpha, \beta)$-vertical quasi-isometry, where $\alpha \ge 1$ and $\beta\ge 0$ depend quantitatively on the data of $\widehat{X}$, $\widehat{Y}$, $\widehat{Z}$, $\varphi$, and $\psi$.
\end{lemma}

\begin{proof}
Let $\gamma \colon (\N_0,0) \to (\widehat{X},\ast)$ be a vertical geodesic.  Then, $\varphi \circ \gamma \colon (\N_0,0) \to (\widehat{Y},\ast)$ is a vertical quasigeodesic.  By Lemma \ref{quasigeod close lemma}, there is a constant $H = H(\alpha_\varphi, \beta_\varphi, \widehat{Y})$ and a vertical geodesic $\sigma \colon (\N_0,\ast) \to (\widehat Y,\ast)$ with $\Hdist(|\sigma|, |\gamma|) \leq H$.

Let $j, j' \in \N_0$ and let $k, k'\in \N_0$ be indices for which $|(\varphi \circ \gamma)(j) - \sigma(k)| \leq H$ and  $|(\varphi \circ \gamma)(j') - \sigma(k')| \leq H$.   Then
\begin{equation*}
|(\varphi \circ \gamma)(j) - (\varphi \circ \gamma)(j')| \leq 2H + |\sigma(k) - \sigma(k')| = 2H + |k - k'|.
\end{equation*}
By Lemma \ref{upper qi}, $\psi$ is $2(\alpha_\psi + \beta_\psi)$-Lipschitz. Thus, for $H' = 2(\alpha_\psi + \beta_\psi)H$, we have
\begin{equation*}
|(\psi \circ \varphi \circ \gamma)(j) - (\psi \circ \sigma)(k)| = |\psi(\varphi(\gamma(j)))-\psi(\sigma(k))| \le  2(\alpha_\psi + \beta_\psi) |\varphi(\gamma(j))-\sigma(k)| \le H'.
\end{equation*}
Likewise
\begin{equation*}
|(\psi \circ \varphi \circ \gamma)(j') - (\psi \circ \sigma)(k')| \le H'.
\end{equation*}

Since $\psi$ is a vertical quasi-isometry and $\sigma$ is a vertical geodesic, the path $\psi \circ \sigma$ is a vertical quasigeodesic. Thus
\begin{align*}
|(\psi \circ \varphi \circ \gamma)(j) - (\psi \circ \varphi \circ \gamma)(j')| &\leq 2H' + |(\psi \circ \sigma)(k) - (\psi \circ \sigma)(k')| \\
&\leq 2H' + \alpha_\psi|k-k'| + \beta_\psi.
\end{align*}
Since $\gamma$ is a vertical geodesic, we also have that
\begin{equation*}
|k - k'| = |\sigma(k) - \sigma(k')| \leq 2H + |(\varphi \circ \gamma)(j) - (\varphi \circ \gamma)(j')| \leq 2H + \alpha_\varphi |j-j'| + \beta_\varphi.
\end{equation*}
Thus
\begin{equation*}
|(\psi \circ \varphi \circ \gamma)(j) - (\psi \circ \varphi \circ \gamma)(j')| \leq 2H' + \alpha_\psi ( 2H + \alpha_\varphi |j-j'| + \beta_\varphi) + \beta_\psi
\end{equation*}
This establishes the upper inequality of the vertical quasi-isometry condition with $\alpha = \alpha_\psi \alpha_\varphi$ and $\beta = 2H' + \alpha_\psi (2H + \beta_\varphi) + \beta_\psi$.

We now establish the lower vertical quasi-isometry inequality. Let $j, j', k, k'$ be indices as above. We apply the vertical quasi-isometry condition to vertical quasi-isometry $\varphi \colon (\widehat X, \ast) \to (\widehat Y,\ast)$ and to vertical geodesic $\gamma \colon (\N_0,0) \to (\widehat X,\ast)$ and rearrange the above inequalities to have that
\begin{equation*}
\alpha_\varphi^{-1} |j-j'| - \beta_\varphi - 2H \leq |(\varphi \circ \gamma)(j) - (\varphi \circ \gamma)(j')| -2H \leq |k - k'|.
\end{equation*}
The vertical quasi-isometry condition applied to vertical quasi-isometry $\psi \colon (\widehat Y,\ast) \to (\widehat Z,\ast)$ and geodesic $\sigma \colon (\N_0,0) \to (\widehat Y,\ast)$
% combined with the $H'$ bound above 
gives
\begin{equation*}
\alpha_\psi ^{-1} |k - k'| - \beta_\psi \leq |(\psi \circ \sigma)(k) - (\psi \circ \sigma)(k')| \leq 2H' + |(\psi \circ \varphi \circ \gamma)(j) - (\psi \circ \varphi \circ \gamma)(j')|.
\end{equation*}
Combining these inequalities yields
\begin{equation*}
\alpha_\varphi^{-1} |j-j'| - \beta_\varphi - 2H \leq \alpha_\psi(2H' + |(\psi \circ \varphi \circ \gamma)(j) - (\psi \circ \varphi \circ \gamma)(j')| + \beta_\psi)
\end{equation*}
and so
\begin{equation*}
\alpha_\psi^{-1}(\alpha_\varphi^{-1} |j-j'| - \beta_\varphi - 2H) - 2H' - \beta_\psi \leq |(\psi \circ \varphi \circ \gamma)(j) - (\psi \circ \varphi \circ \gamma)(j')|.
\end{equation*}
This establishes the lower inequality of the vertical quasi-isometry condition with $\alpha = \alpha_\psi \alpha_\varphi$ and $\beta = \alpha_\psi^{-1}(\beta_\varphi + 2H) + 2H' + \beta_\psi$.
\end{proof}

%%%%%%%%%%%%%%%%%%%%%%%%%%%%%%%%%%%%%%%%%%%%%%%%%%%%%%%%%%%%%%%%%%%%%%%%%%%%%%%
%%%%%%%%%%%%%%%%%%%%%%%%%%%%%%%%%%%%%%%%%%%%%%%%%%%%%%%%%%%%%%%%%%%%%%%%%%%%%%%
%%%%%%%%%%%%%%%%%%%%%%%%%%%%%%%%%%%%%%%%%%%%%%%%%%%%%%%%%%%%%%%%%%%%%%%%%%%%%%%

\part{Extensions}
\label{part:extension}

\section{Extension of branched quasisymmetries into hyperbolic fillings}
\label{BQS to VQI sec}

In this section, we prove Theorem \ref{thm:BQS2VQI} in the introduction. We call the map $\varphi_f \colon \widehat X \to \widehat Y$ in the statement a hyperbolic filling of $f\colon X\to Y$. For the statement, we say that a continuous map $f\colon X\to Y$ is \emph{non-collapsing} if $\diam fB>0$ for all balls $B \subset X$. In particular, a non-constant branched quasisymmetry from a bounded turning space is non-collapsing.

\begin{defn}\label{Extension def}
Let $f\colon X\to Y$ be a non-collapsing map between compact metric spaces $X$ and $Y$, and let $\widehat X \in \HF(X)$ and $\widehat{Y} \in \HF(Y)$ be hyperbolic fillings of $X$ and $Y$, respectively. A map $\varphi \colon \widehat X \to \widehat Y$ is a \emph{hyperbolic filling of $f$} if, for each vertex $v\in \widehat X$, the image $\varphi(v)$ is a vertex in $\widehat Y$ of maximal level $\ell(\varphi_f(v))$ among all vertices $w\in \widehat Y$ satisfying $f(B_v) \subset B_w$.
\end{defn}

\begin{remark}
Note that the assumption that $f$ is non-collapsing is necessary and sufficient condition in the definition for the existence of hyperbolic filling. If we assume that, in addition, the space $X$ is a continuum, we may replace non-collapsing by more natural assumption that the map is light; see Lemma \ref{lemma:continuum-continuum}. 
\end{remark}

Having this terminology and Theorem \ref{thm:BQS to PBQS} at our disposal, it suffices to show that each power branched quasisymmetry has a hyperbolic filling, which is a vertical quasi-isometry. We reformulate Theorem \ref{thm:BQS2VQI} as follows.

\begin{thm}
\label{thm:pBQS2VQI}
Let $X$ and $Y$ be compact metric spaces and suppose that $X$ has bounded turning, let $\widehat{X}$ and $\widehat{Y}$ be hyperbolic fillings of $X$ and $Y$, respectively, and let $f\colon X\to Y$ be a power branched quasisymmetry. Then a hyperbolic filling $\varphi \colon \widehat{X} \to \widehat{Y}$ of $f$ is a vertical quasi-isometry, quantitatively.
\end{thm}

The proof is based on the diameter comparison of images $fB_v$ and $fB_w$ of balls $B_v$ and $B_w$ on a same vertical geodesic in $\widehat X$. We formulate this as the following lemma.

\begin{lemma}
\label{lemma:pBQS2VQI-lower}
Let $X$ and $Y$ be compact metric spaces and suppose that $X$ has $\lambda$-bounded turning, let $\widehat{X}\in \HF_s(X)$ and $\widehat{Y}\in \HF(Y)$ be hyperbolic fillings of $X$ and $Y$, respectively, and let $f\colon X\to Y$ be a branched quasisymmetry. Then, for vertices $v$ and $v'$ on a vertical geodesic $\gamma \colon (\N_0,0) \to (X,\ast)$, we have
\[
\diam(fB_{v'}) \lesssim \eta\left( C \frac{\diam B_{v'}}{\diam B_{v}} \right) \diam(fB_{v}),
\]
where $C>0$ depends only on $\lambda$ and $s$ and the constant of comparability only on $\eta$ and $\lambda$.
\end{lemma}

\begin{proof}
Let $E_v = E_{B_v}$ and $E_{v'} = E_{B_{v'}}$. For $t>0$, let also $F^t_V = E_{t B_v}$ and $F^t_{v'} = E_{t B_{v'}}$. 

By Lemma \ref{basic filling lemma 2}, there exists a constant $a>1$, depending only on $s$, for which $F^a_v \cap F^a_{v'}\ne \emptyset$. Indeed, we may assume that $\ell(v') > \ell(v)$. Then there exists $A>0$, depending on $s$, for which $B_{v'}\subset (A+1)B_v$. Then $E_{v'} \cap E_{(A+1)B_v} \ne \emptyset$. Thus $F^a_v \cap F^a_{v'} \ne \emptyset$ for $a=A+1$. 

Since $B_{v'} \subset F^a_{v'}$, we have that
\begin{align*}
\diam(f B_{v'}) &\le \diam(f F^a_{v'}) \le \eta\left(\frac{\diam F^a_{v'}}{\diam F^a_v}\right) \diam(f F^a_v) \\
&\le \eta\left(\frac{\diam F^a_{v'}}{\diam F^a_v}\right) \eta\left(\frac{\diam F^a_v}{\diam E_v}\right)\diam(f E_v).
\end{align*}
By Lemma \ref{continuum lemma}, we have
\[
\frac{\diam F^a_{v'}}{\diam F^a_v} \le \frac{2\lambda a \diam B_{v'}}{\diam B_v}.
\]
and
\[
\frac{\diam F^a_v}{\diam E_v} \le \frac{2\lambda a \diam B_v}{\diam B_v} = 2\lambda a.
\]
By Lemma \ref{lemma:BQS-ball continuum equiv}, we also have that 
\[
\diam(fE_v) \simeq \diam(fB_v), 
\]
with a constant depending only on $\eta$ and $\lambda$. The claim follows.
\end{proof}

We are now ready for the proof of Theorem \ref{thm:pBQS2VQI}.

\begin{proof}[Proof of Theorem \ref{thm:pBQS2VQI}]
Let $s, u > 1$ be such that $\widehat X \in \HF_s(X)$ and $\widehat Y \in \HF_u(Y)$ and let $f\colon X\to Y$ be a power branched quasisymmetry, that is, $f$ is a branched $\eta$-quasisymmetry with $\eta \colon [0,\infty) \to [0,\infty)$, $t\mapsto C \max\{ t^{1/q}, t^q\}$ for $C>0$ and $q\in (0,1)$.

Let $\gamma \colon (\N_0,0) \to (\widehat X, \ast)$ be a vertical geodesic and let $j,j'\in \N_0$. Let also $v=\gamma(j)$ and $v'=\gamma(j')$.  We may assume that $j'>j$.

In the following we denote by $K$ any constant that depends quantitatively on the parameters.  The value of $K$ may change from line to line or even within the same line.

We note that $\diam(B_w) \simeq s^{-\ell(w)}$ for $w \in \widehat X$ and $\diam(B_w) \simeq u^{-\ell(w)}$ for $w \in \widehat Y$.  It follows that 
\[
\ell(w) + K \leq -\log_s(\diam(B_w)) \leq \ell(w) + K
\]
and 
\[
\ell(\varphi(w)) + K \leq -\log_u(\diam(f B_w)) \leq \ell(\varphi(w)) + K
\]
for $w \in \widehat X$

By Lemma \ref{lemma:pBQS2VQI-lower}, we have that
\[
\diam(fB_{v'}) \le \eta\left(K \frac{\diam B_{v'}}{\diam B_v}\right) \diam(fB_v).
\]
Since $\diam B_{v'} \lesssim \diam B_v$, we have, with $c = \log(s) / \log(u) > 0$ and $Q = q$ or $Q = 1/q$ depending on our input, that
\begin{align*}
\ell(\varphi(v')) &\geq K - \log_u(\diam(f B_{v'})) \ge K - \log_u \left(K \eta\left(K \frac{\diam B_{v'}}{\diam B_v}\right)\right) - \log_u (\diam(fB_v)) \\
&\geq K - \log_u \left(K \left( K \frac{\diam B_{v'}}{\diam B_v}\right)^Q \right) + \ell(\varphi(v)) \\
& \geq K - Q \biggl(\log_u(\diam(B_{v'})) - \log_u(\diam(B_v))\biggr) + \ell(\varphi(v)) \\
& \geq K - Qc \biggl(\log_s(\diam(B_{v'})) - \log_s(\diam(B_v))\biggr) + \ell(\varphi(v)) \\
& \geq K - Qc \biggl(-\ell(v') + \ell(v) \biggr) + \ell(\varphi(v)) \\
&\geq K +  Qc \biggl(\ell(v') - \ell(v) \biggr) + \ell(\varphi(v)). 
\end{align*}
Since $v$ and $v'$ are on the same vertical geodesic and $Q \geq q$, we have that
\[
|\varphi(v')-\varphi(v)| \ge \ell(\varphi(v')) - \ell(\varphi(v)) \geq K + qc |v-v'|.
\]

For the other direction, we consider first adjacent vertices $v$ and $v'$ on the same vertical geodesic in $\widehat X$. Since $\varphi$ is a hyperbolic filling of $f$, we have that 
\[
B_{\varphi(v)} \cap B_{\varphi(v')} \supset f(B_v) \cap f(B_{v'}) \supset f(B_v \cap B_{v'}) \ne \emptyset.
\]
Thus either we have $\ell(\varphi(v)) = \ell(\varphi(v'))$ and $|\varphi(v)-\varphi(v')|\leq 1$ or there exists a geodesic segment $\varphi(v)=w_0, \ldots, w_k = \varphi(v')$ in $\widehat Y$ satisfying $\ell(w_j) = \ell(\varphi(v))+j$ for each $j=0,\ldots, k$. In the first case, $\varphi(v)$ and $\varphi(v')$ are adjacent in $\widehat Y$. Thus we may assume that the latter case holds. Then 
\[
|\varphi(v)-\varphi(v')| = |\ell(\varphi(v))-\ell(\varphi(v'))|.
\] 

Since $v$ and $v'$ are adjacent, we have by Lemma \ref{lemma:pBQS2VQI-lower} that
\[
\diam(fB_v) \simeq \diam fB_{v'},
\]
where the constants depend only on $\eta$, $\lambda$, and $s$. Thus the distance $|\varphi(v)-\varphi(v')|$ is uniformly bounded by a constant depending only on $\eta$, $\lambda$, and $s$. The claim now follows from the triangle inequality.
\end{proof}

\begin{remark}
\label{power bqs level rmk}  
Suppose that $f\colon X\to Y$ is a surjective power branched quasisymmetry with exponent $q \in (0, 1]$ and that $X$ and $Y$ are compact and connected spaces with $\diam(X) = \diam(Y) =  1$. Then, for each continuum $E\subset X$, we have
\begin{equation*}
\diam(fE) = \frac{\diam(f E)}{\diam(Y)}  \lesssim \eta \biggl(\frac{\diam(E)}{\diam(X)} \biggr) \leq \diam(E)^q
\end{equation*}
and
\begin{equation*}
\frac{1}{\diam(fE)} = \frac{\diam(Y)}{\diam(fE)}  \lesssim \eta \biggl(\frac{\diam(Z)}{\diam(E)} \biggr) \leq \diam(E)^{-1/q}.
\end{equation*}
Thus
\begin{equation*}
\diam(E)^{1/q} \lesssim \diam(fE) \lesssim \diam(E)^q.
\end{equation*}
Hence, by Lemma \ref{lemma:BQS-ball continuum equiv}, we have, for each vertex $v\in \widehat X$ in $\widehat X \in \HF_s(X)$, that 
\begin{equation*}
s^{-\ell(v)/q} \lesssim \diam(B_{\varphi(v)}) \lesssim s^{-\ell(v)q}.
\end{equation*}
In particular, there is a constant $C > 0$ such that 
\begin{equation*}
\ell(v)q - C \leq \ell(\varphi(v)) \leq\ell(v)/q + C.
\end{equation*}
for each hyperbolic filling $\varphi \colon \widehat X\to \widehat Y$. 
\end{remark}

\subsection{Finite multiplicity of the filling}

Heuristically, a hyperbolic filling of a discrete and open branched quasisymmetry between doubling metric spaces has finite multiplicity. We formulate this as follows.

\begin{thm}\label{thm:bdd BQS to bdd VQI}
Let $X$ and $Y$ be compact, $\lambda$-bounded turning, and doubling metric spaces, and let $\widehat{X} \in \HF_s(X)$ and $\widehat{Y} \in \HF(Y)$.  Let $f \colon X \to Y$ be a discrete and open power branched quasisymmetry, and let $\varphi \colon \widehat X\to \widehat Y$ be a hyperbolic filling of $f$. Then, $\varphi$ has finite multiplicity.  Moreover, $N(\varphi)$ depends only on $f$, $\lambda$, and the doubling constants of $X$ and $Y$.
\end{thm}

\begin{proof}
Let $w \in \widehat Y$.  Let $v_1, \dots, v_J$ be a set of vertices in $\varphi^{-1}(w)$ for which $B_{v_i} \cap B_{v_j} = \emptyset$ for $i \neq j$. We show first that $J$ is bounded from above independent of $w$.

To see this, let $i\in \{1,\ldots, J\}$. By the definition of $\varphi$ and Lemma \ref{vertex comp}, we have that $\diam(fB_{v_i}) \simeq \diam(B_w)$, where the constants depend only on $s$.  Thus, by Koebe distortion theorem (Theorem \ref{thm:Koebe}), there exists a ball $B_i \subset fB_{v_i}$ of diameter $c_1 \diam(B_w)$, where $c_1>0$ depends only on the data.

Since the balls $B_1,\ldots, B_J$ are mutually disjoint, the upper bound for $J$ follows now from the doubling constant of $X$ and multiplicity of $f$.

To obtain the bound for $|\varphi^{-1}(w)|$, it remains to estimate the size of the set $V_i = \{v \in \widehat X \colon B_v \cap B_i,\ \varphi(v)=\varphi(v_i)\}$ for each $i=1,\ldots, J$. Let $\alpha\ge 1$ and $\beta>0$ be constants for which $\varphi$ is an $(\alpha,\beta)$-vertical quasi-isometry. Let $v\in V_i$. Then
\[
|v-v_i| \le \alpha | \varphi(v)-\varphi(v_i) | + \beta = \beta.
\]
Thus $V_i \subset \bar B(v_i,\beta)$.

Since $X$ is doubling, we have that the number of vertices in the ball $\bar B(v_i,\beta)$ in $\widehat X$ is bounded by a constant depending only on $\beta$ and doubling constant of $X$. This completes the proof.
\end{proof}

\subsection{Coboundedness}

Recall that the image of a quasi-isometry $\psi \colon \widehat X\to \widehat Y$ is cobounded, that is, there exists $R>0$ for which $\widehat Y \subset B(\psi(\widehat X),R)$. A similar property is shared by hyperbolic fillings of surjective branched quasisymmetries.

\begin{lemma}
Let $X$ and $Y$ be compact metric spaces and suppose $X$ has bounded turning. Let $\widehat X\in \HF_s(X)$ and $\widehat Y \in \HF_s(Y)$ be hyperbolic fillings of $X$ and $Y$, respectively. Let $\varphi \colon \widehat{X} \to \widehat{Y}$ be a hyperbolic filling of a surjective branched quasisymmetry  $f \colon X \to Y$. Then there is a constant $R > 0$ satisfying $\widehat Y \subset B(\varphi(\widehat X),R)$.
\end{lemma}

\begin{proof}
Let $w \in \widehat{Y}$ and let $\gamma \colon \N_0 \to \widehat{Y}$ be a vertical geodesic passing through $w$, that is, there exist $j\in\N_0$ for which $\gamma(j) = w$.  Let $y\in Y$ be such that $\gamma(j)\to y$ as $j\to \infty$, and let $x\in X$ be such that $f(x)=y$. Let also $\sigma \colon \N_0 \to \widehat X$ be a vertical geodesic tending to $x$. Then $\varphi \circ \sigma \colon (\N_0,0)\to (\widehat Y,\ast)$ is a vertical quasigeodesic satisfying $\varphi\circ \sigma(j) \to y$ as $j\to \infty$. Thus, by Lemma \ref{quasigeod close lemma}, there exists $H_1>0$, independent of $\gamma$ and $\sigma$, for which $\Hdist(\gamma,\varphi\circ \sigma) \le H_1$. The claim follows.
\end{proof}

\subsection{Hyperbolic fillings of local branched quasisymmetries}

In this section, we revisit the extension into hyperbolic fillings, this time for local branched quasisymmetries. For the statement, we define the notion of an eventual vertical quasi-isometry.

\begin{defn}\label{local VQI def}
A map $\varphi \colon (\widehat X, \ast) \to (\widehat Y, \ast)$ is an \emph{eventual $(\alpha,\beta)$-vertical quasi-isometry (with cutoff $J>0$)} if for each  vertical geodesic $\gamma \colon (\N_0,0) \to (\widehat X, \ast)$ the map $\varphi \circ \gamma \colon (\N_0,0) \to (\widehat Y,\ast)$ satisfies
\begin{equation*}
\frac{1}{\alpha}|j-j'| - \beta \leq |\varphi(\gamma(j)) - \varphi(\gamma(j'))| \leq \alpha|j-j'| + \beta
\end{equation*}
for all $j, j' \geq J$.  
\end{defn}

\begin{remark}
\label{rmk:eventual}
Note that, if $\varphi \colon (\widehat X, \ast)\to (\widehat Y, \ast)$  an eventual $(\alpha,\beta)$-vertical quasi-isometry $\varphi \colon \widehat X\to \widehat Y$ with cutoff $J>0$ between hyperbolic fillings $\widehat X$ and $\widehat Y$, then the map $\psi \colon (\widehat X,\ast) \to (\widehat Y,\ast)$, defined by $v\mapsto \ast$ for $\ell(v) \le J$ and $v\mapsto \varphi(v)$ for $\ell(v)>J$, is an $(\alpha, 2J+\beta)$-vertical quasi-isometry. 
\end{remark}

We formalize a local version of Theorem \ref{thm:BQS2VQI} as follows. Since the proof is almost verbatim, we give merely a sketch of a proof. Recall that under the bounded turning assumption, each local branched quasisymmetry is a local power branched quasisymmetry; see Remark \ref{rmk:local-BQS to pBQS}.

\begin{thm}%[Local version of Theorem \ref{thm:BQS2VQI}]
\label{thm:local BQS to VQI}
Let $X$ and $Y$ be compact metric spaces, where $X$ has $\lambda$-bounded turning. Let $\widehat{X} \in \HF(X)$ and $\widehat{Y} \in \HF(Y)$ be hyperbolic fillings of $X$ and $Y$, respectively.  Let $f \colon X \to Y$ be a discrete and open local power branched quasisymmetry. Then the hyperbolic filling $\varphi \colon \widehat X \to \widehat Y$ of $f$ is an eventual vertical quasi-isometry of finite multiplicity, quantitatively.
\end{thm}

\begin{proof}[Sketch of a proof]
Let $\varepsilon>0$ be the locality scale for $f$, that is, \eqref{BQS} holds for continua of diameter at most $\varepsilon$.

By Remark \ref{rmk:local ball continuum equiv}, we have an effective comparison of diameters of $fB_v$ and $B_{\varphi(v)}$ if $\diam((A_s + 1)B_v) \le \varepsilon$. We may now fix the level $J>0$, which satisfies $s^{-J} < \varepsilon/(A_s +1)$. We may now follow the proof of Theorem \ref{thm:pBQS2VQI}.
\end{proof}

%%%%%%%%%%%%%%%%%%%%%%%%%%%%%%%%%%%%%%%%%%%%%%%%%%%%%%%%%%%%%%%%%%%%%%%%%%%%%%%
%%%%%%%%%%%%%%%%%%%%%%%%%%%%%%%%%%%%%%%%%%%%%%%%%%%%%%%%%%%%%%%%%%%%%%%%%%%%%%%
%%%%%%%%%%%%%%%%%%%%%%%%%%%%%%%%%%%%%%%%%%%%%%%%%%%%%%%%%%%%%%%%%%%%%%%%%%%%%%%

\section{Extension of vertical quasi-isometries to the boundary}
\label{VQI to BQS sec}

In this section we prove that vertical quasi-isometries between hyperbolic fillings of compact, doubling metric spaces extend to branched quasisymmetries between their boundaries. For the statement and discussion, we begin with a discussion on the trace maps $X\to Y$ of maps between hyperbolic fillings $\widehat X\to \widehat Y$.

\subsection{Traces}
\label{sec:traces}

We call the map $f\colon \partial \widehat X\to \partial \widehat Y$ induced by a map $\psi \colon \widehat X\to \widehat Y$ a trace of $\psi$. More formally, we give the following definition.

\begin{defn}\label{bdry VQI extension}
Let $X$ and $Y$ be compact metric spaces and $\widehat X \in \HF(X)$ and $\widehat Y \in \HF(Y)$ their hyperbolic fillings, respectively. A mapping $f_\varphi \colon \partial \widehat X \to \partial \widehat Y$ is a \emph{trace of a mapping $\varphi \colon \widehat X\to \widehat Y$} if, for each $x\in \partial \widehat X$ and a vertical geodesic $\gamma \colon (\N_0,0) \to (\widehat X,\ast)$ tending to $x$, $(\varphi\circ \gamma)(j) \to f_\varphi(x)$ as $j\to \infty$.
\end{defn}

In what follows, we identify, as we may, the boundary $\partial \widehat X$ with the original space $X$ and call the map $f\colon X\to Y$, which coincides with $f_\varphi \colon \partial \widehat X \to \partial \widehat Y$ after the identification, the \emph{trace of $\varphi \colon \widehat X\to \widehat Y$}, although the map $f$ depends on the identifications.

Our first observation is that the trace $X\to Y$ of a vertical quasi-isometry $\widehat X\to \widehat Y$ is well-defined and continuous. 

\begin{lemma}
\label{lemma:trace_basic}
Let $X$ and $Y$ be compact metric spaces and let $\widehat{X} \in \HF_s(X)$ and $\widehat{Y} \in \HF_t(Y)$ be their hyperbolic fillings. Then, each $(\alpha,\beta)$-vertical quasi-isometry, has a continuous trace $f \colon X\to Y$. 
\end{lemma}

\begin{proof}
For the existence, let $x\in X$ and $\gamma \colon (\N_0,0) \to (\widehat X,\ast)$ a vertical geodesic tending to $x$. Since $\varphi$ is a vertical quasi-isometry, $\varphi \circ \gamma \colon (\N_0,0) \to (\widehat Y, \ast)$ is a quasigeodesic tending to a point, say $f(x)$, on $\widehat Y$. By Lemma \ref{bdd hausdorff vqi rays}, for each vertical geodesic $\gamma'\colon (\N,0)\to (\widehat X,\ast)$ tending to $x$, the image $\varphi \circ \gamma'$ is of bounded distance from $\varphi \circ \gamma$. Thus $(\varphi\circ \gamma')(j) \to f(x)$. The map $f\colon X\to Y$, $x\mapsto f(x)$, therefore is well-defined.

To show that the trace $f$ is continuous, let $x\in X$ and $\epsilon > 0$. Let also $\gamma \colon (\N_0,0) \to (\widehat X,\ast)$ be a centered vertical geodesic tending to $x$. Let now $A_t>0$ be the structure constant \eqref{eq:As} and $H=H(\alpha,\beta)>0$ a constant as in Lemma \ref{bdd hausdorff vqi rays} for $\widehat Y$. Since $\ell(\varphi(\gamma(j))) \to \infty$ as $j\to \infty$, we may fix a level $j_0\in \N_0$ for which 
\[
3(A_t + 1) t^H t^{-\ell(\varphi(\gamma(j)))} < \epsilon/2
\]
for all $j\ge j_0$.

Let now $\delta=s^{-j_0}/2$. We show that $f_\varphi B(x,\delta) \subset B(f_\varphi(x),\epsilon)$. Let $x'\in B(x,\delta)$ and let $\gamma'\colon (\N_0,0) \to (\widehat X, \ast)$ be a vertical geodesic tending to $x'$. Since $d(x,x')<s^{-j_0}$, we may further assume that $\gamma(j_0)=\gamma'(j_0)$. Then, by Lemma \ref{qg filling lemma}, we have that
\[
d(f(x), c(B_{\varphi(\gamma'(j_0))}) < 3(A_t +1) t^H t^{-\ell(\varphi(\gamma(j_0)))} < \varepsilon/2
\]
for the center of the ball $B_{\varphi(\gamma'(j_0))}$ in $X$ and likewise for $f(x')$. Thus, by the triangle inequality, $d(f(x),f(x')) < \varepsilon$. Hence, the map $f_\varphi$ is continuous.
\end{proof}

Heuristically, the trace is a left inverse of the hyperbolic filling. We state this formally as follows. Note that, in the statement, the assumption that the hyperbolic filling is a vertical quasi-isometry is required by the definition of the trace.

\begin{lemma}
\label{lemma:trace-filling}
Let $X$ and $Y$ be compact metric spaces and let $\widehat{X} \in \HF(X)$ and $\widehat{Y} \in \HF(Y)$ be their hyperbolic fillings, respectively.  Let $f \colon X \to Y$ be a non-collapsing map admitting a hyperbolic filling $\varphi \colon \widehat X \to \widehat Y$, which is a vertical quasi-isometry. 
Then the trace $f_\varphi \colon X\to Y$ of $\varphi$ satisfies $f_\varphi = f$.
\end{lemma}

\begin{proof}
Let $x\in X$ and let $\gamma \colon (\N_0,0) \to (\widehat{X},\ast)$ be a vertical geodesic centered at $x$. It suffices to show $\varphi \circ\gamma$ tends to $f(x)$. Since $\varphi_f$ is a vertical quasi-isometry, it follows that $\varphi\circ \gamma$ tends to a point $y\in Y$.  Since $f(B_{\gamma(k)}) \subseteq B_{\varphi(\gamma(k))}$ for each $k\in \N_0$, it follows that $f(x) \in B_{\varphi_f(\gamma(k))}$ for all $k$.  Since $\diam(\varphi_f(\gamma(k))) \to 0$ as $k\to \infty$, we have that $y = f(x)$.  
\end{proof}

\subsection{Finite multiplicity of the trace}

We now add the assumption that our vertical quasi-isometry has finite multiplicity and show that the trace has finite multiplicity.

\begin{lemma}\label{degree lemma}
Let $X$ and $Y$ be compact metric spaces, where $Y$ is doubling, and let $\widehat{X} \in \HF_s(X)$ and $\widehat{Y} \in \HF_t(Y)$. Let $f\colon X\to Y$ be a trace of an $(\alpha,\beta)$-vertical quasi-isometry $\varphi \colon \widehat X \to \widehat Y$ of finite multiplicity. Then $f$ has finite multiplicity, quantitatively.  That is,
\[
N(f) \lesssim N(\varphi)\left( \alpha +\beta \right),
\]
where the constant depends only on $t$, $s$, and the doubling constant of $Y$.
\end{lemma}

\begin{proof}
Let $y \in Y$ and let $G = f^{-1}(y)$.  For each $x \in G$, let $\gamma_x \colon (\N_0,0) \to (\widehat X, \ast)$ be a vertical geodesic tending to $x$.
Let $L \in \N$ be a natural number such that $L \leq |f_\varphi^{-1}(y)|$. We show that $L$ has a uniform upper bound.

By Lemma \ref{basic filling lemma 2}, there are points $x_1, \dots, x_L \in G$ and a level $K$ depending on $x_1, \dots, x_L$ such that for all $k, k' \geq K$ we have that $\gamma_{x_i}(k) \neq \gamma_{x_j}(k')$ when $i \neq j$.  Let $n = \sup_i{\ell(\varphi(\gamma_{x_i}(K)))}$.  For each $x_i$, let 
\begin{equation*}
k_i = \inf\{k : k \geq K \text{ and } \ell(\varphi(\gamma_{x_i}(k + 1)) > n\}.
\end{equation*}
Then $\ell(\varphi(\gamma_{x_i}(k_i + 1))) > n$ and $\ell(\varphi(\gamma_{x_i}(k_i))) \leq n$ for each $i=1,\ldots, L$.  

Since $\varphi$ is an $(\alpha, \beta)$-vertical quasi-isometry, we have that $\ell(\varphi(\gamma_{x_i}(k_i + 1))) \leq n + \alpha + \beta$ for each $i=1,\ldots, L$.  Since each $\varphi(\gamma_{x_i})$ tends to $y$, we have, by Lemma \ref{qg filling lemma 2}, that there are at most $m(\alpha + \beta)$ possible images of $\varphi(\gamma_{x_i}(k_i + 1))$, where $m\in \N_0$ depends only on $\alpha$, $\beta$, $\eta$, $t$, and the doubling constant of $Y$.  Hence we have that $L \leq N(\varphi) m(\alpha + \beta)$.  As $y$ and $L$ were arbitrary, we conclude that $N(f) \leq N(\varphi) m(\alpha+\beta)$.
\end{proof}

\subsection{Traces of finite multiplicity vertical quasi-isometries are branched quasisymmetric}
\label{sec:traces_R_BQS}

In this section, we show that traces of vertical quasi-isometries of finite multiplicity between hyperbolic fillings are branched quasisymmetries, i.e.~Theorem \ref{thm:intro-VQI2pBQS} in the introduction. We reformulate this theorem as follows.

\begin{thm}
\label{thm:VQI2pBQS}
Let $X$ and $Y$ be compact, doubling metric spaces and let $\widehat X \in \HF_s(X)$ and $\widehat Y\in \HF_t(Y)$ be hyperbolic fillings. Let $\varphi \colon (\widehat X,\ast) \to (\widehat Y, \ast)$ be a pointed vertical quasi-isometry of finite multiplicity. Then the trace $f\colon X\to Y$ of $\varphi$ is a power branched quasisymmetry of finite multiplicity. 
\end{thm}

The assumption that the vertical quasi-isometry has finite multiplicity is necessary as the following example shows.

\begin{example}
Consider $X = [0,1]^2$ and $Y = [0,1]$.  Let $f \colon X \to Y$ be the second coordinate projection $(x,y) \mapsto y$. Let $\widehat{X} \in \HF(X)$ be the hyperbolic filling of $X$ obtained by taking as vertex centers all dyadic points, that is, $\{(x,y) : x, y \in \{ i2^{-n} : 0 \leq i \leq 2^n\}\}$. Similarly, let  $\widehat{Y} \in \HF(Y)$ be obtained similarly with centers in $\{i2^{-n} : 0 \leq i \leq 2^n\}$.  Let then $\varphi \colon \widehat{X} \to \widehat{Y}$ be the unique map satisfying $c_{\widehat Y}(\varphi(v))=f(c_{\widehat X}(v))$ and $\ell_{\widehat Y}(\varphi(v))=\ell_{\widehat X}(v)$ for all $v\in \widehat X$. Then $f$ is a trace of $\varphi$; in fact, $\varphi$ is also a hyperbolic filling of $f$. The map $\varphi$ is a vertical $(1,0)$-quasi-isometry, but $f$ is not a branched quasisymmetry, since $f^{-1}(y) = [0,1]\times \{y\}$ for each $y\in [0,1]$.
\end{example}

Since we need to verify the distortion condition \eqref{BQS} for intersecting continua in $X$, we begin the proof of Theorem \ref{thm:VQI2pBQS} with a lemma which associates diameters of images of continua in $X$ to diameters of images $fB_v$ of vertex balls $B_v$ for $v\in \widehat X$.

\begin{lemma}\label{ball diam lemma}
Let $X$ and $Y$ be compact metric spaces, and let $\widehat{X} \in \HF_s(X)$ and $\widehat{Y} \in \HF_{t}(Y)$ be their hyperbolic fillings, respectively.  Let $\varphi \colon (\widehat X,\ast) \to (\widehat Y, \ast)$ be a pointed $(\alpha, \beta)$-vertical quasi-isometry of finite multiplicity, and let $f\colon X\to Y$ be the trace of $\varphi$. Then, for a continuum $E \subset X$ and a vertex $v \in \widehat X$ of highest level satisfying $E \subseteq B_v$, 
\begin{equation*}
\diam(fE) \simeq \diam(B_{\varphi(v)}),
\end{equation*}
where the constants depend on $\alpha, \beta, s, t, N(\varphi)$.
\end{lemma}

\begin{proof}
We first show that $\diam(fE) \lesssim \diam(B_{\varphi(v)})$.  Let $x \in E$, $y = f(x)$, and let $\gamma \colon (\N_0,0) \to (\widehat X,\ast)$ be a vertical geodesic passing through $v$ and tending to $x$.
Then, $\varphi\circ \gamma$ is an $(\alpha, \beta)$-quasigeodesic tending to $f(x)$. By Lemma \ref{qg filling lemma}, we have that
\[
d(f(x), c_{\widehat X}(B_v)) \le 3(A_t + 1) t^H t^{-\ell(\varphi(v))}.
\]
We conclude that
\[
\diam(fE) \le 6(A_t + 1) t^H t^{-\ell(\varphi(v))} \lesssim \diam(B_{\varphi(v)}).
\]

We prove now the other direction $\diam(B_{\varphi(v)}) \lesssim \diam(fE)$. By Lemma \ref{degree lemma}, $fE$ is not a point. Since $f$ is continuous by Lemma \ref{lemma:trace_basic}, we have that $fE$ is a continuum. Thus we may fix a vertex $w \in \widehat Y$ of maximal level such that $fE \subset B_w$.  By Lemma \ref{vertex comp}, we have that $\diam(fE) \simeq \diam(B_w) \simeq t^{-\ell(w)}$.  Since we aim to bound $\diam(fE)$ below, we may assume that $\ell(w) \ge \ell(\varphi(v))$.  

Let $\sigma \colon \{0, \dots, \ell(w)\} \to \widehat Y$ be a geodesic segment with $\sigma(\ell(w)) = w$ and $\sigma(0)=\ast$. Let $H_1>0$ be the constant in Lemma \ref{quasigeod close lemma} and let 
\begin{equation*}
\Sigma = \{w' \in \widehat Y \colon \dist(w', \sigma\{0,\ldots, \ell(w)\}) \leq 2H_1 \text{ and } \ell(w) - \alpha - \beta \leq \ell(w') < \ell(w) \}.
\end{equation*}
We fix $J = N(\varphi)|\Sigma| + 1$. By Lemma \ref{geod split lemma}, there exists $k_0\in \N$ depending only on $J$ and $s$, and geodesics $\gamma_1,\ldots, \gamma_J \colon (\N_0,0) \to (\widehat X, \ast)$ with the properties that $B_{\sigma_j(k)} \cap E \ne \emptyset$ for each $k\in \N_0$ and $1\le j\le J$, and $\gamma_j(\ell(v)+k) \ne \gamma_i(\ell(v)+k)$ for $k\ge k_0$ and $j\ne i$. 

We show that 
\begin{equation}
\label{eq:ell_w}
\ell(w) \le \ell(\varphi(v)) + \alpha(k_0+1)+ \beta.
\end{equation}
Suppose towards contradiction that \eqref{eq:ell_w} does not hold.

%For each $j=1,\ldots, J$, let $\sigma_j = \varphi \circ \gamma_j \colon (\N_0,0) \to (\widehat Y,\ast)$. 

Since $\varphi$ is an $(\alpha,\beta)$-vertical quasi-isometry, we have that
\[
\ell(\gamma_j(\ell(v)+k)) \le \ell(\varphi(v)) + \alpha k + \beta.
\]
For each $j=1,\ldots, J$, let $k_j\in \N_0$ be the smallest index $k>0$ for which 
\[
\ell(\gamma_j(\ell(v)+k)) \ge \ell(w).
\]
Since $\ell(w)\ge \ell(\varphi(v))$ and \eqref{eq:ell_w} does not hold, we have that $k_j \ge k_0+1 > k_0$ for each $j=1,\ldots, J$.

Consider the set 
\begin{equation*}
A = \{ \gamma_j (\ell(v) + k_j - 1) : 1 \leq j \leq J\}.
\end{equation*}
Since $k_j - 1 \geq k_0$ for each $j$, we have, by \eqref{item:gpl3} in Lemma \ref{geod split lemma}, that $|A| = J$. Moreover, the levels of vertices in $\varphi(A)$ are in $[\ell(w) - \alpha - \beta , \ell(w))$.

We claim that $\varphi(A) \subseteq \Sigma$. To show this, let $\sigma_j = \varphi \circ \gamma_j \colon (\N_0,0) \to (\widehat Y,\ast)$ for each $j=1,\ldots, J$. We fix  $j\in \{1,\ldots, J\}$ and denote $y\in fE$ the limit of $\sigma_j$. Since $fE \subset B_w$ and $\sigma(\ell(w)) = w$, we may extend $\sigma$ to a geodesic $\sigma \colon (\N_0,0) \to (\widehat Y,\ast)$ tending to $y$. Then, by Lemma \ref{quasigeod close lemma}, we have that $\sigma_j$ and $\sigma$ are within $H_1$ of each other. Since $\sigma_j(\ell(v) + k_j - 1)\in \varphi(A)$, we have by the bound on the levels of $\varphi(A)$ that $\varphi(A) \subseteq \Sigma$. Since $|A| = J = N(\varphi)|X| + 1$ and $|\varphi^{-1}(X)| \leq N(\varphi)|X|$, we have a contradiction. Thus \eqref{eq:ell_w} holds. 

In particular, we have that
\[
\ell(w) \leq \ell(\varphi(v)) + D,
\]
where $D>0$ is a constant depending only on $\alpha, \beta, N(\varphi)$, $s$, and $t$. Thus
\begin{equation*}
\diam(fE) \simeq t^{-\ell(w)} \gtrsim t^{-\ell(\varphi(v))} \simeq \diam(B_{\varphi(v)}).
\end{equation*}
The claim follows.
\end{proof}

We refine Lemma \ref{ball diam lemma} to geodesics.

\begin{lemma}\label{geod ball diam lemma}
Let $X$ and $Y$ be compact metric spaces and let $\widehat{X} \in \HF_s(X)$ and $\widehat{Y} \in \HF_{t}(Y)$ be their hyperbolic fillings, respectively.  Let $\varphi \colon (\widehat X,\ast) \to (\widehat Y,\ast)$ be an $(\alpha, \beta)$-vertical quasi-isometry of finite multiplicity, and let $f\colon X\to Y$ be the trace of $\varphi$. Then, for a continuum $E \subseteq X$ and a vertical geodesic $\gamma\colon (\N_0,0)\to (\widehat X,\ast)$ tending to a point in $E$, we have for the maximal index $j\in\N$ satisfying $E \subset 2A_s B_{\gamma(j)}$ that
\[
\diam(fE) \simeq \diam(B_{\varphi(\gamma(j))}),
\]
where the constants depend only on $\alpha$, $\beta$, $s$, $t$, and $N(\varphi)$.
\end{lemma}

\begin{proof}
We show that the distance between $v=\gamma(j)$ and a vertex $v'\in \widehat{X}$ of highest level satisfying $E \subseteq B_{v'}$ is bounded independently of $\gamma$.  First we estimate $\diam(E)$ in terms of $\diam(B_v)$.  Since $E \subseteq 2A_s B_v$, we have that
\begin{equation*}
\diam(E) \leq \diam(2A_s B_v) \leq 8 A_s s^{-\ell(v)}.
\end{equation*}
Let $x\in X$ be the limit of $\gamma$. Then, by Lemma \ref{basic filling lemma 2}, we have $d(x, c_{\widehat X}(v)) \leq A_s s^{-\ell(v)}$.  Hence, if $\diam(E) \leq A_s s^{-\ell(v)}$, then $E \subseteq 2 A_s B_v$.  As $v$ was the vertex of highest level in $\gamma$ such that $E \subseteq 2A_s B_v$, it follows that
\begin{equation*}
\diam(E) > A_s s^{-\ell(v) - 1}.
\end{equation*}
So $\diam(E) \simeq \diam(B_v)$ with constants depending only on $s$.  

Since $\diam(E)$ is comparable with both $\diam(B_v)$ and $\diam(B_{v'})$ with constants only depending on $s$, there is a constant $C > 0$ depending only on $s$ for which 
\[
|\ell(v) - \ell(v')| \leq C.
\]

Let $i = \min\{ \ell(v), \ell(v')\}$. Since $E \subset B_{v'}\cap 2A_sB_v$, we have, by the triangle inequality, that
\begin{align*}
d(c_{\widehat X}(v), c_{\widehat X}(v')) &\le \frac{1}{2}\left( \diam(2A_s B_v) + \diam(B_{v'})\right) \\
&\le 2A_s s^{-\ell(v)} + s^{-\ell(v')} \le (2A_s+1)s^{-i+C}
\end{align*}
Let $p\in \N_0$ be the smallest integer for which $(2A_s+1)s^C < s^p$. Then, by Corollary \ref{cor:struct lemma}, we have that
\[
|v-v'| \le 2p + |\ell(v)-\ell(v')| \le C'.
\]
where $C'$ depends only on $s$. Hence, by Lemma \ref{upper qi}, we have 
\begin{equation*}
|\varphi(v) - \varphi(v')| \leq 2(\alpha + \beta)C'.
\end{equation*}
Since $\diam(fE) \simeq \diam(B_{\varphi(v')})$ by Lemma \ref{ball diam lemma}, we have that 
\[
\diam(fE) \simeq \diam(B_{\varphi(v')}) \simeq \diam(B_{\varphi(v)}).
\] 
The proof is complete.
\end{proof}

We are now ready to prove Theorem \ref{thm:VQI2pBQS}.

\begin{proof}[Proof of Theorem \ref{thm:VQI2pBQS}]
Let $E$ and $E'$ be intersecting continua in $X$.  Let $x \in E \cap E'$ and let $\gamma \colon (\N_0,0) \to (\widehat X,\ast)$ be a vertical geodesic ray tending to $x$. Let also $v=\gamma(j)$ and $v'=\gamma(j')$ be vertices of highest level in $\gamma$ for which $E \subseteq 2 A_s B_v$ and $E'\subseteq 2A_s B_{v'}$, respectively. Then 
\begin{equation*}
\frac{\diam (E)}{\diam (E')} \simeq \frac{\diam(B_v)}{\diam(B_{v'})} \simeq s^{\ell(v') - \ell(v)}.
\end{equation*}
By Lemma \ref{geod ball diam lemma}, we also have that 
\begin{equation*}
\frac{\diam (fE)}{\diam (fE')} \simeq t^{\ell(\varphi(v')) - \ell(\varphi(v))}.
\end{equation*}
Let $c = \log(t) / \log(s)$.
%, so $s^{cx} = u^x$ when $x \in \R$.  
We apply the vertical quasi-isometry condition in three different cases.  

\smallskip
\noindent
\emph{Case 1:} Suppose first that $\ell(v) \leq \ell(v')$.  Then $\ell(v') - \ell(v) = |v-v'|$ and
\begin{align*}
\ell(\varphi(v')) - \ell(\varphi(v)) &\le |\ell(\varphi(v')) - \ell(\varphi(v))| \leq |\varphi(v') - \varphi(v)| \leq \alpha |v-v'| + \beta \\
&= \alpha(\ell(v')-\ell(v)) + \beta.  
\end{align*}
Hence
\begin{equation*}
\frac{\diam (fE)}{\diam (fE')} \simeq t^{\ell(\varphi(v')) - \ell(\varphi(v))} \leq t^{\alpha (\ell(v') - \ell(v)) + \beta} \simeq t^\beta \biggl(\frac{\diam(E)}{\diam(E')}\biggr)^{c\alpha}.
\end{equation*}

\smallskip
\noindent
\emph{Case 2:} Suppose now that $\ell(v) > \ell(v')$ and $\ell(\varphi(v)) > \ell(\varphi(v'))$.  Then, we have $\ell(v') - \ell(v) = -|v-v'|$.  By Lemma \ref{bdd hausdorff vqi rays}, there is a constant $H = H(\alpha, \beta, \delta)$, where $\delta$ is the Gromov hyperbolicity constant of $\widehat Y$, and a vertical geodesic $\sigma \colon (\N_0,0) \to (\widehat Y, \ast)$ for which $\Hdist(\sigma, \varphi(\gamma))\le H$. In particular, there are vertices $w$ and $w'$ on $\sigma$ for which $|w-\varphi(v)| \le H$ and $|w'-\varphi(v')|\le H$.

Since $\sigma$ is a vertical geodesic, we have, by the triangle inequality, that 
\begin{align*}
|\varphi(v) - \varphi(v')| &\leq |\varphi(v) - w| + |w-w'| + |w' - \varphi(v')| \leq 2H + |w-w'| \\
&=2H + |\ell(w)-\ell(w')| \le 2H + 2H + |\ell(\varphi(v))-\ell(\varphi(v'))|\\
&=4H +  \ell(\varphi(v))-\ell(\varphi(v')).
\end{align*}

On the other hand, we also have that
\begin{equation*}
\alpha^{-1} (\ell(v) - \ell(v')) - \beta = \alpha^{-1} |v-v'| - \beta \leq |\varphi(v) - \varphi(v')|
\end{equation*}
and so 
\begin{equation*}
\ell(\varphi(v')) - \ell(\varphi(v)) \leq \alpha^{-1} (\ell(v') - \ell(v)) + \beta + 4H.
\end{equation*}
It follows that
\begin{equation*}
\frac{\diam (fE)}{\diam (fE')} \simeq t^{\ell(\varphi(v')) - \ell(\varphi(v))} \leq t^{\alpha^{-1} (\ell(v') - \ell(v)) + \beta + 4H} \simeq t^{(\beta + 4H)} \biggl(\frac{\diam(E)}{\diam(E')}\biggr)^{c/\alpha}.
\end{equation*}

\smallskip
\noindent
\emph{Case 3:}  Suppose finally that $\ell(v) > \ell(v')$ and $\ell(\varphi(v)) \leq \ell(\varphi(v'))$.  We bound $\ell(v) - \ell(v')$.  Let $n = \ell(v')$ and let 
\begin{equation*}
k_0 = \sup \{k : \ell(\varphi(\gamma(n+k))) \leq \ell(\varphi(v'))\}. 
\end{equation*}
We have $k_0 \geq \ell(v) - \ell(v')$ as $v$ is such a vertex.  Let $v_0 = \gamma(n + k_0 + 1)$ and $\ell_0 = \ell(\varphi(v_0))$.  Then, as $\varphi\circ \gamma$ is a vertical $(\alpha,\beta)$-quasigeodesic, we have that
\begin{equation*}
\ell(\varphi(v')) <\ell_0 \leq \ell(\varphi(v')) + \alpha + \beta.
\end{equation*}
As in Case 2, we have by Lemma \ref{bdd hausdorff vqi rays} that there is a constant $H = H(\alpha, \beta, \delta)$ and a vertical geodesic $\sigma \colon (\N_0,0) \to (\widehat Y, \ast)$ for which $\Hdist(\sigma, \varphi(\gamma)) \leq H$.  In particular, 
\begin{equation*}
|\varphi(v_0) - \varphi(v')| \leq 4H + \alpha + \beta.
\end{equation*}
Let $\sigma_0$ be a geodesic from $\varphi(v_0)$ to $\varphi(v')$.  Then, the size of the $H$-neighborhood of $\sigma_0$, denoted $N_H(\sigma_0)$, is bounded above by a constant $C_0$ that depends only on $\alpha, \beta, H$, and the degree of $\widehat Y$.  Let $\gamma'$ be the segment on $\gamma$ from $v'$ to $v_0$.  The segment $\gamma'$ is part of a vertical geodesic, so by \cite[Theorem 5.6]{GH} we have $\varphi(\gamma') \subseteq N_H(\sigma_0)$.  Hence the length of $\gamma'$ is bounded above by $N(\varphi) C_0$.  This length is also $\ell(v_0) - \ell(v')$ so, as $\ell(v_0) \geq \ell(v)$, it follows that $\ell(v) - \ell(v') \leq N(\varphi) C_0$.

Now, if $\ell(v) - \ell(v') \leq C$, then 
\[
\frac{\diam(E)}{\diam(E')} \gtrsim s^{\ell(v') - \ell(v)} \gtrsim s^{-C}.
\]
Hence, we need only bound $\diam(fE) / \diam(fE')$ in this case.  For this, we proceed as in Case 1.  We have 
\begin{equation*}
|\ell(\varphi(v')) - \ell(\varphi(v))| \leq |\varphi(v') - \varphi(v)| \leq \alpha |v-v'| + \beta.  
\end{equation*}
It follows that $\ell(\varphi(v')) - \ell(\varphi(v)) \leq \alpha (\ell(v) - \ell(v')) + \beta$ and so
\begin{equation*}
\frac{\diam (fE)}{\diam (fE')} \simeq t^{(\ell(\varphi(v')) - \ell(\varphi(v)))} \leq t^{\alpha (\ell(v) - \ell(v')) + \beta} \simeq t^\beta \biggl(\frac{\diam(E')}{\diam(E)}\biggr)^{c\alpha}.
\end{equation*}

\smallskip
\noindent
\emph{End game:}
Since the map $f$ is a power branched quasisymmetry in each of the cases, we conclude that $f$ is a power quasisymmetry with constants depending only on the data.
\end{proof}

\subsection{Openness of traces}
\label{sec:traces_openness}

We have seen (Lemma \ref{degree lemma}) that finite multiplicity of a vertical quasi-isometry between vertex sets of hyperbolic fillings corresponds to discreteness of the trace map. We show that the following (discrete) path lifting property of the vertical quasi-isometry corresponds the openness of the trace map. We begin with an auxiliary definition.

Let $\widehat X \in \HF(X)$ be a hyperbolic filling of $X$. The \emph{$k$-cylinder $[\gamma]_k$ of a vertical geodesic $\gamma \colon (\N_0,0) \to (\widehat X,0) $ for $k\in \N_0$} is the set of geodesics
\[
[\gamma]_k = \{ \gamma' \in \Gamma(\widehat X, \ast) \colon \gamma'(j) = \gamma(j),\ j = 0,\ldots, k\}.
\]

\begin{defn}%[Vertical quasi-isometry lifting property (VQI LP)]
\label{VQI lifting prop}
Let $X$ and $Y$ be compact metric spaces, and $\widehat X\in \HF(X)$ and $\widehat Y \in \HF(Y)$. A vertical quasi-isometry $\varphi \colon \widehat X \to \widehat Y$ has the \emph{vertical quasi-isometry lifting property} (or \emph{VQI lifting property} for short), if there exists a constant $H'>0$ for which, for each $x \in X$ and centered vertical geodesics $\gamma \in \Gamma_x(\widehat X,\ast)$ and $\sigma \in \Gamma_{f(x)}(\widehat Y,\ast)$, there exist increasing sequences $(L_k)$ and $(M_k)$ in $\N_0$ having the property that, for each $k\in \N_0$ and $\sigma' \in [\sigma]_{M_k}$, there exists  $\gamma'\in [\gamma]_{L_k}$ satisfying $\Hdist(\varphi \circ \gamma', \sigma') \leq H'$, where $f \colon X\to Y$ is the trace of $\varphi \colon \widehat X \to \widehat Y$.  Here $\Gamma_x(\widehat X, \ast)$ denotes the set of vertical geodesics $\gamma$ in $\widehat X$ such that $\gamma \to x$ and $\Gamma_{f(x)}(\widehat Y, \ast)$ is defined similarly.  
\end{defn}

\begin{thm}\label{VQI LP iff open}
Let $X$ and $Y$ be compact metric spaces and let $\widehat{X}\in \HF_s(X)$ and $\widehat{Y}\in \HF_t(Y)$ be hyperbolic fillings of $X$ and $Y$, respectively.  Let $\varphi \colon \widehat{X} \to \widehat{Y}$ be a vertical quasi-isometry and let $f \colon X \to Y$ be the trace of $\varphi$. Then, $f$ is open if and only if $\varphi$ satisfies the VQI Lifting Property.
\end{thm}

\begin{proof}
Suppose first that the trace $f\colon X\to Y$ of $\varphi$ is an open map and let $x \in X$. Let $\gamma \colon (\N_0,0) \to (\widehat X, \ast)$ and $\sigma \colon (\N_0,0) \to (\widehat Y,\ast)$ be centered vertical geodesics tending to $x$ and $f(x)$, respectively. For each $k\in \N$, let $B_k = B(x,1/k) \subseteq X$.

We first specify sequences $(M_k)$ and $(L_k)$. Let $M_0 = 0$, $L_0 = 0$, and suppose that levels $M_0 < M_1< \cdots <M_{k-1}$ and $L_0 \leq L_1 \leq \cdots \leq L_{k-1}$ have been determined.  

We determine the level $M_k\in \N_0$. Since $f$ is open, there exists a ball $B_k' = B(f(x), r_k') \subseteq fB_k$.  Let now $M_k\in \N_0$ be an integer for which $2 A_t t^{-M_k} < r_k'$; we may assume that $M_k > M_{k-1}$. Note that, by Lemma \ref{cor:convergence}, each vertical geodesic $\sigma' \in [\sigma]_{M_k}$ tends to a point in $B'_k$.

Let $L_k \ge 0$ be the largest integer for which $s^{L_k} < k$. Note that, since $d(x,\gamma(L))\le s^{-L}$ for all $L\in \N_0$, we have that $B_k \subseteq B_{\gamma(L_k)}$. Moreover, for each $x'\in B_k$, there exists a vertical geodesic $\gamma' \in [\gamma]_{L_k}$ tending to $x'$.

We now show that, for these sequences $(L_k)$ and $(M_k)$, the map $\varphi \colon \widehat X \to \widehat Y$ satisfies the VQI lifting property.  Let $k\in \N_0$ and let $\sigma'\in [\sigma]_{M_k}$. Then $\sigma'$ tends to a point $y'\in B'_k$. Since $B'_k \subset fB_k$, we may fix $x'\in f^{-1}(y') \cap B_k$. Then, by our choice of $L_k$, there exists a vertical geodesic $\gamma'\in [\gamma]_{L_k}$ tending to $x'$. Since $f(x') = y'$, we have that $\varphi \circ \gamma' \colon (\N_0,0)\to (\widehat Y, \ast)$ is a vertical quasigeodesic tending $y'$. Then, by Lemma \ref{quasigeod close lemma}, we have that $\Hdist(|\varphi \circ \gamma'|, |\sigma'|) \leq H'$, where $H'$ depends only on $\alpha$ and $\beta$ and hyperbolic filling parameters of $\widehat X$ and $\widehat Y$.

We now show the VQI Lifting Property implies that $f$ is open.  It suffices to show that for every ball $B = B(x, r) \subseteq X$, the image $fB$ contains a ball $B' = B(f(x), r')$.  Let $B = B(x, r)$ be an open ball in $X$ and let $y = f(x)$.  Let $\gamma\colon (\N_0,0) \to (\widehat X, \ast)$ and $\sigma \colon (\N_0,0) \to (\widehat Y,\ast)$ be centered vertical geodesics tending to $x$ and $y$, respectively, and let $(L_k)$ and $(M_k)$ be the sequences guaranteed by the VQI lifting property. Since $L_k \to \infty$ as $k\to \infty$ and $\gamma$ is centered, we may fix, by Lemma \ref{basic filling lemma 2}, an index $k\in \N_0$ for which each vertical geodesic in $[\gamma]_{L_k}$ tends to a point in $B$.  

Since $\sigma$ is centered, we have that $y\in B_{\sigma(M_k)}$. Thus there exists $r'>0$ for which $B(y,r') \subseteq B_{\sigma(M_k)}$. Let $y' \in B(y, r')$. Since $B(y,r') \subseteq B_{\sigma(M_k)}$, there exists a vertical geodesic $\sigma' \in [\sigma]_{M_k}$ tending to $y'$. By the VQI lifting property, there exists now a vertical geodesic $\gamma' \in [\gamma]_{L_k}$ for which $\Hdist(\varphi \circ \gamma', \sigma') \le H'$. Then $\gamma'$ tends to a point $x'\in B$ by the choice of $k$. Since $\Hdist(\varphi\circ \gamma',\sigma') \le H'$, we have that $\varphi \circ \gamma'$ tends to $y'$. Thus $f(x') = y'$. We conclude that $B(y,r') \subseteq fB$. Hence, the map $f$ is open.
\end{proof}

\subsection{Surjectivity of traces}
\label{sec:traces_surjectivity}

Recall that a quasi-isometry $\psi \colon Z\to W$ between metric spaces $Z$ and $W$ is a cobounded map satisfying \eqref{QI} and that a map $\psi \colon Z\to W$ is cobounded if there exists a constant $C>0$ for which $W = B(\psi(Z),C)$. 

In the following theorem we show that the coboundedness of a vertical quasi-isometry also characterizes surjectivity of the trace map.

\begin{thm}\label{thm:bdry map surj}
Let $X$ and $Y$ be compact metric spaces and let $\widehat{X}\in \HF_s(X)$ and $\widehat{Y}\in HF_t(Y)$ be hyperbolic fillings for $X$ and $Y$, respectively. Let $\varphi \colon \widehat X \to \widehat Y$ be a vertical quasi-isometry and let $f \colon X \to Y$ be the trace of $\varphi$. Then $\varphi$ is cobounded if and only if $f$ is surjective.
\end{thm}

\begin{proof}
Suppose first that the trace map $f$ is surjective. Let $w \in \widehat{Y}$ and $\sigma \colon (\N_0,0) \to (\widehat{Y}, \ast)$ a vertical geodesic satisfying $\sigma(k)=w$.  Suppose that $\sigma \to y \in Y$.  We fix $x \in f^{-1}(y)$ and let $\gamma \colon (\N_0,0) \to (\widehat{X},\ast)$ be a vertical geodesic tending to $x$.  By Lemma \ref{quasigeod close lemma}, there exists a constant $H_1$ depending only on $\alpha$, $\beta$, $s$, and $t$ for which $\Hdist(\varphi \circ \gamma, \sigma) \leq H_1$.  Hence, there exists an index $j \in \N_0$ for which 
\[
|(\varphi \circ \gamma)(j) - \sigma(k)| \leq H_1.
\]
Thus $w\in B(\varphi(\widehat X), H_1)$. We conclude that $\varphi$ is cobounded.

\bigskip
Suppose now that $\varphi$ is cobounded. We show that $f$ is surjective. Let $y \in Y$ and $\sigma \colon (\N_0,0)\to \widehat Y$ a vertical geodesic tending to $y$. Since $\varphi$ is cobounded, we may fix $C>0$ for which $B(\varphi(\widehat X), C) = \widehat Y$. Then, for each $n\in \N_0$, there exists $v_n \in \widehat X$ for which $|\varphi(v_n) - \sigma(n)| < C$. Note that, since $\sigma$ is a vertical geodesic, $\ell(\varphi(v_n)) \ge n-C$ for each $n\in \N_0$. Let $w_n = \varphi(v_n)$ for each $n\in \N_0$.

For each $n\in \N_0$, we fix a vertical geodesic $\gamma_n \colon (\N_0,0) \to (\widehat X,\ast)$ for which $\gamma_n(\ell(v_n)) = v_n$. Since $\ell(\gamma_n(m)) = m$ for each $n$ and $m$, we may pass to a (diagonal) subsequence of $(\gamma_n)$ having the property that for each $n\in \N_0$ there exits an index $m_n\in \N$ for which $\gamma_m(n) = \gamma_{m_n}(n)$ for each $m\ge m_n$. Let now $\gamma \colon (\N_0,0) \to (\widehat X, \ast)$ be the vertical geodesic $n \mapsto \gamma_{m_n}(n)$.

Let $x$ be the limit of $\gamma$ in $X$. We show that $f(x) = y$. For each $k\in \N_0$, let $x_k$ be the limit of the vertical geodesic $\gamma_{m_k}$. Then $f(x_k)$ is the limit of $\varphi\circ \gamma_{m_k}$. Denote $y_k = f(x_k)$ for each $k\in \N_0$. Since
\[
d(y_k, y) \le d(y_k, c_Y(w_{m_k})) + d(c_Y(w_{m_k}), y)
\]
for each $k\in \N_0$. Since $c_Y(w_{m_k}) = c_Y(\varphi(v_{m_k})) = c_Y(\varphi(\gamma_{m_k}(\ell(v_{m_k}))))$, we have, by Lemma \ref{qg filling lemma}, that
\[
d(c_Y(w_{m_k}), y_{m_k}) \leq 3(A_t + 1) t^H t^{-\ell(w_{m_k})} \to 0
\]
as $\ell(w_{m_k}) \geq m_k - C \to \infty$ as $k \to \infty$. By Lemma  \ref{basic filling lemma 1}, 
\[
d(c_Y(w_{m_k}), c_Y(\sigma(m_k))) \leq A_t t^{C - m_k} \to 0
\]
as $k \to \infty$. 

Since $\sigma$ tends to $y$, we obtain by combining these estimates that $d(y, y_k) \to 0$ as $k\to \infty$. Thus $f(x_k) = y_k \to y$ as $k\to \infty$. Since $\gamma_{m_n}\in [\gamma]_n$ for each $n\in \N_0$, we have that $d(x_k,x)\lesssim s^{-k}$ for each $k\in \N_0$. Thus $x_k \to x$ and $f(x)=y$ by continuity of $f$ (Lemma \ref{lemma:trace_basic}). This completes the proof of the surjectivity of the trace map.
\end{proof}

\subsection{Application to local branched quasisymmetries}

As a consequence of both extension results, we obtain that local branched quasisymmetries self-improve to branched quasisymmetries. More precisely, we have the following.

\begin{thm}
\label{thm:local-self-improvement}
Let $f\colon X\to Y$ be a discrete and open local branched quasisymmetry between compact metric spaces $X$ and $Y$, where $X$ has bounded turning. Then $f$ is a branched quasisymmetry, quantitatively.
\end{thm}

\begin{proof}
Here our quantitative constants are allowed to depend on the locality scale of $f$.
By Theorem \ref{thm:local BQS to VQI}, the hyperbolic filling $\varphi \colon (\widehat X,\ast) \to (\widehat Y,\ast)$ is an eventual vertical quasi-isometry. By Remark \ref{rmk:eventual}, we may fix a vertical quasi-isometry $\psi \colon (\widehat X,\ast) \to (\widehat Y, \ast)$ having the same trace as $\varphi$. 

We claim that $N(\varphi) < \infty$.  This follows from local versions of Theorem \ref{thm:Koebe} and \ref{thm:bdd BQS to bdd VQI}.  Indeed, Theorem \ref{thm:Koebe} applies for all small enough balls $B$, so we may modify the proof of Theorem \ref{thm:bdd BQS to bdd VQI} by only considering non-overlapping balls with levels above the cutoff $J$ for $\varphi$.  As there are only finitely many vertices with level bounded above by $J$, we have $N(\varphi) < \infty$. Clearly also $N(\psi)<\infty$.

We apply Theorem \ref{thm:VQI2pBQS} to the vertical quasi-isometry $\psi$ to conclude that $f$, as the trace of $\psi$, is a branched quasisymmetry. 
\end{proof}

%%%%%%%%%%%%%%%%%%%%%%%%%%%%%%%%%%%%%%%%%%%%%%%%%%%%%%%%%%%%%%%%%%%%%%%%%%%%%
%%%%%%%%%%%%%%%%%%%%%%%%%%%%%%%%%%%%%%%%%%%%%%%%%%%%%%%%%%%%%%%%%%%%%%%%%%%%%
%%%%%%%%%%%%%%%%%%%%%%%%%%%%%%%%%%%%%%%%%%%%%%%%%%%%%%%%%%%%%%%%%%%%%%%%%%%%%

\part{Appendix}

\appendix

\section{Branched quasisymmetries and quasiregular maps}\label{QR and BQS sec}

In this appendix, we prove Theorem \ref{thm:QR=BQS} in the introduction. We consider the claim in two parts. We prove first a version of the Guo--Williams theorem for quasiregular mappings between Riemannian manifolds and then the reverse direction characterizing quasiregularity in terms of branched quasisymmetry.

Recall that a continuous map $f\colon M\to N$ between oriented Riemannian $n$-manifolds is \emph{$K$-quasiregular} if $f$ is in the Sobolev space $W^{1,n}_{\mathrm loc}(M,N)$ and satisfies the distortion inequality \eqref{QR}, that is,
\[
\norm{Df}^n \le K J_f
\]
almost everywhere in $M$. 

\subsection{Quasiregularity implies branched quasisymmetry}

As discussed, we begin with a Riemannian version of the Guo--Williams theorem.

\begin{thm}
\label{thm:Riemannian-GW}
Let $f\colon M \to N$ be a quasiregular mapping between oriented Riemannian $n$-manifolds. Then $f$ is a local branched quasisymmetry, quantitatively.
\end{thm}

In the statement, the quantitatively is understood as follows. Suppose that $f\colon M\to N$ is $K$-quasiregular for $K\ge 1$. Then there exists a radius $r>0$ depending on $f$, $M$, and $N$, such that for all $x \in M$, the restricted map $f|_{B(x,r)} \colon B(x,r) \to N$ is a branched $\eta$-quasisymmetry for $\eta\colon [0,\infty) \to [0,\infty)$ depending only on $K$

A few comments are in order.

\begin{remark}
Theorem \ref{thm:Riemannian-GW} bears some similarity to the correspondence between quasiconformal and quasisymmetric maps. Similar to quasisymmetric maps, branched quasisymmetries $X\to Y$, where $X$ has $\lambda$-bounded turning, are bounded maps. This is not true for quasiconformal nor quasiregular mappings. 

Note, however, by result of Heinonen and Koskela \cite{HK} that quasiconformal   maps between Loewner spaces are quasisymmetric. The analog for branched quasisymmetries is not true. Indeed, consider the exponential map $\exp \colon \C \to \C$, $z\mapsto e^z$. The complex plane $\C$ is a Loewner space and $\exp$ is holomorphic and hence $1$-quasiregular, but $\exp$ is not a branched quasisymmetry. Indeed, consider $k\in \N$ and segments $E=[0,k]$ and $E'=[0,ik]$. Then $\diam \exp(E) \simeq e^k$ and $\diam \exp(E') \sim 1$.
\end{remark}

\begin{remark}
Since the claim is a local statement, it suffices to prove it only for quasiregular mappings $f\colon B^n \to B^n$. Indeed, let $M$ and $N$ be oriented Riemannian $n$-manifolds and $f\colon M\to N$ a $K$-quasiregular map. Let also $\varepsilon>0$. Then there exists atlases $\{ (B_i, \varphi_i)\}_{i\in I}$ and $\{ (B'_j,\varphi'_j)\}_{j\in J}$ of $M$ and $N$, respectively, with following properties:
\begin{itemize}
\item each $B_i=B(x_i,r_i)$ and $B'_j=B(x'_j,r'_j)$ is a ball, 
\item each chart $\varphi_i \colon B_i \to B^n(0,r_i)$ and $\varphi'_j \colon B'_j \to B^n(0,r'_j)$ is diffeomorphism which is $(1+\varepsilon)$-bilipschitz, and
\item for each $i\in I$ there exists $j_i\in J$ for which $fB_i \subset B'_{j_i}$. 
\end{itemize}
Such atlases are given by continuity of $f$ and Riemannian exponential functions $TM \to M$ and $TN \to N$.

Now 
\[
\{\lambda'_{j_i} \circ \varphi'_{j_i} \circ f \circ \varphi_i^{-1} \circ \lambda_i \colon B^n \to B^n \colon i\in I\},
\]
where $\lambda_i \colon \R^n \to \R^n$ and $\lambda'_j \colon \R^n \to \R^n$ are scalings $x\mapsto r_i x$ and $y\mapsto y/r'_j$, respectively, is a family of $(1+\varepsilon)^{4n}K$-quasiregular mappings encoding the map $f$.
\end{remark}

\begin{remark}
Note that, for closed Riemannian manifolds, we obtain that $f$ is a branched quasisymmetry.
\end{remark}

We prove Theorem \ref{thm:Riemannian-GW} by showing that the hyperbolic filling $\varphi \colon \widehat M \to \widehat N$ is an eventual vertical quasi-isometry and use the trace theorem to obtain the result. We do not aim for optimal results and argue using compactness; see \cite{Pankka-Souto} for a similar discussion.

Our first lemma is a diameter distortion estimate for images of balls.

\begin{lemma}\label{QR Upper VQI Lemma}
Let $f\colon B^n \to B^n$ be a $K$-quasiregular map of finite multiplicity and $c > 1$.  Then there exists $C > 0$, with the following property.  If $B = B^n(b, r) \subseteq \tfrac{1}{2}B^n$ is a ball and $A = B(a, r') \subseteq B$ satisfies $\diam(B) \leq c \diam(A)$, then $\diam(fB) \leq C \diam(fA)$.  
\end{lemma}
\begin{proof}
Suppose that the claim fails for a $K$-quasiregular map $f\colon B^n\to \R^n$ of finite multiplicity. Then, for each $m \in \N$ there exist balls $B_m = B(b_m, r_m) \subseteq \tfrac{1}{2} B^n$ and $A_m = B(a_m, r_m')$ with $A_m \subseteq B_m$ and $\diam(B_m) \leq c \diam(A_m)$, but $\diam(fB_m) \geq m \diam(fA_m)$.  Let $\psi_m \colon \R^n \to \R^n$ be the conformal map $x\mapsto \frac{1}{r_m}(x-b_m)$.  Let $\vartheta_m$ be the conformal map $x \mapsto \lambda_m (x-f(b_m))$ where $\lambda_m>0$ is chosen such that $\sup_{z \in \partial B_m} \lambda_m|f(z) - f(b_m)| = 1$. Define $g_m = \vartheta_m \circ f \circ \psi_m^{-1} \colon B^n(0,2) \to \R^n$ for each $m\in \N$.  Here the domain of $g_m$ makes sense as $2 B_m \subseteq B^n$.

Each $g_m$ is $K$-quasiregular and of multiplicity $N(g_m) \le N(f) <\infty$.  Also, each $g_m$ has the properties that $g_m(0) = 0$ and $\sup_{|z| = 1} |g_m(z)| = 1$.  Hence, we may apply \cite[Corollary 2.5]{MSV} to obtain a subsequence also denoted $(g_m)$ and a non-constant $K$-quasiregular map $g \colon B(0, 2) \to \R^n$ such that $g_m \to g$ locally uniformly.  Now, each $\psi_m A_m$ is a ball in $\psi_m B_m = B^n$ and, by our diameter comparison assumption, there exists a constant $r_0 > 0$ depending only on $c$ such that $B(\psi_m(a_m), r_0) \subseteq B^n$.   By compactness, there is a subsequence of $(\psi_m(a_m))$, which we also denote $(\psi_m(a_m))$, which converges to a point $a \in \overline{B^n(0, 1 - r_0)}$.  Consider $B' =  B(a, r_0 / 2)$.  From the convergence of $\psi_m(a_m)$, we see that $B' \subseteq \psi_m A_m$ for large $m$.    By our normalization with $\lambda_m$ we have $\diam(g_m B_m) \leq 2$, so by assumption $\diam(g_m A_m) \to 0$.  Thus, $\diam(g_m B') \to 0$, so $g B'$ is a single point.  This is a contradiction as $g$ is open as it is a non-constant quasiregular map.
\end{proof}

We now obtain a (soft) Koebe distortion theorem for quasiregular maps $B^n \to B^n$.

\begin{lemma}\label{QR finite mult}
Let $f \colon B^n \to B^n$ be a $K$-quasiregular map of finite multiplicity.  Then, there exists a constant $c > 0$ with the following property.  Whenever $B = B(b, r) \subset \tfrac{1}{2}B^n$ is a ball, there exists a ball $B(f(b), c \diam(fB)) \subseteq fB$.   
\end{lemma}

\begin{proof}
Suppose towards a contradiction that the claim does not hold.  Thus, for each $m \in N$, there exists a ball $B_m = B(b_m, r_m) \subset \tfrac{1}{2}B^n$ and a point $w_m \notin fB_m$ with $d(f(b_m), w_m) \leq \diam(fB)/m$.  As in Lemma \ref{QR Upper VQI Lemma}, we may renormalize the maps $f|_{2B_m} \colon 2B_m \to \R^n$ to $K$-quasiregular maps $g_m \colon B^n(0, 2) \to \R^n$ with $g_m(0) = 0$, $N(g_m) \leq N(f)$, and $\sup_{|z| = 1} |g_m(z)| = 1$.  In this case, there also exists a sequence $(w_m')$ in $\R^n$ such that, for each $m\in \N$, we have $|w_m'| \leq 2 / m$ and $w_m' \notin g_m B^n$.  As in Lemma \ref{QR Upper VQI Lemma}, passing to a subsequence (also indexed by $m$) there is a local uniform limit $g \colon B^n(0, 2) \to \R^n$ of $(g_m)$ which is non-constant, quasiregular, and has bounded multiplicity.  

For each $m\in \N_0$, let $\gamma_m'=[0,w'_m]$ be the line segment in $\R^n$ connecting $0$ to $w_m'$, and let $\gamma_m$ be a maximal lift of $\gamma_m'$ under $g_m$ containing the origin.  Since $w'_m\not\in g_mB^n$, we have that $\gamma_m$ leaves $B^n$.  To see this, note that both maximal lifts and local lifts of quasiregular maps always exist by \cite[II.3]{Ri}. 
  
For $\theta \in (0, 1)$, let $S_\theta = \partial B^n(0, \theta) \subseteq B^n$.  Fix $\theta \in (0, 1)$ and let $x_m \in S_\theta \cap \gamma_m$.  Passing to a subsequence, we may assume that $x_m \to x \in S_\theta$ as $m\to \infty$.  We claim that $g(x) = 0$.  Since $g$ is continuous, we have that $g(x_m) \to g(x)$ as $m\to \infty$.  Let $\epsilon > 0$.  For $m\in \N$ large enough, we have that $|g_m(x_m)| < \epsilon$ as $x_m \in \gamma_m$.  Since $g_m \to g$ locally uniformly as $m\to \infty$, we have for large enough $m\in \N$ that $|g_m(x_m) - g(x_m)| < \epsilon$.  It follows that $|g(x)| \leq 2\epsilon$ and, as this is true for all $\epsilon > 0$, we have $g(x) = 0$.  

We have shown that for every $\theta \in (0, 1)$, there is an $x \in S_\theta$ with $g(x) = 0$.  This contradicts the finite multiplicity of $g$.
\end{proof}

The following lemma will be used to obtain the lower bound in vertical quasi-isometry condition for the hyperbolic filling $\widehat X \to \widehat Y$.

\begin{lemma}\label{QR Lower VQI Lemma}
Let $f \colon B^n \to B^n$ be a $K$-quasiregular mapping.  Then, there are constants $\alpha\ge 1$ and $\beta \ge 0$, depending only on $n$ and $K$, and $0<R<1$ depending on $f$, with the following property.  Whenever $B = B^n(b, r) \subseteq \tfrac{1}{2}B^n$ is a ball with $0< r < R$ and $A = B(a, r_A)$ is a ball with $A \subseteq \frac{1}{2}B$, then 
\begin{equation*}
\log \biggl(\frac{\diam(B)}{ \diam(A)}\biggr) \leq \alpha \log \biggl(\frac{\diam(fB)}{ \diam(fA)}\biggr) + \beta.
\end{equation*}
\end{lemma}

\begin{proof}
We first chose a constant $c$ for later in the proof.  Let $\phi$ be the $n$-Loewner function for $\R^n$ (cf. \cite[Chapter 8]{He}).  By \cite[Theorem 3.6]{HK}, there are constants $t_0 > 0$ and $C > 0$ such that for all $t \geq t_0$ we have $ C \log(t)^{1-n} \leq \phi(t)$.  

Since $\phi$ is decreasing, positive, and has $\lim_{t \to \infty} \phi(t) = 0$ we may, for an ambient constant $C_1$, choose $c'$ large enough so that 
\begin{equation*}
 \phi (t_0) =  K C_1 \log \biggl(\frac{c'}{8}\biggr)^{1 - n}
\end{equation*}  
holds.  It follows that, if $s\ge c'$ and $t>0$ are such that
\[
\phi(t) \le K C_1 \log \left( \frac{s}{8}\right)^{1-n},
\]
then we have that $t\ge t_0$ and so 
%It then follows for $t \geq t_0$ that 
\[
C \log(t)^{1-n} \leq K C_1 \log \biggl(\frac{s}{8}\biggr)^{1 - n}.
\]
Set $c = \max\{c', 10\}$.

We also determine $R$.  For this, note that $G = \overline{\tfrac{1}{2}B^n}$ is compact and $f$ is open, so $\varepsilon' := \dist(\partial f(B^n), f(G)) > 0$.  Moreover, $f|_G$ is uniformly continuous, and so there is a radius $R$ such that whenever $r < R$ and $B = B(b,r) \subseteq G$, we have $B(f(b), 10\diam(fB)) \subseteq fB^n$.

Let $B = B(b,r) \subseteq \tfrac{1}{2}B^n$ with $0<r<R$ and $A\subseteq \frac{1}{2}B$ be balls. Because of the additive term $\beta$ on the right hand side of our desired inequality, we may assume that $\diam(B) \geq c \diam(A)$.  Let 
\begin{equation*}
F = \{w \in \R^n : |f(a)-w| = 4 \diam(fB)\}.
\end{equation*}
Note that $F \subseteq f(B^n)$ by our choice of $R$.  Let also $E = f \overline{A}$.  Then, $E$ and $F$ are continua in $\R^n$.  We see that
%As $F$ is contained in a hemisphere, we have 
\begin{equation*}
\Delta(E, F) := \frac{\dist(E, F)}{\min (\diam(E), \diam(F))} \leq \frac{4 \diam fB}{\diam fA}.
\end{equation*}
It follows that 
\begin{equation*}
\nmod(E, F) \geq \phi\biggl(\frac{4 \diam fB}{\diam fA}\biggr).
\end{equation*}
For each path $\gamma$ connecting $E$ to $F$, we lift to a path $\tilde{\gamma}$ starting from a point in $\overline{A}$ (cf. \cite[Corollary II.3.4]{Ri}).  This lift $\tilde{\gamma}$ must leave $B$ as $F \cap fB = \emptyset$.  Let $\Gamma$ be the collection of these lifts.  Then, as $f$ is quasiregular, we have, by \cite[Theorem II.8.1]{Ri}, that
\begin{equation*}
\nmod(E, F) \leq \nmod(f\Gamma) \leq  K\nmod(\Gamma). 
\end{equation*}
We note that $B(a, \diam(B) / 8)) \subseteq B$. By the fact that $c \geq 10$ and \cite[Lemma 7.8]{He}, there is a constant $C_1 > 0$ such that
\begin{equation*}
\nmod(\Gamma) \leq C_1 \log\biggl(\frac{\diam(B)}{8 \diam(A)}\biggr)^{1 - n}.
\end{equation*}
Thus, 
\begin{equation}\label{ineq diam}
\phi \biggl(\frac{4 \diam fB}{\diam fA}\biggr) \leq \nmod(E, F) \leq K C_1 \log \biggl(\frac{\diam(B)}{8 \diam(A)}\biggr)^{1 - n}.
\end{equation}
From the fact that $c \geq c'$ at the beginning of the proof, we must have that 
\[
\frac{4 \diam fB}{\diam fA} \geq t_0
\]
and so
\begin{equation*}
 C \log\biggl(\frac{4 \diam(fB) }{\diam(fA)}\biggr)^{1-n} \leq K C_1 \log\biggl(\frac{\diam(B)}{8 \diam(A)}\biggr)^{1 - n}.
\end{equation*}
Thus, there is a constant $C_2 > 0$ such that
\begin{equation*}
\log\biggl(\frac{\diam(B)}{8 \diam(A)}\biggr) \leq C_2 \log\biggl(\frac{4 \diam(fB)}{ \diam(fA)}\biggr),
\end{equation*}
which proves the result.
\end{proof}

We are now ready to prove that quasiregular maps induce eventual vertical quasi-isometries between hyperbolic fillings.  

\begin{lemma}\label{QR induce EVQI}
Let $f \colon M \to N$ be a $K$-quasiregular map between closed and oriented Riemannian $n$-manifolds.  Let $\widehat{M}\in \HF_s(M)$ and $\widehat{N} \in \HF_t(N)$ be hyperbolic fillings of $M$ and $N$, respectively, and let $\varphi \colon \widehat M \to \widehat N$ be a hyperbolic filing of $f$. Then, $\varphi$ is an eventual $(\alpha,\beta)$-vertical quasi-isometry with finite multiplicity, where $\alpha\ge 1$ and $\beta\ge 0$ depend only on $n$ and $K$.
\end{lemma}

\begin{proof}
It suffices to show that there exist constants $\alpha, \beta, J > 0$ with the property that, for every vertical geodesic $\gamma \colon (\N_0,0) \to (\widehat M,\ast)$ and all $j, j' \geq J$, we have that
\begin{equation*}
\alpha^{-1}|j-j'| - \beta \leq |\varphi(\gamma(j)) - \varphi(\gamma(j'))| \leq \alpha |j-j'| + \beta,
\end{equation*}
and that $\varphi$ has finite multiplicity. We fix first a constant $J_0>0$ with the following properties.  The diameters of balls $B_v$ are below the radius $r>0$, which admits the localization of $f$ using $2$-bilipschitz charts, and we are be able to, for each ball $B_v$ of this small enough radius, find charts where $B^n$ is in the image of the chart and the image of $B_v$ in the chart is contained in $\tfrac{1}{2} B^n$.  We also require that we may apply Lemma \ref{QR Lower VQI Lemma} to the balls $B_v$, which further restricts the size of $\diam(B_v)$ due to the factor $R$ in that lemma.  We note that as $M$ and $N$ are compact, there are only finitely many pairs of charts, and so the fact that the radius $R$ from Lemma \ref{QR Lower VQI Lemma} may change from chart pair to chart pair can be mitigated by taking a large enough $J_0$.

We first prove the upper inequality.  We note that there is a constant $c_1 > 0$ such that, if $v, v'$ are vertices in $\widehat M$ with $|v - v'| = 1$, then there exists a third vertex $v''\in \widehat M$ satisfying $B_v, B_{v'} \subseteq B_{v''}$ and both estimates $|v - v''| \leq c_1$ and $|v' - v''| \leq c_1$ hold.  This allows us to apply Lemma \ref{QR Upper VQI Lemma} and conclude that, if the balls $B_v$, $B_{v'}$, and $B_{v''}$ are small enough, then the diameters of $fB_{v}, fB_{v'}$, and $fB_{v''}$ are all comparable.  As $v''$ is a bounded number of levels away from $v$ and $v'$, small enough in this case just means that there is a $J_1\ge J_0$ such that we may apply Lemma \ref{QR Upper VQI Lemma} to $B_{\gamma(j)}$ and $B_{\gamma(j')}$ when $j, j' \geq J_1$.  As $B_v \cap B_{v'} \neq \emptyset$ when $|v - v'| = 1$, it follows from Lemma \ref{vertex comp} applied to $B_v$ and $B_{v'}$ that $|\varphi(v) - \varphi(v')|$ is bounded independently of $v, v'$. 
The upper vertical quasi-isometry inequality follows now from the triangle inequality.

We now prove the lower inequality.  Let $j > j' \geq J_2$ with $J_2\ge J_1$ to be determined and let $v = \gamma(j)$ and $v' = \gamma(j')$ be vertices on the same vertical geodesic $\gamma \colon (\N_0,0)\to (\widehat M,\ast)$.  Then, by Lemma \ref{basic filling lemma 2}, $B_v \subseteq (A_s + 2) B_{v'}$. By applying Lemma \ref{vertex comp} to the ball $2(A_s + 2) B_{v'}$, we see that there is a constant $c_2 > 0$ independent of $v$, $v'$, and a vertex $v''$ with $B_v \subseteq \frac{1}{2} B_{v''}$ and $|v' - v''| \leq c_2$.  If $J_2$ is large enough, we apply Lemma \ref{QR Lower VQI Lemma} to $B_v$ and $B_{v''}$ to conclude that there are constants $C, D > 0$ independent of $v$ and $v'$ such that
\begin{equation*}
\log \biggl(\frac{\diam(B_{v''})}{ \diam(B_v)}\biggr) \leq C \log \biggl(\frac{\diam(fB_{v''})}{ \diam(fB_v)}\biggr) + D.
\end{equation*}
We see that there is a constant $A > 0$ such that 
\begin{equation*}
\biggl|\log_s \biggl(\frac{\diam(B_{v''})}{ \diam(B_v)}\biggr) - (-\ell(v'') + \ell(v))\biggr| \leq A
\end{equation*}
and
\begin{equation*}
\biggl|\log_t \biggl(\frac{\diam(fB_{v''})}{ \diam(fB_v)}\biggr) - (-\ell(\varphi(v'')) + \ell(\varphi(v)))\biggr| \leq A
\end{equation*}
from which we conclude that there are constants $C', D' > 0$ such that
\begin{equation*}
|\ell(v) - \ell(v'')| \leq C' |\ell(\varphi(v))) - \ell(\varphi(v''))| + D'.
\end{equation*}
As $B_v \cap B_{v''} \neq \emptyset$, this means that
\begin{equation*}
|v - v''| \leq C' |\varphi(v) - \varphi(v'')| + D'.
\end{equation*}
We note that by construction, $|v - v'| \leq |v - v''|$.  If $J_2$ is large enough, we may apply the upper vertical quasi-isometry inequality proven above and conclude that
\begin{equation*}
\begin{split}
|v - v'| &\leq C' |\varphi(v) - \varphi(v')| + C'|\varphi(v') - \varphi(v'')| + D'. \\
& \leq C' |\varphi(v) - \varphi(v')| + D' + C'(\alpha |v' - v''| + \beta).
\end{split}
\end{equation*}
As $|v' - v''|$ is uniformly bounded in our construction, the lower eventual vertical quasi-isometry inequality holds.

It remains to show $\varphi$ has finite multiplicity.  This follows as in the proof of Theorem \ref{thm:local-self-improvement} from the local versions of Theorems \ref{thm:Koebe} and \ref{thm:bdd BQS to bdd VQI}.
\end{proof}

We are now ready to prove Theorem \ref{thm:Riemannian-GW}.

\begin{proof}[Proof of Theorem \ref{thm:Riemannian-GW}]
By Lemma \ref{QR induce EVQI}, $f \colon M\to N$ induces an eventual vertical quasi-isometry $\varphi \colon (\widehat M,\ast) \to (\widehat N,\ast)$ between hyperbolic fillings $\widehat M$ and $\widehat N$ of $M$ and $N$, respectively.  This promotes to a vertical quasi-isometry $\psi \colon (\widehat M,\ast) \to (\widehat N,\ast)$ of finite multiplicity with the same trace $f$ as $\varphi$. Thus, by Theorem \ref{thm:intro-VQI2pBQS}, we see $f$ is a branched quasisymmetry, quantitatively.
\end{proof}

\subsection{Branched quasisymmetry is quasiregular}

To show that branched quasisymmetries are quasiregular, we use a metric characterization of quasiregular mappings based on inverse linear dilatation. The interested reader may want to compare the argument to the proof that quasisymmetries are quasiconformal.

For the definitions, let $f\colon M \to N$ be a discrete and open map between $n$-manifolds. 

A domain $U \subset M$ is a \emph{normal neighborhood of $x\in M$} if $U$ is relatively compact, $f(\partial U) = \partial fU$, and $f^{-1}(f(x))\cap \overline{U} = \{x\}$. For $x\in M$ and $r>0$, we denote also $U(x,f,r) \subset M$ the $x$-component of $f^{-1}B(f(x),r)$. For each $x\in M$ there exists $r_x>0$ for which $U(x,f,r)$ is a normal neighborhood of $x$ whenever $r < r_x$; the Euclidean argument in \cite[Lemma 2.9]{MRV1} works also in the Riemannian case.

For $x\in M$, the \emph{inverse linear dilatation of $f$ at $x$} is
\[
H^*(x,f) = \limsup_{r\to 0} \frac{L^*(x,f,r)}{\ell^*(x,f,r)}
\]
where 
\[
L^*(x,f,r) = \sup_{x'\in \partial U(x,f,r)} d(x,x');
\]
and
\[
\ell^*(x,f,r) = \inf_{x'\in \partial U(x,f,r)} d(x,x')
\]
for $r>0$. Note that $0<\ell^*(x,f,r) \le L^*(x,f,r) < \infty$ for $0<r<r_x$.

The characterization of quasiregular mappings in terms of inverse linear dilatation is proven by Martio, Rickman, and V\"ais\"al\"a in \cite{MRV1}. 

\begin{thm}[{\cite[Theorem 4.14]{MRV1}}]
Let $G\subset \R^n$ be a domain. A non-constant mapping $f\colon G\to \R^n$ is quasiregular if and only if it satisfies the following conditions:
\begin{enumerate}
\item $f$ is sense-preserving, discrete, and open. \label{item:MRV1}
\item $H^*(x,f)$ is locally bounded in $G$.
\item There exists $a^*<\infty$ such that $H^*(x,f) \le a^*$ for almost every $x\in G\setminus B_f$.\label{item:MRV3}
\end{enumerate}
\end{thm}
Here $B_f \subset G$ is the branch set $f$, that is, the set of points $x\in G$ at which $f$ is not a local homeomorphism. 

The proof of this theorem in \cite{MRV1} reveals that, if $f$ satisfies conditions \eqref{item:MRV1}--\eqref{item:MRV3}, then $f$ is $(a^*)^{n-1}$-quasiregular.

Our statement is an almost immediate corollary of this theorem.

\begin{thm}
\label{thm:not-GW}
A sense-preserving, discrete, and open local branched quasisymmetry $f\colon M\to N$ between oriented Riemannian $n$-manifolds is quasiregular, quantitatively.
\end{thm}
\begin{proof}
By localizing the map with $2$-bilipschitz charts, it suffices to show the claim for branched $\eta$-quasisymmetries $f\colon B^n \to B^n$. 

Let $x\in B^n$ and let $r_x>0$ be the radius for which domains $U(x,f,r)$ are normal neighborhoods of $x$ for $r<r_x$. Let now $r < r_x$ be such that $B(f(x), r) \subseteq f(B^n)$ and fix points $y$ and $z$ on $\partial U(x,f,r)$ for which $|y-x| = \ell^*(x,f,r)$ and $|z-x|=L^*(x,f,r)$; note that infimum and supremum are actually minimum and maximum due to compactness of $\partial U(x,f,r)$.

Since $fU(x,f,r) = B^n(f(x),r)$, there exists a continuum $E \subset \overline{U(x,f,r)}$ connecting $z$ to $x$ for which $fE = [f(z),f(x)]$; for this either continuum lifting lemma (Lemma \ref{cont lift}) or path lifting under discrete and open maps (see e.g.~\cite[Lemma 2.7]{MRV1}) can be used. Note that, when we lift $[f(z),f(x)]$ from $z$ in $U(x,f,r)$ the lift is total and hence contains $x$, since $f^{-1}(f(x)) = \{x\}$.

Since $f[y,x]$ connects $S^{n-1}(f(x),r)$ to $f(x)$ and $E$ connects $z$ to $x$, we have that $\diam f[y,x] \ge r$ and $\diam E \ge |z-x|$. Thus, by the $\eta$-branched quasisymmetry condition, 
\begin{align*}
r \le \diam(f[y,x]) &\le \eta\left( \frac{\diam [y,x]}{\diam E}\right) \diam fE \\
&\le \eta\left( \frac{|y-x|}{|z-x|}\right) \diam [f(z),f(x)] = \eta\left( \frac{|y-x|}{|z-x|}\right) r.
\end{align*}
Hence,
\[
\frac{\ell^*(x,f,r)}{L^*(x,f,r)} = \frac{|y-x|}{|z-x|} \ge \eta^{-1}(1),
\]
where $\eta^{-1}$ is the inverse of $\eta$. Thus,
\[
H^*(x,f) = \limsup_{r\to 0} \frac{L^*(x,f,r)}{\ell^*(x,f,r)} \le \frac{1}{\eta^{-1}(1)}.
\]
This concludes the proof.
\end{proof}

Theorem \ref{thm:QR=BQS} follows now from Theorems \ref{thm:Riemannian-GW} and \ref{thm:not-GW}.

%\bibliographystyle{alpha}
%\bibliography{QRfilling}

\end{document}